\documentclass[12pt]{article}
\usepackage{amsfonts,amsmath,amsthm,amssymb}
\usepackage{graphicx,tabularx}
\usepackage{array}
\usepackage[usenames,dvipsnames]{color}
\usepackage{comment}
\usepackage{fullpage,hyperref}
\usepackage{bbm}
\usepackage[dvipsnames]{xcolor}
\usepackage{tikz}
\usetikzlibrary{decorations.markings}
\usepackage{listings}
\usepackage{enumerate}
\usepackage[normalem]{ulem}
\usetikzlibrary{calc}
\usepackage[capitalise]{cleveref}

\def\cI{\mathbb{\mathcal I}}

\def\PB{\color{purple}}

\def\IC{\color{red}}
\def\AK{\color{orange}}
\DeclareMathOperator{\dist}{dist}
\DeclareMathOperator{\mad}{mad}
\newcommand{\D}{{\mathcal D}}
\newcommand{\mb}{\mathbf}

\newtheorem{theorem}{Theorem}[section]

\newtheorem{observation}[theorem]{Observation}

\newtheorem{question}[theorem]{Question}
\newtheorem{lemma}[theorem]{Lemma}
\newtheorem{claim}{Claim}[theorem]
\newtheorem{corollary}[theorem]{Corollary}
\theoremstyle{definition}

\usetikzlibrary{calc}

\def\epsilon{\varepsilon}
\def\phi{\varphi}

\title{\vspace{-0.5in} Flexible DP 3-coloring of sparse multigraphs}

\author{
{{Peter Bradshaw}}\thanks{
\footnotesize{University of Illinois at Urbana--Champaign, Urbana, IL 61801, USA.
 E-mail: \texttt {pb38@illinois.edu}.
 Research %%% of this author
is supported in part by NSF RTG grant DMS-1937241, as well as an AMS-Simons Travel Grant.
}}
\and
{{Ilkyoo Choi}}\thanks{Department of Mathematics, Hankuk University of Foreign Studies, Yongin-si, Gyeonggi-do, Republic of Korea.
 E-mail: \texttt {ilkyoo@hufs.ac.kr}
 and  Discrete Mathematics Group, Institute for Basic Science (IBS), Daejeon, Republic of Korea. 
 Research %%% of this author
is supported by the Hankuk University of Foreign Studies Research Fund, the Institute for Basic Science (IBS-R029-C1), and the National Research Foundation of Korea(NRF) grant funded by the Korea government(MSIT) (RS-2025-23324220).
}
\and
{{Alexandr Kostochka}}\thanks{
\footnotesize {University of Illinois at Urbana--Champaign, Urbana, IL 61801, USA.
 E-mail: \texttt {kostochk@illinois.edu}.
 Research %%% of this author
is supported in part by  NSF  Grant DMS-2153507 and by NSF RTG Grant DMS-1937241.
}}
}

\begin{document}

\maketitle

\begin{abstract}
    A \emph{request} on a graph assigns a preferred color to a subset of the vertices.
    A graph $G$ is \emph{$\epsilon$-flexibly $k$-choosable} if for every $k$-list assignment $L$ and every request $r$ on $G$, there is an $L$-coloring such that an $\epsilon$-fraction of the requests are satisfied. 
    This notion was introduced in 2019 by Dvo\v{r}\'ak, Norin, and Postle, who also proved important properties of flexible colorings and posed several natural problems.
    { However, the weighted version of this problem is a special case of the much older problem of fractional hypergraph matchings, introduced by Lov\'asz in 1975.}
 %   Bradshaw and Bi proved that every graph with maximum average degree less than $3$ is $2^{-32}$-flexibly $3$-choosable; the threshold on the maximum average degree is tight.
 %   The same authors also proved (in a different paper) that every graph with maximum average degree less than $\frac{11}{3}$ is $2^{-145}$-flexibly $4$-choosable. 
 %   {\PB I think this result is kind of weak, maybe there is something better to put in the abstract.
%   Maybe we could mention that the Zhu paper showed that certain classes of mad $<3$ graphs attain $\epsilon = \frac 13$, including planar graphs of girth at least $6$.
%    }
 %   {\IC sure - these results seemed the most similar result to ours - but we can certainly change what's written here}

    We  study  flexibly DP-colorable multigraphs. 
    We prove that every loopless multigraph with maximum average degree less than $3$ is $\frac{1}{5}$-flexibly DP $3$-colorable, except for an infinite family of multigraphs that we completely characterize. 
    {The constant $\epsilon = \frac 15$ is best possible in the weighted setting,
    as shown by an infinite family of tight examples.}
    Our result follows from a stronger statement in terms of potential. 
    We also provide a family of graphs that gives a 
    negative answer to a question by Dvo\v{r}\'ak, Norin, and Postle regarding flexibility for list coloring in the setting of DP-coloring.

\end{abstract}

\section{Introduction}
\subsection{Background}
%%%%%%%%%%%%%%%%%%%%%%%%%%%%%%%%%%%%
%%%%%%%%%% proper coloring
%%%%%%%%%%%%%%%%%%%%%%%%%%%%%%%%%%%%

 We write {\em graphs} for finite simple graphs, and we say that {\em multigraphs} may have multiple edges, but not loops.
Given a multigraph $G$, let $V(G)$ and $E(G)$ denote the vertex set and edge set, respectively, of $G$. 
A multigraph $G$ has a \emph{proper $k$-coloring} if there is a function $\varphi:V(G)\rightarrow\{1, \ldots, k\}$ such that $\varphi(x)\neq\varphi(y)$ for every edge $xy$. 
In the theory of graph coloring,
a statement involving pre-colored vertices is often easier to prove inductively than a simpler statement;
for example, 
in order to prove that every triangle-free planar graph is $3$-colorable, Gr\"otzsch~\cite{1958Grotzsch} proved that every proper $3$-coloring of a $4$-cycle or a $5$-cycle in a planar graph extends to a proper $3$-coloring.
In these settings,
the goal is  usually to determine whether a pre-coloring of some vertex set extends to a coloring of the entire multigraph.
It is easy to come up with examples where a partial (proper) $k$-coloring of a properly $k$-colorable graph does not extend to the entire graph. 
For instance, a pre-coloring of two vertices at arbitrary distance in a planar graph does not necessarily extend to a proper 4‐coloring~\cite{1978Fisk}.
A natural relaxation of this binary problem is to seek a proper $k$-coloring where a large portion of a set of pre-coloring requests 
is satisfied.
In this relaxed setting,
since we no longer need to satisfy all pre-coloring requests, it is natural to lift the restriction that the pre-coloring requests form a partial proper  coloring.
The complete graph on $k$ vertices shows that
for a $k$-chromatic graph with a set of pre-coloring requests,
no more than a $\frac{1}{k}$-fraction of the requests can be satisfied in general; for example, 
the requests 
may ask for 
all vertices to be colored with the same color.
However, one can prove that a $\frac{1}{k}$-fraction of the requests can always be satisfied by permuting the colors of a proper $k$-coloring.
The problem of finding a coloring that satisfies a large fraction of coloring requests was investigated in the list coloring setting by Dvo\v{r}\'ak, Norin, and Postle~\cite{2019DvNoPo}.
%in 2018.
%{\IC should this be 2019?}
%{\PB Actually I looked at arXiv and the first draft was 2016, and they made an update in 2018. Which year is best?}
%{\IC Usually for published results I just use the year of the paper, but we don't have to.} {\AK We do not need to tell any year, just [10].}

%{\PB If we want to aim for a good journal, would it make the introduction stronger to say that this problem really goes back to Lov\'asz's general fractional matching and fractional cover hypergraph framework from 1975?
%We can describe the problem in this setting by letting $V(H)$ be hyperedges and letting $(H,L)$-colorings be vertices.
%Then, we could argue that the problem has more history and that it is of a similar flavor to parameters with consequences outside of graph coloring, such as fractional clique number, which was used by Lov\'asz to prove famous results about Shannon capacity.
%}
%{\IC sounds good to me. }

% In planar graphs, it is known that any precoloring of a set of vertices at distance at least three from one another extends when at least 5 colors are used \cite{1998Albertson} and for sufficiently distant vertices this holds also in the list coloring setting \cite{2017DvLiMoPo}; 

%%%%%%%%%%%%%%%%%%%%%%%%%%%%%%
%%%%%%%%%% list coloring
%%%%%%%%%%%%%%%%%%%%%%%%%%%%%%

List coloring was introduced by Vizing~\cite{1976Vizing} and independently by Erd\H{o}s, Rubin, and Taylor~\cite{1980ErRuTa}.
A \emph{list assignment} $L$ for a multigraph $G$ is a function on $V(G)$ that assigns a list $L(v)$ of available colors to each vertex $v$. 
An \emph{$L$-coloring}  of $G$  is a proper coloring $\varphi$ such that $\varphi(v)\in L(v)$ for every vertex $v$.
A multigraph is \emph{$k$-choosable} if it has an $L$-coloring for every list assignment $L$ such that $|L(v)|\geq k$ for every vertex $v$. 
Note that the proper $k$-coloring
problem
arises exactly from the case
that all lists have the same set of $k$ available colors. 
Dvo\v{r}\'ak, Norin, and Postle~\cite{2019DvNoPo} formalized the idea of 
a list coloring that satisfies a
positive fraction of pre-coloring requests:
 a \emph{request} on a graph $G$ with a list assignment $L$ is a function $r$ on
some nonempty 
$U\subseteq V(G)$ that requests a color $r(v)\in L(v)$ for each vertex $v\in U$. 
A graph $G$ with a list assignment $L$ is \emph{$\epsilon$-flexible} if for every request $r$ on $G$, there is an $L$-coloring $\varphi$
that satisfies
an $\epsilon$-fraction 
of the requests;
namely, $\varphi(v)=r(v)$ for at least $\epsilon |U|$ vertices. 
% For $\varepsilon>0$, a request $r$ is \emph{$\varepsilon$-satisfiable} if there is an $L$-coloring $\varphi$ such that $\varphi(v)=r(v)$ for at least $\varepsilon|\dom(r)|$ vertices $v\in \dom(r)$. 
% A graph $G$ is \emph{$\varepsilon$-flexible} if $G$ is $\varepsilon$-satisfiable for every list assignment $L$ and every request $r$.
As used in~\cite{2022BrMaSt}, we say a graph is \emph{$\varepsilon$-flexibly $k$-choosable} if it is $\epsilon$-flexible for every $k$-list assignment.
This concept of flexibility attracted the attention of many researchers and has led to a number of interesting questions and results. 

%%%%%%%%%%%%%%%%%%%%%%%%%%%%%%%%%%%%%%%%
%%%%%%%%%% degenerate graphs
%%%%%%%%%%%%%%%%%%%%%%%%%%%%%%%%%%%%%%%%

A graph is \emph{$d$-degenerate} if there is a linear ordering of the vertices such that every vertex has at most $d$ earlier neighbors. 
A simple greedy algorithm shows that every $d$-degenerate graph is $(d+1)$-choosable.
An important and fascinating question posed in~\cite{2019DvNoPo} asks whether $d$-degenerate graphs are $\varepsilon$-flexibly $(d+1)$-choosable. 

\begin{question}[Dvo\v{r}\'ak, Norin, Postle~\cite{2019DvNoPo}]\label{ques:degen}
    For every integer $d\geq 0$, is there an $\epsilon>0$ such that every $d$-degenerate graph is $\epsilon$-flexibly $(d+1)$-choosable? 
\end{question} % actually weighted version

The above question is easily proved when $d\leq 1$
and still open for each $d\geq 2$. 
 In fact, the weaker statement that a single pre-coloring assignment on a $2$-degenerate graph $G$ with a $3$-list assignment $L$ is satisfied by some $L$-coloring of $G$ has a highly nontrivial proof. 
 %(Feel free to delete this if you think it is not necessary.)}
The originators of \Cref{ques:degen} proved
that list size $d+2$ is sufficient, which is $1$ more than the questioned $d+1$, with $\epsilon=\frac{1}{2d^{2d}}$.
(Kaul, Mathew, Mudrock, Pelsmajer~\cite{2024KaMaMuPe}
recently improved this value to $\epsilon = 2^{-(d+1)}$.)
On the other hand, they answered \Cref{ques:degen} in the affirmative for a subclass of $d$-degenerate graphs, namely, for graphs with maximum average degree less than $d+\frac{2}{d+3}$.
Recall that the \emph{maximum average degree} of a graph $G$, denoted $\mad(G)$, is the maximum of $\frac{2|E(H)|}{|V(H)|}$ for every subgraph $H$ of $G$. 

\begin{theorem}[Dvo\v{r}\'ak, Norin, Postle~\cite{2019DvNoPo}]
    For every integer $d\geq 0$, there exists $\epsilon>0$ such that every $d$-degenerate graph is $\epsilon$-flexibly $(d+2)$-choosable. 
    % ($\epsilon=\frac{1}{2d^{2d}}$)
\end{theorem} % actually weighted version

\begin{theorem}[Dvo\v{r}\'ak, Norin, Postle~\cite{2019DvNoPo}]
    For every integer $d\geq 1$, there exists $\epsilon>0$ such that every graph $G$ with $\mad(G)<d+\frac{2}{d+3}$ is $\epsilon$-flexibly $(d+1)$-choosable. 
\end{theorem} % actualyl weighted

%A closely related question is the  question below, which is more of a meta-question that has been investigated by numerous researchers
%for specific graph classes $\mathcal G$:
%{\IC some paper quote this as DNP, but was it explicitly stated there?}
%{\PB No, but they have many questions of this form, and this seems to be the general motivation of their paper}

%\b<egin{question}
%[Dvo\v{r}\'ak, Norin, Postle~\cite{2019DvNoPo}]
%    If $\mathcal G$ is a specific  graph class where every graph in $\mathcal G$ is $k$-choosable, then is there an $\epsilon>0$ such that every graph in $\mathcal G$ is $\varepsilon$-flexibly $k$-choosable?
%\end{question} % actually weighted version

%{\AK I think, it is a vague question and does not relate much to ours. We may safely delete it.}
%{\IC I don't mind deleting it.}
%{\PB Neither do I}

% DNP's results planar graphs are flexible with list size 6 and planar graphs with girth 5 are flexible with list size 4 - they asked if the these list sizes can be reduced to the choosability bound. DNP also asked if planar graphs with girth 4 are flexible with list size 4 and if planar graphs with girth 6 (they proved girth 12) are flexible with list size 3. 

%%%%%%%%%%%%%%%%%%%%
%%%%%%%%%% 2 degen
%%%%%%%%%%%%%%%%%%%%
We first discuss some known
results on $2$-degenerate graphs, which are $3$-choosable.
Answering a question in~\cite{2019DvNoPo}, Dvo\v{r}\'ak, Masa\v{r}\'ik, Mus\'ilek, and Pangr\'ac~\cite{2020DvMaMuPa} proved that every planar graph with girth at least 6 is $\varepsilon$-flexibly $3$-choosable for some small universal $\epsilon > 0$. 
(Cambie, Cames van Batenburg, Zhu~\cite{arXiv_CaCaZh2023} recently proved
that the statement is still true when
$\varepsilon = \frac13$.)
The result was extended to a larger graph class by Bi and Bradshaw~\cite{arXiv_2025BiBr}, who proved that  every graph $G$ with $\mad(G)<3$ is $2^{-30}$-flexibly $3$-choosable; the threshold on the maximum average degree is tight since the complete graph on four vertices is not even $3$-choosable.  
Bradshaw, Masa\v{r}\'ik, and Stacho~\cite{2022BrMaSt} also 
proved that graphs with treewidth  at most 2 are
$\frac{1}{3}$-flexibly $3$-choosable, answering a question of Choi et al.~\cite{2022ChClFeHoMaMa}.
% Treedepth $k$ are $\frac{1}{k^2}$-flexibly $k$-choosable. 

%%%%%%%%%%%%%%%%%%%%
%%%%%%%%%% 3degen
%%%%%%%%%%%%%%%%%%%%
Regarding $3$-degenerate graphs, which are $4$-choosable, Bi and Bradshaw~\cite{arXiv_2024BiBr} proved that every graph $G$ with $\mad(G)<11/3$ is $2^{-145}$-flexibly $4$-choosable.
Dvo\v{r}\'ak, Masa\v{r}\'ik, Mus\'ilek, Pangr\'ac~\cite{2021DvMaMuPa} proved that every planar graph with girth at least $4$ is $2^{-126}$-flexibly $4$-choosable, and 
Yang and Yang~\cite{arXiv_YaYa2020} proved that every planar graph with neither $4$-cycles nor $5$-cycles is $\varepsilon$-flexibly $4$-choosable for some universal $\epsilon > 0$.

% 2023 Cambie, Cames van Batenburg, Zhu~\cite{arXiv_CaCaZh2023}: planar with girth 4 is weighted $\frac{1}{5}$-flexibly 5-choosable, planar with girth 5 weighted $\frac{1}{4}$-flexibly 4-choosable, planar with girth 6 is weighted $\frac{1}{3}$-flexibly 3-choosable. 

We now mention results for a graph $G$ with bounded maximum degree $D$, which is $(D+1)$-choosable.
Vizing~\cite{1976Vizing} proved that for $D\geq 3$, if $G$ does  not contain $K_{D+1}$, then $G$ is $D$-choosable.
%{\PB Didn't Vizing show this originally? Should we cite him? (Also, Sasha, if you could check the attached file called ``Vizing 1976.pdf" and see if this theorem is in there, that would be extremely helpful, because I can't read it.)}. 
Cambie, Cames van Batenburg, Davies, and Kang~\cite{2024CaCaDaKa} proved that $G$ is 
%weighted
$\frac{1}{D+1}$-flexibly $(D+1)$-choosable.
Bradshaw, Masa\v{r}\'ik, and Stacho~\cite{2022BrMaSt} proved that if $D\geq3$ and $G$ has no $K_{D+1}$-subgraph,
then $G$ is
%weighted $\frac{1}{2D^4}$-flexible $D$-choosable and 
$\frac{1}{6D}$-flexibly $D$-choosable; 
the value $\epsilon = \frac{1}{6D}$ is best possible up to a constant factor. 
See also~\cite{2022LiMaMuZe,2019Masarik,2022Yang} for relevant variations and~\cite{2022ChClFeHoMaMa} whose introduction surveys results for planar graphs.

%%%%%%%%%%%%%%%%%%%%%%%%%%%%%%
%%%%%%%%%% DP-coloring
%%%%%%%%%%%%%%%%%%%%%%%%%%%%%%
In this paper,
we investigate the notion of flexibility for \emph{DP-coloring}, introduced by Dvo\v{r}\'ak and Postle~\cite{2018DvPo}, where they use the term \emph{correspondence coloring}. 
Given a multigraph $G$ and a function $h:V(G) \rightarrow \mathbb N$, an \emph{$h$-cover} $(H, L)$ of $G$ is a graph $H$ constructed as follows. 
\begin{itemize}
    \item For each vertex $v \in V(G)$, $H$ has a clique $L(v)$ of size $h(v)$. 
    \item { For each pair $\{x,y\}$ of vertices of $G$, if $G$ has exactly $s$ edges joining $x$ to $y$, then $H$ is the union of at most $s$
     matchings between $L(x)$ and $L(y)$.} 
\end{itemize}
% If each set $L(v) \subseteq V(H)$ is size $k$, then we say that $H$ is a \emph{$k$-cover}. 
We often refer to the elements of $V(H)$ as \emph{colors}.
% If each matching joining two parts $L(u), L(v) \subseteq V(H)$ is a perfect matching of size $k$, then we say that $H$ is a \emph{topological $k$-cover}. 
% Note that if $H$ is a topological $k$-cover of $G$, then $H$ is a $k$-sheeted cover of $G$ when $G$ and $H$ are considered as CW-complexes with the standard topology.
An \emph{$(H, L)$-coloring} of a multigraph $G$ with a cover $(H, L)$ is a function $\varphi:V(G)\rightarrow V(H)$ such that $\varphi(v)\in L(v)$ for every vertex $v$ of $G$ and $\varphi(x)\varphi(y)\not\in E(H)$ for every edge $xy$ of $G$. 
A multigraph $G$ is \emph{DP $k$-colorable} if there is an $(H, L)$-coloring for every $k$-cover $(H, L)$ of $G$. 
We remark that for proper coloring and list coloring, multiplicity of edges is irrelevant, as the additional edges do not add extra restrictions.
However, for DP-coloring, the multiplicity of edges is important, as edge multiplicity determines the
number of matchings added to the cover 
between the available colors of the corresponding vertices.
For example,  the $2$-vertex multigraph $K_2^3$, which has two vertices joined by three edges, is properly $2$-colorable and $2$-choosable, but it may not even  DP 3-colorable,  depending on the cover.

% A graph $G$ has a \emph{DP $k$-coloring} if for every $k$-cover $(H, L)$ of $G$, $(H, L)$ has an independent set containing exactly one vertex in each set $L(v)$, for $v \in V(G)$. 
% When a cover $H$ of $G$ is fixed, a \emph{$H$-DP-coloring} of $G$ is a mapping $\phi:V(G) \rightarrow V(H)$ that maps each vertex $v \in V(G)$ to an element $\phi(v) \in L(v)$, with the property that $\{\phi(v): v \in V(G) \}$ is an independent set in $H$. 
% We often call $\varphi$ a \emph{DP-coloring} when $H$ is clear from context.
We say a multigraph $G$ is \emph{$\varepsilon$-flexibly DP $k$-colorable} if for every $k$-cover $(H, L)$ and every request $r$ on $G$, there is an $(H, L)$-coloring $\varphi$
that satisfies
an $\epsilon$-fraction of the requests given by $r$.

\subsection{Weighted flexibility and fractional packings}

Given a graph $G$ with a cover $(H, L)$, 
a \emph{weighted request} on $G$ is a mapping $w:V(H) \rightarrow \mathbb R_{\geq 0}$. Given a value $\epsilon > 0$, we say that 
$(G,H,L)$ is \emph{weighted $\epsilon$-flexible}
if 
for every weighted request 
%$G$ is
%\emph{weighted $\epsilon$-flexibly DP $k$-colorable} if for every $k$-cover $(H, L)$ and weighted request 
$w:V(H) \rightarrow \mathbb R_{\geq 0}$, $G$ has an $(H,L)$-coloring $\phi:V(G) \rightarrow V(H)$
satisfying
\begin{equation*}
\label{eqn:weighted-flexibility}
\sum_{v \in V(G)} w(\phi(v)) \geq \epsilon \sum_{v \in V(G)} \sum_{c \in L(v)} w(c) .
\end{equation*}
We say that $G$ is 
\emph{weighted $\epsilon$-flexibly DP $k$-colorable} 
if $(G,H,L)$ is weighted $\epsilon$-flexible for every $k$-cover $(H,L)$ of $G$.

% In this paper, we aim to show that there exists a universal constant $\epsilon > 0$ such that every $K_4$-free graph of maximum average degree less than $3.01$ is weighted $\epsilon$-flexibly DP $3$-colorable. 
%Similarly to Dvo\v{r}\'ak and Postle~\cite{2018DvPo}, we use probability distributions on DP-colorings in order to show weighted
%flexibility.
The following lemma, appearing in \cite{2019DvNoPo},
gives a sufficient condition
for weighted flexibility in terms of distributions on $(H,L)$-colorings. As pointed out in \cite{2019DvNoPo},
linear programming duality can be used to show that the sufficient condition is also necessary.
For linear programming duality, 
one can express
the weighted flexibility problem 
as a fractional hypergraph 
matching problem
involving a hypergraph $\mathcal H$ for which $E(\mathcal H) = V(H)$, $V(\mathcal H)$ is the set of $(H,L)$-colorings of $G$, and incidence in $\mathcal H$ is given by containment.
Then, the linear programming dual is the fractional cover problem on $\mathcal H$, which can be phrased in terms of distributions on $(H,L)$-colorings of $G$.

\begin{lemma}[\cite{2019DvNoPo}]
\label{lem:prob}
Let $G$ be a multigraph. Suppose that for every $k$-cover $(H, L)$ of $G$, there exists a probability distribution on DP-colorings $\phi:V(G) \rightarrow V(H)$ such that for each vertex $v \in V(G)$ and color $c \in L(v)$, $\Pr(\phi(v) = c) \geq \epsilon$. Then  $G$ is weighted $\epsilon$-flexibly DP $k$-colorable.
\end{lemma}
\begin{proof}
    Let $(H,L)$ be a $k$-cover of $G$, 
    and let $w:V(H) \rightarrow \mathbb R_{\geq 0}$ be a weighted request. Consider a DP-coloring $\phi:V(G) \rightarrow V(H)$ chosen randomly from our distribution. Then  
    \[\sum_{v \in V(G)} w(\phi(v)) = \sum_{v \in V(G)} \sum_{c \in L(v)} w(c) \mathbbm 1_{\phi(v) = c}.\]
    Therefore, 
     \[E  \left ( \sum_{v \in V(G)} w(\phi(v)) \right ) = \sum_{v \in V(G)} \sum_{c \in L(v)} w(c) \Pr(\phi(v) = c) \geq \epsilon \sum_{v \in V(G)} \sum_{c \in L(v)} w(c).\]
     Hence, at least one $(H,L)$-coloring $\phi$ satisfies 
      \[  \sum_{v \in V(G)} w(\phi(v))   \geq \epsilon \sum_{v \in V(G)} \sum_{c \in L(v)} w(c).\]
\end{proof}

% \subsection{3-covers}
% Now, we focus on the class $\mathcal G$ of graphs of $K_4$-free graphs of maximum average degree less than $3.01$.
% ({\RB cannot do better than $\frac{22}{7}=3.1428...$ since there is a $K_4$-free graph that is not properly 3-colorable.})
% {\PB What is the graph? I think we may still be able to do better asymptotically, since Kostochka and Yancey show that $4$-critical graphs % on $n$ vertices have at least $\frac{5n-2}{3}$ edges.}
% {\PB Sasha and I observed there is a family of nonflexible graphs satisfying $8n - 5e = -3$}
Given a multigraph $G$ and a $k$-cover $(H, L)$ of $G$, we say that a distribution on $(H,L)$-colorings of $G$ is an \emph{$\epsilon$-distribution} if for every vertex $v \in V(G)$ and color $c \in L(v)$, 
\[ \Pr(\phi(v) = c)\geq\epsilon .\]
By \Cref{lem:prob}, if every $k$-cover $(H, L)$ of $G$ admits an $\epsilon$-distribution on $(H,L)$-colorings,
%{\PB if and only if }
then $G$ is weighted $\epsilon$-flexibly DP $k$-colorable.
%Dvo\v{r}\'ak, Norin, and Postle~\cite{2019DvNoPo}
%mention that the converse holds by linear programming duality.

\bigskip

Given a graph $G$ with a $k$-cover $(H,L)$, 
%Cambie, Cames van Batenburg, Davies, and  Kang~\cite{2024CaCaDaKa}
%define a \emph{correspondence $L$-packing of $H$}
%{\PB I saw them define $\mathcal H  = (H,L)$ and talk about an $\mathcal H$-packing. Is it okay if we just say 
an \emph{$(H,L)$-packing} is a set of $k$ pairwise disjoint $(H,L)$-colorings of $G$ whose union gives $V(H)$.
If $G$ has an $(H,L)$-packing  for every $k$-cover $(H,L)$,
then the distribution given by choosing one of the $k$ disjoint $(H,L)$-colorings of $G$ uniformly at random shows that $G$ is weighted $\frac 1k$-flexibly DP $k$-colorable.
Cambie, Cames van Batenburg, and Zhu~\cite{arXiv_CaCaZh2023}
showed that the fractional version of the correspondence packing problem involving a $k$-cover can be equivalently defined in terms of flexible DP-coloring.
Thus, they define a \emph{fractional packing} of a graph $G$ with a $k$-cover $(H,L)$ as a distribution on $(H,L)$-colorings $\phi$ of $G$ such that for each $v \in V(G)$ and $c \in L(v)$, $\Pr(\phi(v) = c) = \frac 1k$.
For convenience, we call such a distribution a \emph{fractional  $(H,L)$-packing}.
In this way, a graph $G$ has a fractional $(H,L)$-packing for every $k$-cover
$(H,L)$ if and only if $G$ is $\frac 1k$-flexibly DP $k$-colorable.

For many graph classes $\mathcal G$, when $k$ is sufficiently large,
every graph in $\mathcal G$ is $\frac 1k$-flexibly DP $k$-colorable.
Indeed,
Cambie, Cames van Batenburg, and Zhu~\cite{arXiv_CaCaZh2023}
showed
when $\mathcal P_g$ is the class of planar graphs with girth at least $g$,
graphs in $\mathcal P_3$ are $\frac 18$-flexibly DP $8$-colorable,
graphs in $\mathcal P_4$ are $\frac 15$-flexibly DP $5$-colorable,
graphs in $\mathcal P_5$ are $\frac 14$-flexibly DP $4$-colorable, 
and graphs in $\mathcal P_6$ are $\frac 13$-flexibly DP $3$-colorable.
Cambie and Cames van Batenburg~\cite{CaCa2024}
also gave a layering condition
that implies the existence of fractional packings for many graph classes; in particular, they show that 
graphs with pathwidth $k$ are $\frac 1k$-flexibly $k$-choosable, and graphs with treedepth $k-1$ are $\frac 1k$-flexibly $k$-choosable.

\subsection{Our results}

%%%%%%%%%%%%%%%%%%%%%%%%%%%%%%
%%%%%%%%%% main theorem 
%%%%%%%%%%%%%%%%%%%%%%%%%%%%%%
For our  main contribution, we prove that a multigraph $G$ with $\mad(G)<3$ is $\frac{1}{5}$-flexibly DP $3$-colorable, except for graphs in a certain family
$\cI$.
This gives a large class of $2$-degenerate graphs that are flexibly DP $3$-colorable for the constant $\epsilon = \frac 15$.

%{\PB I am not sure how I feel about the placement of this theorem.
%It is our main result, but it is kind of hard to find in the introduction.
%Should it get its own subsection?
%}
%{\IC maybe \Cref{thm:main-mad} should come before the above counterexample to DNP's question. }

\begin{theorem}
\label{thm:main-mad}
    Every multigraph with maximum average degree less than $3$ is $\frac 15$-flexibly DP $3$-colorable, except for graphs with a subgraph in $\cI$. 
\end{theorem}

In the setting of a multigraph $G$ with maximum average degree less than $3$,
we %unfortunately 
 cannot aim in general to find a 
fractional $(H,L)$-packing
for an arbitrary $3$-cover $(H,L)$ of $G$.
Indeed, Cambie, Cames van Batenburg, and Zhu~\cite{arXiv_CaCaZh2023} showed that when $G$ is 
the (simple) graph
obtained from $K_{2,3}$ by adding an edge between two vertices in the part of size $3$,
there is a $3$-cover $(H,L)$ of $G$ that admits no fractional packing.
Nevertheless, we will see that the notion of fractional packings plays a key role in the proof of 
\Cref{thm:main-mad}.

%%%%%%%%%%%%%%%%%%%%
%%%%%%%%%% example
%%%%%%%%%%%%%%%%%%%%

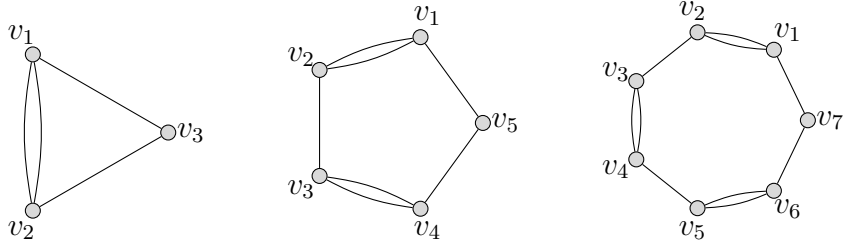
\begin{figure}
\begin{center}
\begin{tikzpicture}[scale=1.2, every node/.style={circle, draw, fill=gray!30, inner sep=2pt}]
  % Define positions of the 5 vertices on a circle
\foreach \i in {1,...,3}
    \node[draw = white, fill = white] (z\i) at ({360/3*\i}:1.25) {$v_{\i}$};
  
  \foreach \i in {1,...,3}
    \node (v\i) at ({360/3*\i}:1) {};

\node[draw = white,fill = white] (z) at (2,0) {};

  % Draw single edges
  \draw (v2) -- (v3);
  \draw (v3) -- (v1);

  % Draw double edges using slight bending
  \draw[bend left=10] (v1) to (v2);
  \draw[bend right=10] (v1) to (v2);
\end{tikzpicture}
\begin{tikzpicture}[scale=1.2, every node/.style={circle, draw, fill=gray!30, inner sep=2pt}]
  % Define positions of the 5 vertices on a circle
\foreach \i in {1,...,5}
    \node[draw = white, fill = white] (z\i) at ({360/5*\i}:1.25) {$v_{\i}$};
  
  \foreach \i in {1,...,5}
    \node (v\i) at ({360/5*\i}:1) {};
    
\node[draw = white,fill = white] (z) at (2,0) {};
  % Draw single edges
  \draw (v2) -- (v3);
  \draw (v4) -- (v5);
  \draw (v5) -- (v1);

  % Draw double edges using slight bending
  \draw[bend left=10] (v1) to (v2);
  \draw[bend right=10] (v1) to (v2);

  \draw[bend left=10] (v3) to (v4);
  \draw[bend right=10] (v3) to (v4);
\end{tikzpicture}
\begin{tikzpicture}[scale=1.2, every node/.style={circle, draw, fill=gray!30, inner sep=2pt}]
  % Define positions of the 7 vertices on a circle

\foreach \i in {1,...,7}
    \node[draw = white, fill = white] (z\i) at ({360/7*\i}:1.25) {$v_{\i}$};
  
  \foreach \i in {1,...,7}
    \node (v\i) at ({360/7*\i}:1) {};

  % Draw single edges
  \draw (v3) -- (v2);
  \draw (v5) -- (v4);
  \draw (v7) -- (v6);
  \draw (v7) -- (v1);

  % Draw double edges using slight bending
  \draw[bend left=10] (v2) to (v1);
  \draw[bend right=10] (v2) to (v1);

  \draw[bend left=10] (v4) to (v3);
  \draw[bend right=10] (v4) to (v3);

  \draw[bend left=10] (v6) to (v5);
  \draw[bend right=10] (v6) to (v5);

\end{tikzpicture}
\end{center}
\caption{The multigraphs $I_1$, $I_2$, and $I_3$ are shown from left to right. Each multigraph $I_{m+1}$ is obtained from $I_m$ by performing two subdivisions on any edge of multiplicity $1$ and then adding a second edge joining the two new vertices. Each of them has maximum average degree less than $3$ and is not flexibly DP $3$-colorable.}
\label{fig:Ilkyoo}
\end{figure}

%%%%%%%%%%%%%%%%%%%%%%%%%%%%
%%%%%%%% construction
%%%%%%%%%%%%%%%%%%%%%%%%%%%%
The graph class $\cI$ is defined as the following:
We first define the $2$-degenerate 
multigraph $I_m$ to be the multigraph obtained from a $(2m+1)$-cycle $v_1v_2\ldots v_{2m+1}v_1$ by duplicating the edge $v_{2i-1}v_{2i}$
for each $i\in\{1, \ldots, m\}$.
The graphs $I_1$, $I_2$, and $I_3$ are shown in \Cref{fig:Ilkyoo}.
We write $\cI=\{I_m: m\geq 1\}$.
We will see that for some 
DP $3$-cover $(H,L)$ of $I_m$,
no $(H,L)$-coloring assigns the color $3_{v_{2m+1}}$ to $v_{2m+1}$; therefore, a single request $r(v_{2m+1}) = 3_{v_{2m+1}}$ is never satisfied.
This shows that $I_m$ is not flexible for DP $3$-coloring.

\begin{figure}
\begin{center}
\begin{tikzpicture}[scale=1.2, every node/.style={circle, draw, fill=gray!30, inner sep=2pt}]
\def\n{3}   

\foreach \i in {3,...,3}
    \node[draw = white, fill = white] (z\i) at ({360/\n*\i}:1.3) {$v_{\i}$};

  \foreach \i in {1,...,\n}
    \node (v\i) at ({360/\n*\i}:1) {};

\foreach \i in {1,...,2}
{
    \node[draw = white, fill = white] (z\i) at ({360/\n*\i}:2.2) {$w_{\i}$};
    \node[draw = white, fill = white] (z\i) at ({360/\n*\i}:0.6) {$v_{\i}$};
    \node (w\i) at ({360/\n*\i}:1.85) {};
    \draw[bend left=10] (v\i) to (w\i);
    \draw[bend right=10] (v\i) to (w\i);
}

\draw (v1) -- (v2) -- (v3) -- (v1);

\end{tikzpicture}
\indent
\begin{tikzpicture}[scale=1.2, every node/.style={circle, draw, fill=gray!30, inner sep=2pt}]
  % Define positions of the 7 vertices on a circle

\def\n{5}   
\node[draw = white, fill = white] (z) at ({360/\n*1}:0.7) {$v_1$};
\node[draw = white, fill = white] (z) at ({360 - 360 / \n}:0.7) {$v_4$};
\node[draw = white, fill = white] (z) at ({360}:1.25) {$v_5$};

\foreach \i in {2,...,3}
    \node[draw = white, fill = white] (z\i) at ({360/\n*\i}:1.25) {$v_{\i}$};
  
  \foreach \i in {1,...,\n}
    \node (v\i) at ({360/\n*\i}:1) {};

\foreach \i in {1,4}
{
    \node[draw = white, fill = white] (z) at ({360/\n*\i}:2) {$w_{\i}$};
    \node (w\i) at ({360/\n*\i}:1.75) {};
    \draw[bend left=10] (v\i) to (w\i);
    \draw[bend right=10] (v\i) to (w\i);
}

  % Draw single edges
  \draw (v1) -- (v2);
  \draw (v3) -- (v4) -- (v5);
  \draw (v5) -- (v1);

  % Draw double edges using slight bending
  \draw[bend left=10] (v2) to (v3);
  \draw[bend right=10] (v2) to (v3);

\end{tikzpicture}
\indent \indent
\begin{tikzpicture}[scale=1.2, every node/.style={circle, draw, fill=gray!30, inner sep=2pt}]
  % Define positions of the 7 vertices on a circle

\node[draw = white, fill = white] (z) at ({360/7*1}:0.7) {$v_1$};
\node[draw = white, fill = white] (z) at ({360/7*6}:0.7) {$v_6$};
\node[draw = white, fill = white] (z) at ({360/7*7}:1.25) {$v_7$};

\foreach \i in {2,...,5}
    \node[draw = white, fill = white] (z\i) at ({360/7*\i}:1.25) {$v_{\i}$};
  
  \foreach \i in {1,...,7}
    \node (v\i) at ({360/7*\i}:1) {};

\foreach \i in {1,6}
{
    \node[draw = white, fill = white] (z) at ({360/7*\i}:2.4) {$w_{\i}$};
    \node (w\i) at ({360/7*\i}:2) {};
    \draw[bend left=10] (v\i) to (w\i);
    \draw[bend right=10] (v\i) to (w\i);
}

  % Draw single edges
  \draw (v1) -- (v2);
  \draw (v3) -- (v4);
  \draw (v5) -- (v6);
  \draw (v7) -- (v1);
  \draw (v7) -- (v6);

  % Draw double edges using slight bending
  \draw[bend left=10] (v2) to (v3);
  \draw[bend right=10] (v2) to (v3);

  \draw[bend left=10] (v4) to (v5);
  \draw[bend right=10] (v4) to (v5);

\end{tikzpicture}
\end{center}
\caption{The graphs $J_1$, $J_2$, and $J_3$ are shown from left to right. Each graph $J_{m+1}$ is obtained from $J_m$ by performing two subdivisions on an edge of multiplicity $1$ not incident with $v_{2m+1}$ and then adding a new edge joining the two new vertices.}
\label{fig:1/5}
\end{figure}
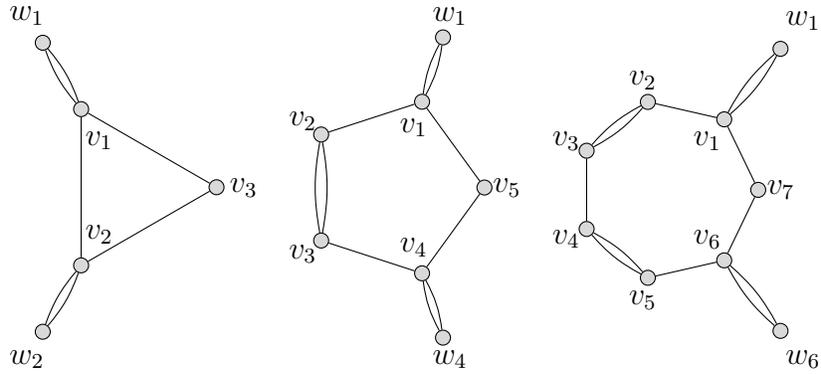

We also provide an infinite family of examples showing that the constant $\epsilon = \frac 15$ cannot be improved in the weighted setting, defined below.
For each integer $m \geq 1$,
let $J_m$ be obtained from the cycle $v_1\ldots v_{2m+1}v_1$ as follows. 
First, add two vertices $w_1$ and $w_{2m}$, and join each $w_j$ to $v_j$ by two parallel edges.
Then, for each even  $i\in\{2, \ldots, 2m-2\}$, add a second edge joining $v_i$ and $v_{i+1}$
(see \Cref{fig:1/5}).
We write $\mathcal J = \{J_m: m \geq 1\}$.

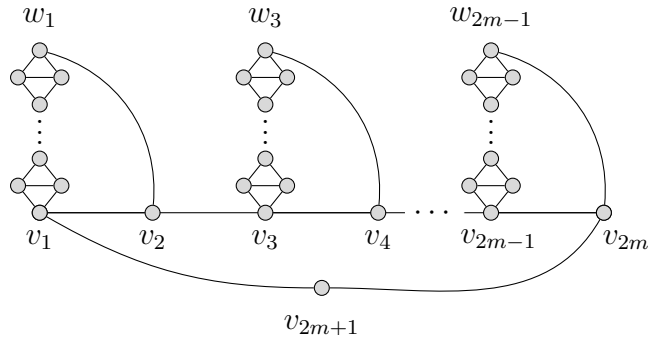
\begin{figure}
\begin{center}
\begin{tikzpicture}[scale=1, every node/.style={circle, draw, fill=gray!30, inner sep=2pt}]
\clip (-2,-1.62) rectangle + (11,4.6);

\def\s{3} %Spacing between diamonds
% Offsets
\def\dy{0.36}                 % vertical step per diamond
\def\dx{0.577*0.5}               % half width (to form equilateral triangle)

\node[draw = white, fill = white] (z) at (\s*2+0.1, -\dy) {$v_{2m-1}$};

\node[draw = white, fill = white] (z) at (\s*2.6 , -\dy) {$v_{2m}$};

\node[draw = white, fill = white] (z) at (0 , -\dy) {$v_{1}$};

\node[draw = white, fill = white] (z) at (\s*0.5 , -\dy) {$v_{2}$};

\node[draw = white, fill = white] (z) at (\s*1.5 , -\dy) {$v_{4}$};

\node[draw = white, fill = white] (z) at (\s*1.25 , -1.5) {$v_{2m+1}$};

\node[draw = white, fill = white] (z) at (0 , 7.2*\dy) {$w_{1}$};

\node[draw = white, fill = white] (z) at (\s , -\dy) {$v_{3}$};

\node[draw = white, fill = white] (z) at (\s , 7.2*\dy) {$w_{3}$};

\node[draw = white, fill = white] (z) at (\s*2 , 7.2*\dy) {$w_{2m-1}$};

\node (z1) at (0, 0) {};
\node (z2) at ( \s*2.5, 0) {};
\draw  (z1) -- (z2) ;

\node(bottom) at (1.25*\s,-1) {};
\draw [out = 180, in = -30] (bottom) to (z1);
\draw [out = 0, in = -120] (bottom) to (z2);

\foreach \i in {0,1,2}{
    % Bottom vertex
    \node (v0) at (\s*\i, 0) {};

    % Diamond 3
    \node[draw = white,fill=white] (label) at (\s*\i,3*\dy+0.1) {$\vdots$};
    \node (c3) at (\s*\i, 4*\dy) {};
    
    % Diamond 1
    \node (a1) at (\s*\i-\dx, \dy) {};
    \node (b1) at ( \s*\i+\dx, \dy) {};
    \node (c1) at (\s*\i, 2*\dy) {};
    
    \draw (v0) -- (a1) -- (b1) -- (v0);
    \draw (a1) -- (c1) -- (b1);
    
    % Diamond 4
    \node (a4) at (\s*\i-\dx, 5*\dy) {};
    \node (b4) at (\s*\i+ \dx, 5*\dy) {};
    \node (c4) at (\s*\i, 6*\dy) {};
    
    \draw (c3) -- (a4) -- (b4) -- (c3);
    \draw (a4) -- (c4) -- (b4);
    
    %Make a C2
    \node(c) at (\s*\i + 0.5*\s, 0) {};
    \draw (c) -- (v0);
    \draw[bend right=40]  (c) to (c4);
}

\node[draw = white, fill= white] (z) at (\s*1.75,0) {$\cdots$};

\end{tikzpicture}
\end{center}
\caption{
A  $2$-degenerate  graph that is not flexibly DP $3$-colorable.}
\label{fig:diamond-chains}
\end{figure}

 We briefly note that the family $\mathcal I$ can be modified to give a family of $2$-degenerate (simple) graphs that are not flexibly DP $3$-colorable, which gives a negative answer to \Cref{ques:degen} in the DP-coloring setting.
A \emph{diamond chain} is defined as follows. First, a \emph{diamond} is the graph $K_4^-$ obtained from $K_4$ by deleting an edge.
A diamond is a diamond chain. 
Furthermore, every graph obtained by identifying a $2$-vertex of a diamond with a $2$-vertex of a diamond chain is a diamond chain.
Note that every diamond chain has $3t+1$ vertices for some integer $t \geq 1$.
We write $\mathcal S$ for the family of simple graphs $G$ obtained through the following process.
Begin with a cycle $v_1\dots v_{2m+1}v_1$.
    For each odd $j\in\{1, \ldots, 2m-1\}$, let $D_j$ be a diamond chain, and identify one of the $2$-vertices of $D_i$ with $v_j$.
    Let $w_j$ be the second $2$-vertex of $D_j$, and add an edge $w_j v_{j+1}$.
    The family $\mathcal S$ is shown in \Cref{fig:diamond-chains}.
    It is easy to check that each graph in $\mathcal S$ is $2$-degenerate.
    When  the cover $(H,L)$ joins colors of common indices in the diamond chains, each $(H,L)$-coloring corresponds to a DP $3$-coloring of some $I_m$, 
and thus a similar argument shows that these graphs are
also not flexible  for DP $3$-coloring.

To prove \Cref{thm:main-mad},
we actually prove a statement in terms of potential that implies \Cref{thm:main-mad}.
Given a graph $G$, the \emph{potential} of a vertex set $U\subseteq V(G)$ is defined as $\rho_G(U)=6|U|-4|E(G[U])|$.
We prove the following statement:
\iffalse
{\PB 
Do we want to state the thing below for $\frac 15$ or for general $\epsilon \leq \frac 15$? I think it is true for general $\epsilon$, but to show that, a lot of the reduction proofs would need to be changed, and I am not sure that it is worth it. 
}
{\IC maybe have the ones that can be stated with $\epsilon$ without much effort - but I also don't mind changing everything to $\frac16$ if that makes life easier.}
{\PB My opinion is that always writing $\epsilon = \frac 15$ makes life easier.
I don't think the framework is much use outside of this problem, and we basically solved it completely, so I don't know how useful generality is.
}
\fi

\begin{theorem}
\label{thm:intro-potential}
Suppose that $G$ is not $\frac 15$-flexibly DP $3$-colorable, but every proper subgraph of $G$ is $\frac 15$-flexibly DP $3$-colorable.
    Then  $G\in\cI$ or $\rho_G(V(G)) \leq 0$. 
 %   If $G$ is multigraph that is not $\varepsilon$-flexibly DP $3$-colorable, then either $G\in\cI$ or $\rho_G(V(G))\leq 0$. 
\end{theorem}

% {\IC maybe move all observations and proofs stating these examples are not flexible after the above theorem.}

In fact, we prove an even stronger statement
with a more involved potential function
that implies the theorem above (see~\Cref{sec:our-goal}).

%{\PB I tried moving things around, but I am not quite sure about this organization. The fractional packing section refers to our main result, and I think the reference that it makes is worthwhile. Does it make sense to have Our Results, Fractional Packings, then Preliminaries?}

%\subsection{Fractional packings}

\subsection{Preliminaries}

Let $G$ be a multigraph.
 For an integer $d \geq 1$, we say that $v \in V(G)$ is a \emph{$d$-vertex}, \emph{$d^+$-vertex}, \emph{$d^-$-vertex} if $d(v) = d$, $d(v)\geq d$, $d(v)\leq d$, respectively.
%For a DP $k$-cover $(H, L)$ of $G$, an \emph{$(H,L)$-packing} of $G$ is a set of $k$ $(H,L)$-colorings $\phi_1, \dots, \phi_k$ of $G$ for which $\phi_1 \cup \cdots \cup \phi_k = V(H)$.
%A \emph{fractional $(H,L)$-packing} of $G$ is a distribution on $(H,L)$-colorings $\phi$ of $G$ such that for each $v \in V(G)$ and $c \in L(v)$, $\Pr(\phi(v) = c) = \frac 1k$.
%Given an $(H,L)$-packing $\{\phi_1, \dots, \phi_k\}$ of $G$, the distribution that assigns one of the colorings $\phi_j$ uniformly at random is a fractional $(H,L)$-packing.
The \emph{connectivity} of $G$ is defined as the minimum value $k$ such that either $G - S$ is disconnected for some set $S\subseteq V(G)$ of size $k$ or $G$ is isomorphic to $K_{k+1}$. 
% If every distinct pair of vertices in $G$ is adjacent, then we define the connectivity of $G$ as  the minimum degree $\delta(G)$.
Given a subset $S \subseteq V(G)$, we write $\overline S = V(G) \setminus S$.
Given two disjoint subsets $S, T \subseteq V(G)$, we write $E_G(S,T)$ for the set of edges in $G$ with an endpoint in $S$ and an endpoint in $T$.
When $S$ or $T$ is a single vertex, we simply write that vertex instead of using set notation. 
When it is clear from context, we often remove the subscript. 
We use $G[S,T]$ to denote the bipartite subgraph of $G$ with vertex set $S \cup T$ and edge set $E_G(S,T)$.

Given a graph $G$ with a DP-cover $(H,L)$,
a subgraph $G'$ of $G$,
and a subset $L'(v) \subseteq L(v)$ for each $v \in V(G)$,
we say that an $(H,L')$-coloring of $G'$ 
is an assignment $\phi:V(G') \rightarrow V(H)$
such that $\phi(v) \in L'(v)$ for each $v \in V(G')$, and such that the image of $\phi$ is an independent set in $H$.
In other words,
we often consider partial DP-colorings, as well as DP-colorings with restricted lists, without explicitly defining a new cover.
For each vertex $v$, we typically write $L(v) = \{1_v, 2_v,3_v\}$.

\section{The Setup}\label{sec:multigraphs}

\subsection{Our framework}

%When we consider a DP $3$-cover $(H,L)$ of a graph $G$, we imagine that each list $L(v)$ is a \emph{pointed set}, that is, a set with a single distinguished element {\IC $\omega(v)$. 
%\sout{Given a list $L(v)$, we write $\omega(v)$ for the distinguished element of $L(v)$.}}

Let $G$ be a loopless (multi)graph.
%{\IC\sout{with a DP $3$-cover $(H,L)$.}} 
% {\PB Note: Ilkyoo, I am modifying your suggestion because I am afraid that if we make $\Pi$ be the thing that defines $\rho$ and $h$, then we need all of our tuples to have $\Pi$ instead of $\rho$. 
% Below is a suggestion where $\rho$ is the fundamental object that defines everything else.}
% {\IC this also works.}
 A function $\rho:V(G) \rightarrow \{3,4,6\}$ is called a \emph{potential function}.
We say that $\rho$ defines a partition $\Pi = \Pi(\rho)$ of $V(G)$ into parts $\Pi_3, \Pi_4, \Pi_6$, where $\Pi_i = \rho^{-1}(i)$ for each $i \in \{3,4,6\}$.
We also say that $\rho$ defines a \emph{list-size function} $h_{\rho}$ as follows:
$h_{\rho}(v) = 2$ for each $v \in \Pi_3$, and $h_{\rho}(v) = 3$ for each $v \in \Pi_4 \cup \Pi_6$.
Let $(H, L)$ be a DP $3$-cover of $G$.
An \emph{$h_\rho$-list assignment} $L'$ on $(G,\rho, H,L)$ is a function that assigns a list $L'(v) \subseteq L(v)$ of size $h_{\rho}(v)$ to each $v \in V(G)$.
%We define $\Pi_{\geq i}$ and $\Pi_{\leq i}$ similarly.
 %{\IC Can we say $\Pi$ is a partition of $V(G)$ into parts $\Pi_3, \Pi_4, \Pi_6$, and then say both $\rho$ and $h$ is defined by $\Pi$? We can say $\rho=\rho_\Pi$ and $h=h_\Pi$. This way we can refer to $\Pi$ instead of referring to $\rho$ and $h$ whenever necessary. I volunteer to change things if we think this is better. }
 %{\PB That would be fantastic}
A \emph{list distribution} $\mathcal D$ on $(G,\rho,H,L)$ is a function that assigns a probability to each $h_{\rho}$-list assignment on $(G,H,L)$. %, where $h$ is defined by $\rho$.
Hence, the distribution $\mathcal D$ produces a random $h_{\rho}$-list assignment $L_\D$ on $(G, \rho,H,L)$.
%{\PB I read about probability, and it seems that a ``distribution" is necessarily a measure function.
%I tried to rewrite everything to be precise and also to be clear. If you have any suggestions about notation, please say so.
%}
%{\IC so far so good.}
For a real  $\varepsilon\in(0,1)$, we say that $\mathcal D$ is \emph{$\epsilon$-admissible} on $(G,\rho,H,L)$ if the following holds: 
\begin{itemize}
    \item For each $v \in V(G)$ with $h_{\rho}(v)=2$ and $c \in L(v)$, $\Pr( \{c\} = L(v) \setminus L_{\D}(v)) \geq \epsilon$.
\end{itemize}
For a vertex $v \in V(G)$ with $h_{\rho}(v) = 2$, the unique color in $L(v) \setminus L_{\mathcal D}(v)$ is \emph{forbidden} by $L_{\D}$.

A \emph{coloring distribution} $\Phi$ on $(G,\rho,H,L,\mathcal D)$ is a function that maps each $h_{ \rho}$-list assignment $L_{\mathcal D}$ produced by $\mathcal D$  to a
 distribution $\Phi(L_{\mathcal D})$ 
 on $(H,L_{\mathcal D})$-colorings of $G$.
Abusing notation slightly, we say that
$\Phi$ assigns a probability to each pair $(L_{\mathcal D}, \phi)$, where $L_{\mathcal D}$ is an outcome of $\mathcal D$ and $\phi$ is an $(H,L_{\mathcal D})$-coloring of $G$.
% {\IC (just to be pedantic: first you say $\Phi$ maps an $h$-list assignment to distribution, then you say $\Phi$ maps $(L_\D, \phi)$ to a probability. 
% These two are technically different, no? 
 %{\PB You're right. Rather than saying ``formally," 
 %would it be better to say ``in this way"?
 %I guess it's like $f$ is a function taking inputs from $X$ and giving functions from $Y$ to $Z$ as output, then the function $(x,y) \mapsto f(x)(y)$ is implicitly defined.
 %}
 %{\IC I think we should say we pick one to be the `formal' definition, and then say we abuse notation a bit and allow the other definition at times.}
 %The former is $A\rightarrow (B\rightarrow C)$ where the latter is $A\times B\rightarrow C$ for appropriate $A, B, C$. The reader should be able to understand though.}
Furthermore, $\Phi$ satisfies the property that for each outcome $L_{\mathcal D}$ of $\mathcal D$, 
$\Pr(L_{\mathcal D}) =   \sum_{\phi} \Pr(L_{\mathcal D},\phi  ) $, where the sum is taken over all  possible $(H,L_{\mathcal D})$-colorings $\phi$ of $G$.
We note that $\Phi$ also gives a distribution on $(H,L)$-colorings $\phi$ of $G$ by letting $\Pr(\phi) = \sum_{L_{\mathcal D}} \Pr(L_{\mathcal D},\phi)$, 
where the sum is taken over all possible outcomes $L_\D$ of $\D$.
%
%A \emph{random $\mathcal D$-coloring}
%is a function $\Phi$ that maps each $h$-list assignment $L_{\mathcal D}$ produced by $\mathcal D$ to a
 %distribution $\Phi(L_{\mathcal D})$ 
 %on $(H,L_{\mathcal D})$-colorings of $G$.
%We note that a random $\mathcal D$-coloring $\Phi$ gives a distribution on $(H,L)$-colorings of $G$ via the following experiment: First, randomly choose an $h$-list assignment $L_{\mathcal D}$ on $(G,H,L)$; then, use the
%distribution
%$\Phi(L_{\mathcal D})$ to give $G$ a random $(H,L_{\mathcal D})$-coloring $\phi$.
%We also note that the existence of a random $\mathcal D$-coloring $\Phi$ implies that $G$ has an $(H,L')$-coloring for every $h$-list assignment $L'$ on $(G,H,L)$.
%
%{\PB I tried to rewrite things so that we can refer to an $(H,L)$ and $\mathcal D$ 
%making $(G,\rho)$ flexible or inflexible. This will make certain things easier later.}
%

For convenience, we let each list $L(v)$ be a \emph{pointed set}, that is, a set with a distinguished \emph{basepoint} $\omega(v) \in L(v)$.
Then  we say that a coloring distribution $\Phi$ on $(G,\rho,H,L,\mathcal D)$ is an \emph{$\epsilon$-distribution} if 
the random $(H,L)$-coloring $\phi$ of $G$ produced by $\Phi$ satisfies the following
%$(G,\rho)$ and a DP $3$-cover $(H,L)$ of $G$, 
%and given an $\epsilon$-admissible $h$-list distribution $\mathcal D$ on $(G, H, L)$, 
%we say that
%$(G,\rho,H,L,\mathcal D)$ is \emph{$\epsilon$-flexible} if 
%there exists a random $\mathcal D$-coloring distribution $\Phi$ such that the random $(H,L{\IC_\D})$-coloring $\phi$ of $G$ 
%produced by $\Phi{\IC (L_\D)}$
%satisfies the following {\IC for every color $c\in L_\D(v)$ for every vertex $v$}:
for every $v \in V(G)$: % and $c \in L(v)$: 
% {\PB Here is the definition!}
\begin{itemize}
    \item $\Pr(\phi(v) = c)  \geq \epsilon $  for each $c \in L(v)$;
    \item $\Pr(\phi(v) = \omega(v)) = 2 \epsilon$, and $\Pr(\phi(v) = c) = \frac 12 - \epsilon$ for each $c \in L(v) \setminus \{\omega(v)\}$ whenever $\rho(v) = 4$.
\end{itemize}
Note that the probabilities above are considered with respect to the overall distribution $\Phi$, not with respect to a distribution $\Phi(L_{\mathcal D})$ corresponding with a specific outcome $L_{\mathcal D}$ of $\mathcal D$.
If $(G,\rho,H,L,\mathcal D)$ has an $\epsilon$-distribution, then we say that $(G,\rho,H,L,\mathcal D)$ is \emph{$\epsilon$-flexible}.
%We say that such a random $\mathcal D$-coloring distribution is an \emph{$\epsilon$-distribution} for $(G,\rho,H,L,\mathcal D)$.

Finally, we say that $(G,\rho)$ is \emph{$\epsilon$-flexible} if $(G,\rho,H,L,\mathcal D)$ is $\epsilon$-flexible for every
DP $3$-cover $(H,L)$ of $G$ and every $\epsilon$-admissible
list distribution $\mathcal D$ on $(G,\rho,H,L)$.
We simply say that $G$ is \emph{$\epsilon$-flexible} if $(G,\rho)$ is $\epsilon$-flexible for the constant function $\rho \equiv 6$.
We say that $G$ is \emph{$(\rho, \epsilon)$-critical} if $(G, \rho)$ is not $\varepsilon$-flexible, but for every proper subgraph $G'$ of $G$, $(G', \rho)$ is $\varepsilon$-flexible (when $\rho$ is restricted to $V(G')$).

\begin{observation}
    \label{obs:minimal}
    If $(G,\rho)$ is not $\epsilon$-flexible, then $G$ has a subgraph $G'$ that is $(\rho,\epsilon)$-critical.
\end{observation}

Given a pair $(G,\rho)$ and
a vertex set $U \subseteq V(G)$, we define
 $\sigma_U(v) = \rho(v) - 2d_U(v) $ for each $v \in U$. Then  we define the \emph{potential} of $U$ as 
\[\rho_{G}(U) = \sum_{v \in U} \rho(v) - 4 |E(G[U])| = \sum_{v \in U} (\rho(v) - 2d_U(v) )=\sum_{v\in U}  \sigma_U(v) .\]
We also write $\sigma(v) = \sigma_{V(G)}(v) = \rho(v) - 2d(v)$ for each $v \in V(G)$, so that $\rho_{G}(V(G)) = \sum_{v \in V(G)} \sigma(v)$.

\subsection{Our goal}
\label{sec:our-goal}

In order to prove \Cref{thm:main-mad},
we prove the following stronger statement.

\begin{theorem}\label{thethm}
    If $G$ is $(\rho,\frac 15)$-critical, then one of the following holds:
    \begin{itemize}
        \item $G\in \cI$ and $\rho_G(v)=6$ for all $v\in V(G)$,
        \item $G$ is a $C_2$ satisfying $\rho_G(V(G)) = 2$,  or
        \item $\rho_G(V(G))\leq 0$.
    \end{itemize}
    
\end{theorem}

Given a graph $G$ with a potential function $\rho$,
we say that a subgraph $G'$ of $G$ is \emph{exceptional} if $G' \in \mathcal I$ or $G'$ is a $C_2$ for which $\rho_G(V(G')) = 2$.
Note that an exceptional $C_2$ necessarily has one  vertex of $\Pi_4$ and one vertex of $\Pi_6$.
Thus, \Cref{thethm} states that if $G$ is $(\rho, \frac 15)$-critical,
then either $V(G)$ has nonpositive potential, or $G$ is exceptional. 

We will prove \Cref{thethm} by considering a minimum counterexample $(G,\rho)$ to the theorem, and then use minimality to conclude many properties of $(G,\rho)$. 
Then  we will use a discharging argument to reach a contradiction.

\iffalse
\Cref{thethm} has the following immediate corollary.
\begin{corollary}
    If $G$ is a multigraph with maximum average degree less than $3$,
    then $G$ has a subgraph in $\mathcal I$, or $G$ is $\epsilon$-flexibly DP $3$-colorable.
\end{corollary}
\begin{proof}
    % {\PB The ordinary flexible DP-coloring setting is represented by letting $h=3$ and everything be light}
    Suppose that $G$ is not $\epsilon$-flexibly DP $3$-colorable for a $3$-cover $(H, L)$ of $G$. 
        % Let $(H,L)$ be a $3$-cover of $G$ for which there is no distribution on $(H,L)$-colorings $\phi$ that satisfies $\Pr(\phi(v) = c) \geq \epsilon$ for each $v \in V(G)$ and $c \in L(v)$. 
    In particular, $(G, \rho)$ is not $\epsilon$-flexible when $\rho(v) = 6$ for all $v \in V(G)$. 
    Therefore, by \Cref{obs:minimal}, $G$ contains a subgraph $G'$ that is $(\rho,\epsilon)$-critical.
    As the average degree of $G'$ is less than $3$,
    \[\rho_{G'}(V(G')) = 6|V(G')| - 4|E(G')|  > 0.\]
    Therefore, by \Cref{thethm}, $G' \in \mathcal I$. 
\end{proof}
{\PB I'm not sure if we need this corollary, since it is basically the same as \Cref{thm:main-mad}.}
\fi

\subsection{Tightness examples}
In this section, we show that the bound on the potential in \Cref{thethm}
is tight by describing a family $\mathcal I$ of
$(\rho,\epsilon)$-critical graphs with potential $2$ for all $\epsilon>0$.
We also note that the graph $K_4$ with potential  function $\rho \equiv 6$ is $(\rho,\epsilon)$-critical for every $\epsilon > 0$ and has potential $0$,
so the upper bound of $\rho_G(V(G)) \leq 0$ cannot be decreased for graphs outside of $\mathcal I$.
In addition to $K_4$, we show an infinite family of pairs $(G,\rho)$ satisfying $\rho_G(V(G)) = 0$ that are not $\epsilon$-flexible for $\epsilon > \frac 16$.
We finally show that our value of $\epsilon = \frac 15$
is tight for an infinite family of graphs with potential $2$.

\begin{observation}
\label{obs:C2}
    Let $G$ be isomorphic to $C_2$ with a potential function $\rho$ satisfying $\rho_G(V(G)) = 2$. Then  for every $\epsilon > \frac 16$, $(G,\rho)$ is not $\epsilon$-flexible.
\end{observation}
\begin{proof}
    Let $V(G) = \{u,v\}$.
    Since $\rho_G(V(G)) = 2$, it 
    holds without loss of generality that $\rho(u) = 4$ and $\rho(v) = 6$.
    Consider the following $3$-cover $(H,L)$ of $G$. Let $H$ contain the edge $i_u i_v$ for each $i \in \{1,2,3\}$, as well as the edges $1_u 2_v, 2_u 1_v$.
    Let $\mathcal D$ be the empty distribution on forbidden colors.
    Finally, let $L(u)$ have a basepoint $\omega(u) = 1_u$.

    Now, suppose that we have a distribution on $(H,L)$-colorings $\phi$ of $G$.
    Since the events $\phi(v) = 1_v$ and $\phi(v) = 2_v$ occur only if $\phi(u) = 3_u$, we have without loss of generality that $\Pr(\phi(v) = 1_v) \leq \frac 12 \Pr(\phi(u) = 3_u) = \frac 14 - \frac 12 \epsilon < \epsilon$, as $\epsilon > \frac 16$.
    Therefore, $(G,\rho,H,L,\mathcal D)$ has no $\epsilon$-distribution, so $(G,\rho)$ is not $\epsilon$-flexible.
\end{proof}
It is easy to see that 
if the pair $(G,\rho)$ in \Cref{obs:C2}
has a cover $(H,L)$
such that $(G,\rho,H,L,\mathcal D)$ admits no $\frac 15$-distribution (where $\mathcal D$ is the empty forbidden color distribution),
then writing $V(G) = \{u,v\}$,
$H[L(u), L(v)]$ consists of a $C_2$ and a $C_4$.
% For each $m \geq 1$, we define a graph $I_m$ as follows. We begin with a $(2m+1)$-cycle $v_1,v_2,\ldots,v_{2m+1},v_1$, and we duplicate the edge $v_{2i-1}v_{2i}$ for each $1\leq i\leq m$.
%{\IC I_m is already defined in the intro}

%{\IC we should probably define the edges in the cover graph too..}
%{\PB I talked about this with Sasha. He seems to prefer just defining the graphs $I_m$. However, regardless of how we do things, it's probably a good idea to say which cover makes them critical, or else it's not clear why there's defined at all.}
% Next,  we let each $v_i \in V(I_m)$ be a light vertex, and we assign $h(v_i) = 3$.
%Thus $I_m$ has $3m+1$ edges and is not flexibly DP 3-colorable.
%We have $\rho(V(I_m))=6(2m+1)-4(3m+1)=2$. 
% We write $\cI=\{I_m: m\geq 1\}$.

 Recall the definitions of $I_m$ and $\mathcal I$ from the introduction.

\begin{observation}
\label{obs:Ilkyoo}
    Let $m \geq 1$, and
    let $\rho$ be a potential function for $I_m$.
    If $\rho(v) = 6$ for each $v \in V(I_m)$, then
   $\rho(V(I_m)) = 2$. Otherwise, $\rho(V(I_m)) \leq 0$.
\end{observation}
\begin{proof}
    Suppose first that $\rho(v) = 6$
    for each $v \in V(I_m)$. Then, as $I_m$ has $2m+1$ vertices and $3m+1$ edges,
    \[\rho_{I_m}(V(I_m)) = 6(2m+1) - 4(3m+1) = 2.\]
    Next, if some vertex $v \in V(I_m)$ satisfies $\rho(v) \leq 4$,
    then
    \[\rho_{I_m}(V(I_m)) \leq 6(2m+1) - 4(3m+1) -2 = 0.\]
\end{proof}

The following observation shows that 
%even when $\rho_{I_m}(V(I_m)) = 2$, it is possible that 
for every positive $\epsilon$,
$I_m$
is 
not $\epsilon$-flexible.
%\sout{$\epsilon$-flexible for 
%no 
%positive $\epsilon,\alpha$. }}
%{\PB I put a comma before ``for"; does this seem okay? I wanted to make it clear that the modified clause of ``not" is ``$\epsilon$-flexible," not ``flexible for every positive...."
%}
%{\IC Sure. We can also go with `it is possible that for every  $\epsilon, \alpha>0$,  $(I_m, \rho)$ is not $\epsilon$-flexible.' Your choice. }

\begin{observation}
\label{obs:Ilkyoo-inflexible}
    If $I_m \in \mathcal I$, 
    then for every $\epsilon> 0$ and every function $\rho$ on $I_m$,
    $(I_m,\rho)$ is not $\epsilon$-flexible.
    %{\PB Using the definition introduced above, can we just say that $I_m$ is not $\epsilon$-flexible?}
    %{\IC this works.}
    %{\PB Maybe I forgot what $(H,L)$ is. I thought it has an $(H,L)$-coloring, but it's impossible to color the $2$-vertex with color $3$.}
    %{\IC you're right.}
    %{\IC (should we define a graph $G$ to be \emph{$\epsilon$-flexible} if $(G, \rho)$ is $\epsilon$-flexible for every $\rho$?)}
    %{\PB We could, but this gets annoying when you consider constant functions $\rho \equiv 3$.
    %Could we define it this way for \emph{some} $\rho$?
    %Or equivalently, for $\rho \equiv 6$.
    %}
    %{\IC what part is annoying? I meant define both `$G$ is $\epsilon$-flexible' and `$(G, \rho)$ is $\epsilon$-flexible'}
    %{\PB It's annoying because most graphs $G$ will not be flexible, not even cycles, because of $\rho \equiv 3$.}
\end{observation}
\begin{proof}
    Consider the following DP $3$-cover $(H,L)$ of $I_m$.
   % {\IC\sout{For each $v \in V(I_m)$, we write $L(v) = \{1_v, 2_v,3_v\}$. }}
    For each edge $uv \in E( I_m)$, let $H$ contain the edge $i_u i_v$ for $i \in \{1,2,3\}$. 
    Furthermore, for each pair $u,v \in V(I_m)$ that induces a $C_2$, include the additional edges $1_u 2_v$ and $2_u 1_v$ in $H$.

    We claim that $I_m$ has no $(H,L)$-coloring $\phi$ for which $\phi(v_{2m+1}) = 3_{v_{2m+1}}$. This claim implies the statement of the observation, as it implies that no distribution on $(H,L)$-colorings of $I_m$ assigns $3_{v_{2m+1}}$ to $v_{2m+1}$ with positive probability. Thus, we aim to prove this claim.

    Suppose we assign $\phi(v_{2m+1}) = 3_{v_{2m+1}}$. Then  we are forced to assign $\phi(v_1) \in \{1_{v_1}, 2_{v_1}\}$, which forces $\phi(v_2) = 3_{v_2}$. This again forces  $\phi(v_3) \in \{1_{v_3}, 2_{v_3}\}$, which forces $\phi(v_4) = 3_{v_4}$. Continuing this pattern, we are forced to assign $\phi(v_{2m}) = 3_{2m}$. Then  the vertices $v_{2m}$ and $v_{2m+1}$ are assigned adjacent colors $3_{v_{2m}}$ and $3_{v_{2m+1}}$, showing that $\phi$ is not an $(H,L)$-coloring. This proves the claim and implies the statement in the observation.
\end{proof}

 An observation of the family $\mathcal I$ is given in \Cref{obs:Im+1}.

\begin{observation}
    \label{obs:Im+1}
    Let $G$ be a graph, and let $u'v'$ be an edge of multiplicity $1$.
    Obtain $G'$ from $G$ by deleting the edge $u'v'$, adding two new vertices $u$ and $v$, and finally adding edges $u' u$, $v'v$, and two copies of $uv$. 
    If $G' \in \mathcal I$, then $G \in \mathcal I$.
\end{observation}

 Recall the definitions of $D_i$ and $\mathcal S$ from the introduction.

\begin{observation}
    If $G \in \mathcal S$, then for every $\epsilon > 0$ and every potential function $\rho$ on $G$, $(G,\rho)$ is not $\epsilon$-flexible.
\end{observation}
\begin{proof}
    Let $G \in \mathcal S$ be obtained from a cycle $C$ of length $2m+1$.
    Consider the following DP $3$-cover $(H',L')$ of $G$.
    For each edge $uv$ belonging to the cycle $C$ or to a diamond chain $D_j$, let $H'$ contain the edge $i_u i_v$. For each edge $w_j v_{j+1}$, let $H'$ contain the edges $1_{w_j} 2_{v_{j+1}}, 2_{w_j} 1_{v_{j+1}}, 3_{w_j}  3_{v_{j+1}}$.
    By construction of $H'$, for each odd $j\in\{1, \ldots, 2m-1\}$, and for each $(H',L')$-coloring $\phi$ of $G$, there is some $i \in \{1,2,3\}$ for which $\phi(v_j) = i_{v_j}$ and $\phi(w_j) = i_{w_j}$.
    Therefore, 
    for each odd $j\in\{1, \ldots,  2m-1\}$,
    $\phi$ does not contain the pair $1_{v_j} 2_{v_{j+1}}$ or $2_{v_j} 1_{v_{j+1}}$.

    Now, we claim that $G$ has no $(H',L')$-coloring $\phi$ for which $\phi(v_{2m+1}) = 3_{v_{2m+1}}$.
    Indeed, given such an $(H',L')$-coloring $\phi$, our observation above shows that the restriction of $\phi$ to $H[L(V(C))]$
    gives an $(H,L)$-coloring of $I_m$ with respect to the cover $(H,L)$ described in  \Cref{obs:Ilkyoo-inflexible}.
    As $\phi(v_{2m+1}) = 3_{v_{2m+1}}$,
    this is a contradiction.
    This proves the claim and implies the statement in the observation.
\end{proof}

Finally, we give an 
infinite family of pairs $(G,\rho)$ with $\rho_G(V(G)) = 2$
for which $G$ is not $\epsilon$-flexibly DP $3$-colorable for any value $\epsilon > \frac 15$.
Recall the definitions of $J_m$ and $\mathcal J$ from the introduction.
 We observe that when $\rho \equiv 6$ is a potential function for $J_m$, $\rho_{J_m}(V(J_m)) = 2$.

\begin{observation}
    If $J_m \in \mathcal J$, then for every  $\epsilon > \frac 15$ and potential function $\rho$ on $J_m$,
    $(J_m,\rho)$ is not $\epsilon$-flexible.
\end{observation}
\begin{proof}
    We consider the following $3$-cover $(H,L)$ of $J_m$.
    For each $j\in\{1, \ldots, 2m+1\}$ and $i \in \{1,2,3\}$, let $H$ contain the edge $i_{v_j} i_{v_{j+1}}$.
    For each even value $j\in\{2, \ldots, 2m-2\}$,
    let $H$ also contain the edges $1_{v_j} 2_{v_{j+1}}, 2_{v_j} 1_{v_{j+1}}$.
    Finally, for $j \in \{1,2m\}$ and $i \in \{1,2,3\}$, let $H$ contain the edges $i_{v_j} i_{w_j}$, as well as the edges
    $1_{v_j} 2_{w_j}, 2_{v_j} 1_{w_j}$.
    % {\IC edges between $v_{2m+1}$ and $v_1$?}

    Let $\Phi$ be a distribution on $(H,L)$-colorings $\phi$ of $J_m$.
    We claim that the events $\phi(v_1) = 3_{v_1}$, $\phi(v_{2m}) = 3_{v_{2m}}$, and $\phi(v_{2m+1}) = 3_{v_{2m+1}}$ are pairwise disjoint.
    By our construction of $H$, clearly the event at $v_{2m+1}$ is disjoint with each of the other two events.
    Now, we show that the event $\phi(v_1) = 3_{v_1}$ is disjoint with the event that 
    $\phi(v_{2m}) = 3_{v_{2m}}$.
    To this end, suppose that $\phi(v_1) = 3_{v_1}$.
    Then, $\phi(v_2) \neq 3_{v_2}$, so by construction of $H$,
    the assignment $\phi(v_3) = 3_{v_3}$ is forced.
    Then, $\phi(v_4) \neq 3_{v_4}$, so by construction of $H$,
    the assignment $\phi(v_5) = 3_{v_5}$ is forced.
    Continuing this pattern, the assignment $\phi(v_{2m-1}) = 3_{v_{2m-1}}$ is forced, so that $\phi(v_{2m}) \neq 3_{v_{2m}}$.
    Therefore, the three given events are pairwise disjoint.

    Now, suppose that for each $v \in V(J_m)$ and $c \in L(v)$, $\Pr(\phi(v) = c) \geq \epsilon$, for some $\epsilon > 0$.
    We see that the assignment $\phi(w_1) = 1_{w_1}$ and $\phi(w_1) = 2_{w_1}$ only occur when $\phi(v_1) = 3_{v_1}$.
    Thus, without loss of generality, $\epsilon \leq \Pr(\phi(w_1) = 1_{w_1}) \leq \frac 12 \Pr(\phi(v_1) = 3_{v_1})$.
    Rearranging, $\Pr(\phi(v_1) = 3_{v_1}) \geq 2 \epsilon$.
    By a similar argument, $\Pr(\phi(v_{2m}) = 3_{v_{2m}}) \geq 2 \epsilon$.
    Finally, $\Pr(\phi(v_{2m+1}) = 3_{v_{2m+1}}) \geq \epsilon$.
    Since these three events are pairwise disjoint, $2 \epsilon + 2 \epsilon + \epsilon \leq 1$, implying that $\epsilon \leq \frac 15$.
\end{proof}

A similar argument shows that if $G$ is the graph obtained from $J_m$ by adding a new vertex $w_3$ adjacent to $v_{2m+1}$ by parallel edges, and if $\rho \equiv 6$ is a constant function,
then 
$\rho_G(V(G)) = 0$ and 
$(G,\rho)$
is not $\epsilon$-flexible for any $\epsilon > \frac 16$.
Therefore,
our upper bound on $\rho_G(V(G))$ in  \Cref{thethm}
cannot be improved to $\rho_G(V(G)) \leq -1$ even when $K_4$ is excluded.

\iffalse
{\PB With the figure, is the proof even necessary?}

{\AK Maybe not, but I wanted to complete the hw.}
\begin{proof} By definition, $I_m$ is obtained from a $(2m+1)$-cycle $v_1,v_2,\ldots,v_{2m+1},v_1$ by duplicating the edge $v_{2i-1}v_{2i}$ for each $1\leq i\leq m$. Let $u'v' \in E(I_m)$ be an edge of multiplicity $1$. By renaming the vertices, we may assume that $u'v'=v_{2j}v_{2j+1}$ for some $1\leq j\leq m$.
Then the multigraph $G'$ described in the statement is simply  $I_{m+1}$. %This proves the lemma.
\end{proof}
\fi

\subsection{Some lemmas}
%%%%%%%%%%%%%%%%%%%%%%%%%%%%%%%%%%%
%%%%%% reductive lemmas %%%%%%%
%%%%%%%%%%%%%%%%%%%%%%%%%%%%%%%%%%%

% {\IC All lemmas in this subsection except \Cref{lem:path} do not use our framework. Maybe create a new section called `preliminaries' and let it contain the subsection 1.5 named 'preliminaries' and the lemmas in this subsection, except \Cref{lem:path}. }
% {\PB This is fine with me}
% {\IC Hmm.. then again, the intro is already 7 pages long. And including a new preliminary section will make the actual statement we prove appear much later in the paper. Need to rethink this. }
% {\PB Then maybe it's better to leave these here?}
% {\IC sure. Does Sasha have an opinion?}

The following statement is implicitly observed in~\cite{2018DvPo}. We include a proof for completeness.
\begin{lemma}
\label{lem:tree-packing}
    Let $T$ be a tree.
%    and let $\ell:V(T) \rightarrow \{2,3\}$.
    If $(H,L)$ is a $2$-cover of $T$, then $T$ has
    an $(H,L)$-packing;
    that is, there are two disjoint
    $(H,L)$-colorings $\phi_1,\phi_2$ such that for each $c \in V(H)$, $c$ belongs to exactly one of $\phi_1, \phi_2$.
\end{lemma}
\begin{proof}
    Let $(H, L$) be a $2$-cover of $T$. 
    %We obtain a $2$-cover $(H',L')$ of $T$ by deleting a color uniformly at random from each list $L(v)$ of size $3$.
      Since $T$ is acyclic, 
      for each edge $vw$ of $T$, we may assume without loss of generality that
     $H$ contains the edges $1_{v} 1_{w}$ and $2_{v} 2_{w}$. 
     Fix a root $r \in V(T)$.
    For $j \in \{1,2\}$, let $\phi_j:V(T) \rightarrow V(H)$ assign $1_{v}$ to each $v$ for which $\dist(r,v) + j$ is odd and assign $2_{v}$ to each $v$ for which $\dist(r,v)+j$ is even. 
    %Then, the distribution obtained by choosing each of $\phi_1$ and $\phi_2$ with probability $\frac 12$ 
    %is a fractional $(H',L')$-packing of $T$.
    %Therefore, if $\ell(u) = 2$ and $c \in L(u)$, then $\Pr(\phi(u) = c) = \frac 12$. 
    %If $\ell(u) = 3$ and $c \in L(u)$, then $c \in L'(u)$ with probability $\frac 23$, and hence $\Pr(\phi(u) = c) = \frac 13$.
    This gives us our two disjoint $(H,L)$-colorings $\phi_1, \phi_2$.
\end{proof}

Using \Cref{lem:tree-packing},
Cambie,
Cames van Batenburg,
and  Zhu~\cite[Theorem 9.4]{arXiv_CaCaZh2023} proved the following result. 
A graph $G$ is \emph{almost cubic} if $G$ is $2$-connected and every vertex is a $3$-vertex except for one $2$-vertex. 

\begin{lemma}
    \label{lem:almost-cubic}
    Let $G$ be a graph of maximum average degree less than $3$ with no almost cubic subgraph. If $(H,L)$ is a DP $3$-cover of $G$, then $G$ has a fractional $(H,L)$-packing.
\end{lemma}

\Cref{lem:almost-cubic} has the following corollary.

\begin{corollary}
\label{lem:red-3}
    Let $G$ be a $2$-connected simple graph with maximum degree at most $3$ and at least two $2$-vertices.
    If $(H,L)$ is a DP $3$-cover of $G$, then $G$ has a fractional $(H,L)$-packing.
\end{corollary}
\begin{proof}
    Since $G$ has no $3$-regular subgraph, the maximum average degree of $G$ is less than $3$. 
    To apply \Cref{lem:almost-cubic}, we need to show that $G$ has no almost cubic subgraph. Suppose that $G'$ is an almost cubic subgraph of $G$.
    If $G'$ is not induced, then $G'$ contains at least two $2^-$-vertices; therefore, $G'$ is induced.
    As $G$ has two $2$-vertices, $G'$ is a proper subgraph of $G$;
    this further implies  $G'$ is joined to $G - V(G')$ by a single cut-edge, contradicting the $2$-connectedness of $G$.
\end{proof}

Using \Cref{lem:tree-packing},
we obtain the following lemma, which is similar in spirit to \Cref{lem:red-3}.

\begin{lemma} 
\label{lem:path}
    Let $G$ be a connected graph with a potential function $\rho$, and suppose that the following hold:
    \begin{itemize}
        %\item {\IC \sout{$G$ is connected and simple.}}
        \item For each nonempty subset $U \subseteq V(G)$, $\rho_G(U) \geq 4$. In particular, $\rho(v) \geq 4$ for each $v \in V(G)$.
        \item For each $v \in V(G)$, $d(v) \leq \frac 12 \rho(v)$.
        %\item $G$ has two vertices $u,v$ for which $\rho(v) - 2d(v) \geq 2$.
    \end{itemize}
    Then, for every $3$-cover $(H,L)$ of $G$,
    there exists a distribution on $(H,L)$-colorings $\phi$ of $G$ for which the following hold:
    \begin{itemize}
        \item For each $v \in V(G)$ and $c \in L(v)$, $\Pr(\phi(v) = c) \in [\frac 3{10} , \frac 25]$.
        \item For each $v \in \Pi_4$ and $c \in L(v) \setminus \{\omega(v)\}$,
        $\Pr(\phi(v) = c) = \frac 3{10}$.
    \end{itemize}
\end{lemma}
\begin{proof}
    Let $(H,L)$ be a $3$-cover of $G$, and let $\rho$ be a potential function satisfying the condition in the lemma.
    Note that $\rho$ is even-valued.
    We proceed by induction on $|V(G)|$.
    If $|V(G)| = 1$, then the lemma clearly holds.
    Therefore, we assume that $|V(G)| \geq 2$.

    We first claim that there are two vertices $v \in V(G)$ for which $\sigma(v) \geq 2$.
    Indeed, we have $\rho_G(V(G)) = \sum_{v \in V(G)} \sigma(v) \geq 4$.
    Therefore, by parity,
     either exactly one vertex $v \in V(G)$ satisfies $\sigma(v) \geq 4$, or two vertices $v \in V(G)$ satisfy $\sigma(v) \geq 2$.
    In the first case, since $G$ is connected, it holds that $\rho(v) = 6$ and $d(v) = 1$, so that $\sigma(v) = \rho_G(V(G)) = 4$.
    Then, $\rho_G(V(G) \setminus \{v\}) = \rho_G(V(G)) - \rho(v) + 4 = \rho_G(V(G)) - 2 = 2$, a contradiction. Therefore, there are two vertices $v \in V(G)$ satisfying $\sigma(v) \geq 2$.
    Let $u,v \in V(G)$ be two such vertices.
    
    Let $P$ be a shortest path from $u$ to $v$, and observe that $P$ is induced.
    Then, let $G' = G - V(P)$.
    Observe that each $v \in V(P)$ has at most one neighbor in $G'$.
    We claim that $G'$ has a distribution $\Phi'$ on $(H,L)$-colorings satisfying the properties in the lemma. 
    For this, it is enough to observe that the conditions of the lemma hold for each component of $G'$ by the induction hypothesis.

    Now, we construct a distribution $\Phi$ on $(H,L)$-colorings $\phi$ of $G$.
    First, sample an $(H,L)$-coloring $\phi'$ of $G'$ using $\Phi'$.  Let $\phi$ agree with $\phi'$ on $G'$.
    Next, for each $v \in V(P)$ with a neighbor in $G'$, let $L'(v) = L(v) \setminus N_H(\phi)$, and observe that $|L'(v)| = 2$.
    If $v$ has no neighbor in $G'$ (which holds in particular for all vertices of $P$ in $\Pi_4$),
    then obtain $L'(v)$ from $L(v)$ by deleting $\omega(v)$ with probability $\frac 15$ and deleting each other color of $L(v)$ with probability $\frac 25$.
    Finally, use \Cref{lem:tree-packing}
    to find an $(H,L')$-coloring $\phi$ of $P$ such that for each $v \in V(P)$ and $c \in L'(v)$, 
    the conditional probability that $\phi(v) = c$ is $\frac 12$.

    Now, we check that $\Pr(\phi(v ) = c)$ satisfies the lemma for each $v \in V(G)$ and $c \in L(v)$.
    If $v \in V(G')$, then by the induction hypothesis,
    $\Pr(\phi(v) = c) \in [\frac 3{10}, \frac 25]$, 
    and $\Pr(\phi(v) = c) = \frac{3}{10}$ whenever $v \in \Pi_4$ and $c \in L(v) \setminus \{\omega(v)\}$.
    Otherwise, suppose that $v \in V(P)$.
    If $\rho(v) = 4$,
    then $\Pr(\phi(v) = c) = \frac 35 \cdot \frac 12 = \frac 3{10}$ for each $c \in L(v) \setminus \{\omega(v)\}$.
    If $\rho(v) = 6$,
    then for each $c \in L(v)$, $c \in L'(v)$ with probability in $[\frac 35, \frac 7{10}]$.
    Therefore, $\Pr(\phi(v) = c) \in [\frac 3{10}, \frac 7{20}] \subseteq  [\frac 3{10}, \frac 25]$. 
    Therefore, the distribution $\Phi$ satisfies the conditions of the lemma.
\end{proof}

Our final lemma
of this section
shows that any distribution on $(H,L)$-colorings of a graph $G$ with rational probabilities at each color can be realized
as a uniform distribution over a multiset of $(H,L)$-colorings.
\begin{lemma}
    \label{lem:rational}
    Let $G$ be a graph, and let $(H,L)$ be a DP $3$-cover of $G$.
    Suppose that $\Phi$ is a distribution on $(H,L)$-colorings of $G$ such that for each $c \in V(H)$, a randomly sampled $\phi \in \Phi$
    satisfies $c \in \phi$ with some rational probability $p_c$.
    Then, there exists
    a
    multiset $\Psi$ of $(H,L)$-colorings $\psi$ of $G$
    such that for each $v \in V(G)$ and $c \in L(v)$, if $\psi \in \Psi$ is sampled
    uniformly at random, then $\Pr(\psi(v) = c) = p_c$.
\end{lemma}
\begin{proof}
    Let $\{\phi_1, \dots, \phi_k\}$ be the set of all $(H,L)$-colorings of $G$. 
    We define a matrix $A$ with rows indexed by $V(H)$, columns indexed by $\phi_1, \dots, \phi_k$, and with entries $A_{c,i}$ where 
    for each color $c \in V(H)$, $A_{c,i} = 1$ if $c \in \phi_i$, and $A_{c,i} = 0$ otherwise.
    Let $\Phi'$ be a distribution 
    on $(H,L)$-colorings of $G$ that realizes each probability $p_c$.
    For each $i \in \{1, \dots, k\}$, let $x_i$ be the probability that $\phi_i$ is chosen when an $(H,L)$-coloring is sampled from $\Phi'$.
    For each $c \in V(H)$, $\sum_{i=1}^k A_{c,i} x_i = p_c$.
    We define column vectors 
    $\mb p  = (p_c)_{c \in V(H)}$ (of dimension $|V(H)|\times 1$) and
    $\mb x = [x_1, \dots, x_k]^T$, and we observe that
    $A \mb x = \mb p$
    and that $\mb x \geq \mb 0$.
    As $A$ and $\mb p$ are rational,
    a well-known result from the theory of polytopes (see Kaibel~\cite{book_2011Kaibel})
    states that there is a rational vector $\mb q = [q_1, \dots, q_k]^T$ for which $A \mb q = \mb p$ and $\mb q \geq \mb 0$.
    %Therefore, $\mb p$ is a vector in the column space of $A$.

    %Now, let $R$ be a $|V(H)| \times |V(H)|$ matrix for which $RA$ is in reduced row echelon form. Since $R$ is constructed from the product of elementary row operation matrices applied to $A$, $R$ has rational entries. Furthermore, we observe that $RA \mb x = R \mb p = \mb b$ 
    %for some rational
    %vector $\mb b =  [b_1, \dots, b_{|V(H)|}]^T$.
    %We define a new rational 
    %vector $\mb q \in \mathbb Q^k$ as follows.    
    %For each row
    %$j$ of $RA$ with a leading $1$, we define $q_j = b_j$, and we let all other entries of $\mb q$ equal $0$. As $RA$ is in reduced row echelon form, it then follows that $RA\mb q = R \mb p$. 
    %As $R$ is invertible, we therefore see that $A \mb q = \mb p$.
    %{\PB I forgot to argue nonnegativity}
    
    Next, we define a distribution $\Phi$ 
    on $(H,L)$-colorings of $G$ 
    so that each $(H,L)$-coloring $\phi_i$ is sampled with probability $q_i$.
    Then, when $\phi \in \Phi$ 
    is sampled, it holds that
    for each $c \in V(H)$, 
    the probability that $c \in \phi$ is equal to $\sum_{i=1}^k  A_{c,i} q_i = p_c$.

    Finally, using the fact that each $q_i$ is rational,
    let $N$ be an integer such that each value $N q_i$ is an integer.
    Let $\Psi$ be a multiset of size $N$ 
    in which each $\phi_i$ appears with multiplicity $Nq_i$.
    Then,
    the distribution on $(H,L)$-colorings of $G$ obtained by sampling from $\Psi$ uniformly at random is equivalent to the distribution $\Phi$, and thus $\Psi$ satisfies the condition in the lemma.
\end{proof}

\section{Properties of a counterexample}

% \iffalse
% {\PB When will we start writing $\epsilon = \frac 15$?
% I am not sure if we should put general $\epsilon$ in the theorem, because we would have to rewrite proofs for general $\epsilon$.
% }
% {\IC maybe have the ones that can be stated with $\epsilon$ without much effort - but I also don't mind changing everything to $\frac16$ if that makes life easier.}
% {\PB I think changing everything to $\frac 15$ makes our life easier; I'm not sure how much value there is in having this framework for general $\epsilon$.}
% \fi

We now begin our proof of \Cref{thethm}. 
To this end, let $G$ be a counterexample to \Cref{thethm} for which   $|V(G)|+|E(G)|$ is minimum, and  among all potential functions $\rho$ on $V(G)$ for which $G$ is $(\rho, \frac 15)$-critical, choose $\rho$ that maximizes $\rho_G(V(G))$.
%Note that $G \not \in \mathcal I$, and $\rho_G(V(G)) > 0$.
 %We write $n=|V(G)|$. 
 Recall that $\rho$ gives a partition $\Pi$ of $V(G)$ with parts $\Pi_3, \Pi_4, \Pi_6$,
 where $\rho(v) = i$ if and only if $v \in \Pi_i$.
For each integer $j$, we also define $\Pi_{\geq j}$ to be the union of all parts $\Pi_i$ for which $i \geq j$. We define $\Pi_{\leq j}$ similarly.
We observe that $G$ is $(\rho,\frac 15)$-critical,
$G$ is not exceptional,  and 
$\rho_{G}(V(G)) \geq 1$.

As $G$ is $(\rho,\frac 15)$-critical, $G$ is connected.
As $(G,\rho)$ is not $\frac 15$-flexible,
there exists a DP $3$-cover $(H,L)$ of $G$ along with a $\frac 15$-admissible $h_{\rho}$-list distribution $\mathcal D$ for which $(G,\rho,H,L,\mathcal D)$ is not $\frac 15$-flexible.
 Without loss of generality, for each  $uv \in E(G)$, $H[L(u) \cup L(v)]$ is the  union of $|E_G(u,v)|$ perfect matchings.
We fix $H$, $L$, and $\mathcal D$, and we use them for the rest of the proof.

\subsection{Subgraphs of $G$}
We observe some properties of subgraphs of $G$.

\begin{observation}
\label{obs:no-exception}
    $(G,\rho)$ has no exceptional subgraph.
\end{observation}
\begin{proof}
    Suppose that $G$ has a subgraph $G'$
    for which either $G' \in \mathcal I$ or $G'$ is a $C_2$ for which $\rho_G(V(G')) = 2$.
    If $G' = G$, then $G$ is not a counterexample to \Cref{thethm}.
    If $G'$ is a proper subgraph of $G$, then as $(G',\rho)$ is not $\frac 15$-flexible by Observations \ref{obs:C2} and \ref{obs:Ilkyoo-inflexible},
    $G'$ has a $(\rho,\frac 15)$-critical subgraph, contradicting the assumption that $G$ is $(\rho,\frac 15)$-critical.
\end{proof}

The following lemma shows that the random lists of size $2$ assigned to vertices of $\Pi_3$ can
often help us reduce the flexibility problem on $G$ to a problem on a subgraph.
%The following lemma shows that the probabilistically reduced lists of the vertices of $\Pi_3$
%can help us to connect the flexibility problem for $G$ to flexibility problems for subgraphs of $G$.
%{\IC unclear sentence..? }
%{\PB Haha sorry. I want to say that the advantage of having vertices in $\Pi_3$ is that we can try to do reductions, in the way that you can do a reduction in list coloring by coloring part of the graph and reducing the list sizes of the neighbors. Do you have a good way of expressing this?}
%{\IC Nope, :P}
%{\PB How is this?}

\begin{lemma}
    \label{lem:reducible}    
    Let $U \subseteq V(G)$ be nonempty. Suppose that no $u \in U \cap \Pi_{\leq 4}$ has a neighbor in $V(G) \setminus U$ and that each $u \in U \cap \Pi_{6}$ satisfies $|E_G(u,V(G) \setminus U)| \leq 1$.
    Define $\rho':U \rightarrow \{3,4,6\}$ so that
    $\rho'(u) = 3$ if $u$ has a neighbor in $V(G) \setminus U$ and $\rho'(u) = \rho(u)$ otherwise.
%    Define $\ell:U \rightarrow \{2,3\}$
 %   so that 
  % 
   % Define $\rho':U \rightarrow \{3,4,6\}$ so that $\rho'$ is a potential function of $(G[U],\ell)$.
    %{\PB This will need to be rephrased based on how we do the partition notation. 
    %I wonder if we can have things so that we can say ``so that $\ell = h_{\rho'}$."}.
    Then, there is a $\frac 15$-admissible $h_{\rho'}$-list distribution
    $\mathcal D'$ such that $(G[U], \rho', H,L,\mathcal D')$ is not $\frac 15$-flexible.
\end{lemma}
\begin{proof}
 %\[h_{\rho'} = h_{\rho}(u) - |E_G(u,V(G) \setminus  U)| \text{ for each $u \in U$.}\]
   Since $G$ is $(\rho,\frac 15)$-critical,
   %$(L,\cD)$-critical {\PB What is $\mathcal D$? Answer: After talking to Sasha, it seems that $\mathcal D$ is a distribution on lists $L_{\mathcal D}$.}, 
   $(G - U,\rho,H,L,\mathcal D)$ has a $\frac 15$-distribution $\Psi$. 

Let $W \subseteq U$ be the set of vertices $w$ for which  $\rho'(w)=3$. 
By the hypotheses of the lemma, each $w \in W$ either satisfies $\rho(w) = 3$,
or $\rho(w) = 6$ and $w$ has a unique neighbor $w' \in  V(G)-U$.
We define a distribution $\mathcal D'$ on the forbidden colors of vertices $w \in W$ using the following experiment. First,
we sample an outcome $L_{\mathcal D}$ of $\mathcal D$ and 
assign $L_{\mathcal D'}(v) = L_{\mathcal D}(v)$ for each $v \in W$ with $h_{\rho}(v) = 2$.
%Note that this step assigns a forbidden color to each $w \in W$ with $h(w) = 2$. 
Next, we choose a random $(H,L_{\mathcal D})$-coloring $\psi$ of $G-U$ according to the distribution $\Psi$.
Finally, for each $w \in W$ for which $\rho(w) = 6$,
%{\IC \sout{$w$ has a unique  neighbor $w' \in  V(G)-U$.
%We}} 
let $L_{\mathcal D'}(w) = L(w) \setminus N_H(\psi)$.
As $L(w)$ and $L(w')$ are joined by a perfect matching in $H$,
a unique color $c \in L(w)$ has a neighbor in $\psi$; therefore, $|L_{\mathcal D'}(w)| = 2$.
For each $u \in U \setminus W$, we let $L_{\mathcal D'}(u) = L(u)$.
Since $\Psi$ is a $\frac 15$-distribution, this distribution $\mathcal D'$ is $\frac 15$-admissible.
%{\PB What is written above is an attempt to rewrite the paragraph below more precisely.}

%Let $u\in N(S)\cap V(G')$.
% By~\eqref{ee1} and~\eqref{ee1'}, $\rho_{h,\omega}(u)=6$, and $u$ has exactly one neighbor in $S$, say $u'$. %We decrease the potential of $u$ to $3$ {\PB 
% We observe that $h'(u) = 2$, and 
 %let the restraining distribution be the distribution of colors on $u'$  {\AK maybe $\Phi$?}.
%Do this for every $u\in N(S)\cap  V(G')$ and  let $\cD'$ be the resulting restraining distribution for $(G',L',H')$.

We claim that 
    $(G[U],\rho',H,L,\mathcal D')$ has no $\frac 15$-distribution. Indeed, suppose that $(G[U],\rho',H,L,\mathcal D')$ has an $\frac 15$-distribution $\Phi'$. We define a $\frac 15$-distribution $\Phi$ for $(G,\rho,H,L,\mathcal D)$ as follows.

    For each outcome $L_{\mathcal D}$ of $\mathcal D$, we construct a random $(H,L_{\mathcal D})$-coloring $\phi$ of $G$ as follows.
    First, we sample an $(H,L_{\mathcal D})$-coloring $\psi$ of $G-U$
    using $\Psi(L_{\mathcal D})$.
    For each $v \in V(G-U)$, we let $\phi(v) = \psi(v)$.
    Then, 
    for each $u \in U$, we let $L'(u) = L_{\mathcal D}(u) \setminus N_H(\psi)$.
    By definition of $\mathcal D'$, 
    $L'$ has the same distribution as $L_{\mathcal D'}$.
    %we let $L'$ be defined from $L_{\mathcal D}$ and $\psi$ as described above.
    %{\PB Please see the sentence I commented it out. Is it clear if we just delete this sentence and leave it like this?}
    %By our definition of $\mathcal D'$,
    %sampling an outcome $L_{\mathcal D'}$ in this way is equivalent to sampling an outcome of $\mathcal D'$
    %using the probability distribution of $\mathcal D'$
    %directly {\PB I don't like the way I wrote this}.
    Finally,
    using the fact that $L'$ is in the outcome space of $L_{\mathcal D'}$,
    we sample an $(H,L')$-coloring $\phi'$ of 
    $G[U]$
    using $\Phi'(L')$---that is, 
    using the conditional distribution $\Phi'(L_{\mathcal D'})$
    arising from the outcome $L_{\mathcal D'} = L'$.
    We let $\phi(v) = \phi'(v)$ for each $v \in U$. 
    %By our definition of $\mathcal D'$, the coloring $\phi$ produced is an
    %$(H,L_{\mathcal D})$-coloring of $G$.
    Hence,
    this process gives us a coloring distribution $\Phi$ for $(G,\rho,H,L,\mathcal D)$.

    We claim that $\Phi$ is an $\frac 15$-distribution for $(G,\rho,H,L,\mathcal D)$. For $v \in V(G) \setminus U$,
    if $\rho(v) = 4$ and $c \in L(v) \setminus \{\omega(v)\}$,
    then $\Pr(\phi(v) = c) = \Pr(\psi(v) = c) = \frac 3{10}$.
    %and $c \in L(v)$, $\Pr(\phi(v) = c) = \Pr(\psi(v) = c) =$ when $\rho(v) = 4$, and 
        Otherwise, for each $c \in L(v)$,
     $\Pr(\phi(v) = c) = \Pr(\psi(v) = c) \geq \frac 15$.
     For $v \in U$,  $\rho(v) = 4$ if and only if $\rho'(v) = 4$; therefore,
     if $\rho(v) = 4$ and $c \in L(v) \setminus \{\omega(v)\}$, $\Pr(\phi(v) = c) = \Pr(\phi'(v) = c) = \frac 3{10}$. Otherwise, for each $c \in L(v)$, 
     $\Pr(\phi(v) = c) = \Pr(\phi'(v) = c) \geq \frac 15$. Therefore, $\Phi$ is an $\frac 15$-distribution for $(G,\rho,H,L,\mathcal D)$, contradicting the choice of $(G,\rho,H,L,\mathcal D)$.
\end{proof}

\begin{lemma} 
\label{lem:reducible-tree}
Let $U \subseteq V(G)$ be nonempty.
Suppose that no $u \in U \cap \Pi_{\leq 4}$ has a neighbor in $V(G) \setminus U$  and that each $u \in U \cap \Pi_6$
satisfies $|E_G(u,V(G) \setminus U)| \leq 1$.
Then, $G[U]$ is not a tree.
 %For each $u \in U$, define $\ell(u) = h(u) - |E_G(u,V(G) \setminus U)|$. Suppose that $\ell(u) \geq 2$ for each $u \in U$
  %  and $\ell(u) = 3$ for each $u \in U \cap \Pi_4$.
   % Then, 
    % $G[U]$ is not a tree.
\end{lemma}
\begin{proof}
    Suppose that $G[U]$ is a tree.
    We show that \Cref{lem:reducible}
    is violated. 
    Define $\rho':U \rightarrow \{3,4,6\}$ so that $\rho'(u) = 3$ if $u$ has a neighbor in $V(G) \setminus U$ and $\rho'(u) = \rho(u)$ otherwise.
    Using \Cref{lem:reducible},
    let $\mathcal D'$ be a $\frac 15$-admissible $h_{\rho'}$-list distribution such that $(G[U],\rho',H,L,\mathcal D')$ is not $\frac 15$-flexible.

    We consider the following coloring distribution $\Phi$ for $(G[U],\rho',H,L,\mathcal D')$.
    Consider an outcome $L_{\D'}$ of $\D'$.
    For each $u \in U$
    with $h_{\rho'}(u) = 2$, we write $L'(u) = L_{\mathcal D'}(u)$. 
    For each $u \in U$ with $h_{\rho'}(u) = 3$,
    we obtain a list $L'(u)$ of size $2$
    from $L_{\mathcal D'}(u) = L(u)$ by deleting $\omega(u)$ with probability $\frac 15$ and deleting each other color of $L(u)$ with probability $\frac 25$.
    Then, we
    use \Cref{lem:tree-packing} to construct two disjoint $(H,L')$-colorings of $G[U]$ and choose one of them uniformly at random.
    Thus, we obtain a coloring distribution $\Phi$
    for $(G[U],\rho,H,L,\mathcal D')$.
    When $\phi \in \Phi$ is sampled, we see that
    for each $u \in U$ with $h_{\rho'}(u) = 3$,
    $\Pr(\phi(u) = \omega(u)) = \frac 45 \cdot \frac 12 = \frac 25$, and $\Pr(\phi(u) = c) = \frac 35 \cdot \frac 12 = \frac 3{10}$ for each $c \in  L(u) \setminus \{\omega(u)\}$; in particular, this equality holds whenever $\rho(u) = 4$.
    For each $u \in U$ with $h_{\rho'}(u) = 2$ and color $c \in L(u)$,
    since $\mathcal D'$ is $\frac 15$-admissible,
    we have $c \in L'(u)$ with probability at least $\frac 25$;
    therefore, $\Pr(\phi(u) = c) \geq \frac 25 \cdot \frac{1}{2} = \frac 15$.
    Therefore, $\Phi$ is a $\frac 15$-distribution for $(G[U],\rho',H,L,\mathcal D')$, a contradiction.
\end{proof}

\subsection{A gap lemma}
The goal of this subsection is to 
prove a ``gap lemma", which
shows that every proper vertex subset $S \subsetneq V(G)$
has a lower bound on its potential in terms of the number of edges joining $S$ to its complement. First, we need the following lemma.

%{\PB Unfortunately,  this lemma is annoying to write for general $\epsilon$ and is needed for the gap lemma.}
\begin{lemma}
\label{lem:pendent-heavy}
    If $u \in \Pi_{\geq 4}$, then $d(u) \geq 2$. 
\end{lemma}
\begin{proof}
   % Without loss of generality, assume that $\omega(v) = 1_v$.
   %{\PB This still works for $\epsilon = 1/5$}
    If $d(u) = 0$, then as $G$ is connected, $V(G) = \{u\}$.
    Then, the distribution $\Phi$ that assigns a color to $u$ uniformly at random
    shows that $(G, \rho, H,L,\mathcal D)$ is $\frac 15$-flexible, a contradiction.

    If $d(u) = 1$, then assume without loss of generality that $\omega(u) = 1_u$, and let $v$ be the unique neighbor of $u$. 
    Without loss of generality, assume that $H$ contains the edges $i_u i_v$ for each $i \in \{1,2,3\}$. 
    As $G$ is $(\rho,\frac 15)$-critical, $(G-u, \rho, H,L,\mathcal D)$ has
    a $\frac 15$-distribution $\Phi'$ on $(H,L)$-colorings of $G-u$.

    We construct a $\frac 15$-distribution $\Phi$ for $(G,\rho,H,L,\mathcal D)$ as follows.
    For each outcome $L_{\mathcal D}$ of $\mathcal D$, we sample an $(H,L_\D)$-coloring $\phi'$ of $G - u$ from $ \Phi'$.
    For each $i \in \{1,2,3\}$, let $p_i = \Pr(\phi'(v) = i_v)$,
    and let $q_i = \frac 52 p_i - \frac 12$.
    Observe that $q_i \in [0,1]$ and $q_1 + q_2 + q_3 = 1$.
    We construct an $(H,L)$-coloring $\phi$ for $G$ by letting $\phi$ agree with $\phi'$ on $V(G) \setminus \{u\}$,
    and then
    whenever $\phi'(v) = i_v$,
    we let $\phi(u) = j_u$ with conditional probability equal to entry $A_{ji}$ in the following matrix $A$:
    \[
    A = 
        \frac 1{10}
        \begin{bmatrix} 
        0 & \frac{2q_1 + 3q_2}{p_2} & \frac{1 + q_1 + 3q_3 }{p_3} \\
        \frac{1 + 2q_1 + q_2}{p_1} & 0 & \frac{q_2 + 2q_3}{p_3} \\
        \frac{3q_1 + q_3}{p_1} & \frac{ 3q_2 + 2q_3}{p_2} & 0
        \end{bmatrix}.
    \]

    As the diagonal values of $A$ are $0$, $\phi(u)$ and $\phi(v)$ are not adjacent in $H$.
    Furthermore, as the entries of $A$ are nonnegative and each column of $A$ sums to $1$,
    the columns of $A$ give valid conditional probability distributions.
    Letting $\mb p = [p_1, p_2,p_3]^T$, we see that
    $A \mb p = [\frac 25, \frac 3{10}, \frac 3{10}] ^T$
    gives the probability vector for the assignments of colors $(1_u,2_u,3_u)$ at $u$.
    In particular,
    %$\Pr(\phi(u) = \omega(u) ) = \frac 25$, and
    for each $c \in L(u) \setminus \{\omega(u)\}$, $\Pr(\phi(u) = c) = \frac 3{10}$.
    Therefore, $\Phi$ is a $\frac 15$-distribution for $(G, \rho, H,L, \mathcal D)$, contradicting the choice of $(G,\rho,H,L,\mathcal D)$.
\end{proof}

%Call a vertex $v$ for which $h(v) = 3$    {\em full} and 
%a vertex $v$ for which $h(v) = 2$ \emph{weak}.
%{\PB Let's not do this.}

The following ``gap lemma" gives a lower bound for the potentials of proper vertex subsets of $G$.
\begin{lemma}
\label{lem:j+2}
  For every subset $S \subseteq V(G)$, $\rho_G(S)\geq 1+ |E(S,\overline{S})|$.
\end{lemma}
\begin{proof} Suppose the lemma fails.
Choose a largest set 
$S\subseteq V(G)$ for which $\rho_{G}(S)\leq |E(S,\overline{S})|=:j$. 
If $S = V(G)$, then $\rho_G(V(G)) = |E(V(G), \overline{V(G)})| = 0$, contradicting the choice of $G$.
%If $|S|=|V(G)|-1$ and $V(G)-S=\{u\}$, then
%$\rho_G(S) \leq j=d(u)$, and 
%\[\rho_{G}(V(G))=\rho_{G}(S)+\rho(u)-4d(u)\leq \rho(u)-3d(u).\] If $d(u)\geq 2$, then $\rho_{G}(V(G)) \leq 6-3(2)=0$, and $(G,\rho)$ is not a counterexample. As $G$ is connected, it follows that $d(u)=1$.
%Then, $\rho(u)=3$ by \Cref{lem:pendent-heavy}, and we have $\rho_G(V(G))\leq 3-3(1)=0$. In both cases, this is a contradiction. Thus, $|S|\leq |V(G)|-2$. 
Therefore, $|S| \leq |V(G)| - 1$.

For each $v\in V(G)-S$,
\begin{equation}
   \label{rho S+v upper}
\rho_{G}(S+v)=\rho_{G}(S)+\rho(v)-4|E(v,S)|\leq j+\rho(v)-4|E(v,S)|, 
\end{equation}
  and by the maximality of $S$, %and the fact that $|S|\leq |V(G)|-2$,
\begin{equation}\label{rho S+v lower}
\rho_G(S+v)\geq 1+|E(S+v,\overline{S+v})|=1+j+d(v)-2|E(v,S)|.
%\geq 1+j-|E(v,S)|.
\end{equation}
 By \eqref{rho S+v upper}~and \eqref{rho S+v lower}, we see that
$1+j+d(v)-2|E(v,S)|\leq j+\rho(v)-4|E(v,S)|,$ which yields 
\begin{equation}
\label{eqn:2Ed}
 \rho(v)  \geq d(v)+2|E(v,S)| + 1.
\end{equation}

% \iffalse
% {\PB In order to account for heavy vertices, we need to make a slightly more careful argument.
% By (\ref{eqn:2Ed}), we see that if $v \in N(S) \setminus S$ is heavy, then $|E(v,S)| = 1$. Furthermore, by \Cref{lem:pendent-heavy}, $d_G(v) \geq 2$.
% Therefore, using the inequality 
% \[\rho_{G}(S+v) \geq 1 + j + d(v) - 2|E(v,S)|\]
% from (\ref{rho S+v lower}), we actually see that 
% \[\rho(S+v) \geq 1 + j. \]
% Then, using (\ref{rho S+v upper}),
% we see that 
% \[1 + j \leq \rho(S+v) \leq j + \rho(v) - 4|E(v,S)| = j,\]
% giving us a contradiction.
% Therefore, no $v \in N(S) \setminus S$ is heavy.
% }
% \fi

Recall that by  \Cref{lem:pendent-heavy}, $d_G(v) \geq 2$ for each $v\in \Pi_{\geq 4}$.
Therefore, as $\rho(v) \leq 6$ for each $v \in V(G)$,
%since $\rho_{h,\omega}(v) \leq 6$ for every $v\in V(G)$, $\rho_{h,\omega}(v) \leq 4$ for every heavy $v\in V(G)$,
%and $\rho_{h,\omega}(v) =3$ for every $v\in V(G)$ satisfying $h(v) = 3$,
\eqref{eqn:2Ed} implies that
%We observe that 
\begin{equation}\label{ee1}
\mbox{\em for each $v\in V(G)-S$, $|E(v,S)|\leq 1$}
\end{equation}
 %Indeed, if $v\in V(G)-S$ and $|E(v,S)|\geq 2$,
 % , then $|E(S+v,V(G)-S-v)|=j+d(v)-2|E(v,S)|$ and
% \begin{equation}\label{ee2}
% \rho(S+v)=\rho(S)+\rho(v)-4|E(v,S)|\leq j+6-4|E(v,S)|.
% \end{equation}
%then (\ref{eqn:2Ed}) implies that
%$\rho_{h,\omega}(v) \geq 7$, a contradiction.
%Similarly, we observe that
and
\begin{equation}\label{ee1'}
\mbox{\em for each $v\in (V(G)-S) \cap \Pi_{\leq 4}$, %with $\rho_h(v)=3$
$E(v,S)=\emptyset$.
}
\end{equation}

Now, define $G' = G-S$.
Furthermore, define $\rho':V(G') \rightarrow \{3,4,6\}$ so that $\rho'(v) = 3$ whenever $v$ has a neighbor in $S$ and $\rho'(v) = \rho(v)$ otherwise.
%Note that by \eqref{ee1} and \eqref{ee1'}, $\ell(v)\in\{2,3\}$ for all $v\in V(G')$.
%Define $\rho':V(G') \rightarrow \{3,4,6\}$
%so that {\PB $\ell = h_{\rho'}$?} $\rho'$ is a potential function for $(G',\ell)$.
Note that by \eqref{ee1} and \eqref{ee1'}, 
the conditions of \Cref{lem:reducible}
hold for $U = V(G')$.
Thus, 
there is a $\frac 15$-admissible $h_{\rho'}$-list distribution $\D'$  such that $(G',\rho',H,L,\mathcal D')$ has no $\frac 15$-distribution.
Therefore,
 $G'$ has an $(\rho',\frac 15)$-critical subgraph $G''$.  
 Since $G''$ is a subgraph of $G$, and as $(G,\rho)$ has no exceptional 
 subgraph by \Cref{obs:no-exception},
 it follows that $(G'',\rho)$ is not exceptional.
 We write $W = V(G'')$.  
By definition,
$\rho'(w) = 3$ for each $w \in W$ satisfying $\rho(w) \neq \rho'(w)$.
Thus, as $(G,\rho)$ is minimal and has no exceptional subgraph, $\rho'_{G}(W) \leq 0$.
 %By the minimality of $G$,
    %$\rho'_{G'}(U) \leq 0$.
    Since $G$ is 
    $(\rho,\frac 15)$-critical, and 
    since $\rho'(w) = \rho(w)$
    for each $w \in V(G') \setminus N(S)$,
    there exists a vertex 
    $w\in N(S)\cap W$, for which $\rho'(w) \neq \rho(w)$.
     Therefore, $|E_G(S, W )|=:t \geq 1$.
    %in particular, $\rho'(u) = 3$, and $\rho(u) = 6$.
    %this $u$ is weak in $(G'',L')$,
% $G''\notin \cI$ {\PB If we want to make this statement, we need to say that the critical guys in $\mathcal I$ are only considered with full lists. Alternatively, we could have a lemma that if $G \in \mathcal I$ is a critical graph with a weak vertex, then $\rho(G) \leq 0$ and hence $G$ is not a counterexample.}

Note also that for each vertex $v \in N(S) \cap W$, we have
$\rho(v) = 6$ and $\rho'(v) = 3$
%$h(v) = 3$ and $\ell(v) = 2$. 
%Since each $v \in N(S) \cap V(G'')$ is in $\Pi_6$ 
by (\ref{ee1'}).
%this implies .
Let $S'=S\cup W$. 
Note that $|N(S)\cap W|\leq t$.
Thus,
\begin{equation}\label{ee21}
\rho_{G}(S')=\rho_{G}(S)+\rho_{G}(W)-4t \leq j+(\rho'_{G'}(W)+3t )-4t \leq j-t.
\end{equation}
If $S'=V(G)$, then $t=j = |E(S, \overline S)|$, and by~\eqref{ee21}, $\rho_G(V(G))\leq 0$,
 a contradiction to the choice of $G$. 
 Otherwise, $|E_G(S',\overline{S'})|\geq j-t$, so by~\eqref{ee21}, $S'$ is a larger choice than $S$, a contradiction to the maximality of $S$.
\end{proof}

\Cref{lem:j+2} immediately implies an upper bound on degrees of vertices in $G$.
\begin{lemma}
\label{lem:weak-deg} 
For every $v \in V(G)$, $d(v)\leq \rho(v)-1$.
In particular, if $\rho(v) = 3$, then $d(v) \leq 2$;
if $\rho(v) = 4$, then $d(v) \leq 3$; and
 $d(v) \leq 5$ for every $v \in V(G)$.
\end{lemma}
%\begin{proof}    By Lemma \ref{lem:j+2}, for every $v \in V(G)$,  $6\geq \rho_h(v)\geq d(v)+1$, and if $v$ is weak, then 
   % $3 = \rho_h(v) \geq d(v) + 1$.
%\end{proof}

%If true, 
%The  lemma above implies that $\Delta(G)\leq 5$ and that for each $v\in V(G)$ with $\rho_{h,\omega}(v)=3$, $d(v)\leq 2$. {\AK is this the next lemma?}
%It would also imply that for each multiple edge $uv$, $d(u)+d(v)\leq 7$. And I think we had a proof that this is impossible.
%Thus if true, the lemma would imply that $G$ has no multiple edges.

\subsection{A reduction using fractional $(H,L)$-packings}

\begin{lemma}
\label{lem:gen-diamond}
    $G$ has no vertex set $U \subseteq \Pi_6$ 
    of size at least $2$
    satisfying the following:
    \begin{itemize}
        \item The set $U' \subseteq U$ consisting of vertices $u \in U$ with a neighbor in $V(G) \setminus U$
        has size at most $2$; %{\PB In fact, if we do the stronger gap lemma mentioned above, then this condition is redundant}
        \item $\rho_G(U) \leq 4$; %, or 
        %$\rho_G(U) = 6$, 
        %$U \subseteq \Pi_6$, and 
        %$|N(U') \setminus U| = 1$; 
        \item $G[U]$ has a fractional $(H,L)$-packing.
    \end{itemize}
\end{lemma}
\begin{proof}
    %{\PB For $\epsilon = 1/5$,
    %we would only be able to use this when $U \subseteq \Pi_6$, or $U$ has a single heavy vertex.}
    Suppose that the lemma is false, and let $U \subseteq V(G)$ be a violating vertex subset.
    By the lemma's assumption, $|U| \geq 2$. 
    Since $G[U]$ has a fractional $(H, L)$-packing, by
     \Cref{lem:rational},
    we find a multiset $\Psi$
    of $(H,L)$-colorings of $G[U]$
    such that for each $u \in U$ and $c \in L(u)$,
    $c$ appears in exactly $|\Psi|/3$ colorings in $\Psi$.

    Let $U' \subseteq U$ be the set consisting of all $u \in U$ with a neighbor in $V(G) \setminus U$.
    If $U' = \emptyset$,
    then as $G$ is connected, 
    $U = V(G)$.
    Since $V(G) \subseteq \Pi_{6}$ and $G$ has a fractional $(H,L)$-packing,
    there is a $\frac 13$-distribution for $(G,\rho,H,L,\mathcal D)$, a contradiction.
    Otherwise, $|U'| \in \{1,2\}$.
    We write $U' = \{u_1, u_2\}$, where $u_1 = u_2$ whenever $|U'| = 1$.
    We construct an auxiliary bipartite graph $B$ with partite sets $L(u_1)$ and $L(u_2)$. For each $\psi \in \Psi$, we add the edge $e_{\psi} = \psi(u_1) \psi(u_2)$ to $B$, and we observe that $B$ is $\frac 13 |\Psi|$-regular.
    Using Hall's theorem, we decompose $E(B)$ into a set
    $\mathcal M = \{M_1, \dots, M_{|\Psi|/3} \}$ of
    perfect matchings. For each $i_{u_1} \in L(u_1)$ and $M \in \mathcal M$, we write $M(i_{u_1})$ for the unique neighbor of $i_{u_1}$ in $M$.

    Now, we construct a graph $G'$ with a  $3$-cover $(H',L')$ as follows. Choose $M \in \mathcal M$ uniformly at random. To obtain $G'$, delete $U \setminus \{u_1,u_2\}$ from $G$, and
    identify
    $\{u_1, u_2\}$ into a single vertex $u^*$.
    To obtain $H'$, delete $L(u)$ from $H$ for each $u \in U \setminus \{u_1, u_2\}$, and identify each color pair $(i_{u_1}, M(i_{u_1}))$ to a single color $i_{u^*}$.
    We write $L'(u^*) = \{1_{u^*}, 2_{u^*}, 3_{u^*}\}$, and we write $L'(v) = L(v)$ for each $v \in V(G') \setminus \{u^*\}$.
    We extend $\rho$ to let $\rho(u^*) = 6$.
    %If $U$ has a vertex of $\Pi_4$, then we assign $\rho'(u^*) = 4$; otherwise, we assign $\rho'(u^*) = 6$.
    %We let $\rho'(v) = \rho(v)$ for each $v \in V(G') \setminus \{u^*\}$, and we let $h':V(G') \rightarrow \{2,3\}$ be defined so that $\rho'$ is a potential function for $(G',h')$; namely, $h'(u^*)=3$.
    We let $\mathcal D'$ be an $h'_{\rho}$-list distribution on $(G',\rho, H',L')$ defined to agree with $\mathcal D$ on $V(G') \setminus \{u^*\}$; note that a color in $L'(u^*)$ is never forbidden by $L_{\mathcal D'}$.

    \begin{claim} 
    \label{claim:G'-no-good}
    There exists a $3$-cover $(H',L')$ for which
    $(G',\rho,H',L',\mathcal D')$ is not $\frac 15$-flexible.
    \end{claim}
    \begin{proof}
        Suppose that $(G',\rho,H',L',\mathcal D')$
        has a  $\frac 15$-distribution $\Phi'$ for every $3$-cover $(H',L')$ constructed above.
         Recall that $(G,\rho,H,L,\mathcal D)$ has no $\frac 15$-distribution. 
        We obtain a contradiction by constructing a
        $\frac 15$-distribution $\Phi$ for $(G,\rho,H,L ,\mathcal D)$ as follows.
        
        For each outcome $L_{\mathcal D}$ of $\mathcal D$, let $L'_{\mathcal D'}$ be the corresponding outcome of $\mathcal D'$.
        Let $\phi'$ be an $(H',L'_{\mathcal D'})$-coloring of $G'$ sampled from $\Phi'(L'_{\mathcal D'})$.
        We let $\phi(v) = \phi'(v)$ for each $v \in V(G') \setminus \{u^*\} = V(G) \setminus U$.

        For each $u \in U$, we assign $\phi(u)$ as follows. First, we observe that the color $\phi'(u^*) = i_{u^*}$ assigned to $u^*$ is associated with a unique edge $e = i_{u_1} M(i_{u_1})$ for some $M \in
        \mathcal M$. Furthermore, the edge $e$ is associated with a unique coloring $\psi_e$ of $G[U]$. We let $\phi$ agree with $\psi_e$ on $U$.
        This random procedure for constructing $\phi$ gives us a  distribution $\Phi$ on $(G,\rho,H,L,\mathcal D)$.

        We first argue that the coloring $\phi$ 
        that we construct from $\phi'$
        is an $(H,L_{\mathcal D})$-coloring of $G$.
        It is easy to check that $\phi$ contains exactly one vertex from each set $L_{\mathcal D}(v)$ for $v \in V(G) \setminus U$. 
        Since $h_{\rho}(u) = 3$ 
        for each $u \in U$, $L_{\mathcal D}(u) = L(u)$, 
        so $\phi$ thus contains exactly one color from each $L_{\mathcal D}(u)$ for $u \in U$.
        Additionally,
        as $\phi$ agrees with $\phi'$ on $V(G) \setminus U$ and $\phi$ agrees with $\psi_e$ on $U$, 
        $\phi$ contains no adjacent pair from $H[L(V(G) \setminus U)]$ and no adjacent pair from $H[L(U)]$.

        Now, suppose that $u \in U$ and $v \in V(G) \setminus U$, and that $uv \in E(G)$. 
        As no vertex of $U \setminus \{u_1,u_2\}$ has a neighbor in $V(G) \setminus U$,
        $u \in \{u_1,u_2\}$.
        Since $\phi'$ assigns the color $i_{u^*}$ to $u^*$, and since $i_{u^*}$ is obtained as contraction of $i_{u_1}$ and $M(i_{u_1})$, 
        it follows that $\phi(u) \in \{i_{u_1}, M(i_{u_1})\}$.
        Furthermore, as $\phi'(v) i_{u^*} \not \in E(H')$, it follows from the construction of $H'$ that $\phi'(v)$ has no neighbor in $\{ i_{u_1}, M(i_{u_1})\}$ in $H$. As $\phi(v) = \phi'(v)$, it thus follows that $\phi(u) \phi(v) \not \in E(H)$.

        Next, we argue that $\Phi$ is a $\frac 15$-distribution for $(G,\rho,H,L,\mathcal D)$. For each $v \in V(G) \setminus U$ and color $c \in L(v)$, since $\Phi'$ is a $\frac 15$-distribution, 
        \[\Pr(\phi(v) = c)= \Pr(\phi'(v) = c) \geq \frac 15.\]
        Furthermore, if $\rho(v) = 4$, then
        %$\rho'(v) = 4$, so
        %\[\Pr(\phi(v) = c)= \Pr(\phi'(v) = c) \geq 2 \epsilon.\]
        \[ \Pr(\phi(v) = c) = \frac{3}{10} \text{ for each $c \in L(v) \setminus \{\omega(v)\} $}.\]

        Next, suppose that $u \in U$, and consider a color $c \in L(u)$.
        We recall that $\phi_{\vert U}$ is chosen to agree with a random $(H,L)$-coloring $\psi_e$ of $G[U]$
        associated with an edge $e \in E(B)$, where $B$ is an auxiliary bipartite graph joining $L(u_1)$ and $L(u_2)$.
        %, which is taken from a decomposition of $E(B)$. 
        Thus, $\phi(u) = c$ if and only if  $\psi_e(u) = c$.

        Given an edge $f \in E(B)$, we estimate the probability that $\psi_f$ is the coloring used to define $\phi$ on $U$. 
        Observe that $\psi_f$ is used to define $\phi$ on $U$ if and only if the following two independent events occur:
        \begin{itemize}
            \item The randomly chosen matching $M \in \mathcal M$ contains the edge $f$.
            \item There exists $i \in \{1,2,3\}$ such that 
            $\phi'(u^*) = i_{u^*}$ and $\psi_f(u_1) = i_{u_1}$. 
        \end{itemize}
        The probability of the first event is $\frac 3{|\Psi|}$, and the probability of the second event
        is  
       % at least $2 \epsilon$
        %the interval $[2\epsilon, 1 - 4 \epsilon]$
%        when $\rho'(u^*) = 4$ and 
        at least $\frac 15$. %otherwise.
        %is in $[\epsilon,1-2\epsilon]$ otherwise.
        Therefore, the probability that $\psi_f$ is the coloring that defines $\phi$ on $U$ 
        is % in the interval 
        % $[\frac 3{|\Psi|} 2\epsilon, \frac 3{|\Psi|} (1 -4\epsilon)]$
        %at least $\frac 3{ |\Psi|}(2 \epsilon)$
        %when $\rho'(u^*) = 4$ and
        at least $\frac 3 { 5 |\Psi|} $.
        %is in 
        %$[\frac 3{|\Psi|} \epsilon, \frac 3{|\Psi|} (1-2\epsilon)]$
        %otherwise.
        
      Next, by construction of $\Psi$, exactly 
      $\frac 13 |\Psi|$
      colorings $\psi_f \in \Psi$ assign $\psi_f(u) = c$.
      Therefore, $\psi_e(u) = c$ for exactly $\frac 13 |\Psi|$
      outcomes of the random coloring $\psi_e$.
      Hence, using the probability estimates above for each fixed outcome $\psi_f$ of the random coloring $\psi_e$, we have that 
      %$\Pr(\psi_e(u) = c) \geq 2 \epsilon$ 
      %when $\rho'(u^*) = 4$, and 
      $\Pr(\psi_e(u) = c) \geq \frac 15$. % otherwise. 
        As $\phi(u) = \psi_e(u)$, %and $ \rho(u) = \rho(u^*) =6$, 
        we thus see that % $\Pr(\phi(u) = c) \geq 2 \epsilon$ whenever $\rho(u) = 4$, and 
        $\Pr(\phi(u) = c) \geq \frac 15$. %otherwise. 
        As $\rho(u) = 6$, $\Phi$ is a $\frac 15$-distribution for $(G,\rho,H,L,\mathcal D)$.
    \end{proof}
       As $(G',\rho,H',L',\mathcal D')$
is not $\frac 15$-flexible by \Cref{claim:G'-no-good},
    %and as $G$ is $(\rho,\epsilon)$-critical,
    the minimality of the counterexample $G$ tells us that 
    there is a set $S' \subseteq V(G')$ containing $u^*$
    for which
    $G'[S']$
    is an exceptional subgraph of potential $2$, or 
    $\rho_{G'}(S') \leq 0$; in both cases, $\rho_{G'}(S') \leq 2$.
    %If $G'[S]$ is a $C_2$ and $\rho_{G'}(S') = 2$,
    %this implies that $|U'| = 2$ and the two vertices of $U'$ have a common neighbor $w$ for which $\rho(w) = 4$.
    %Then, $\rho_G(U + w) \leq \rho_G(U) - 4 \leq 0$, contradicting Lemma \ref{lem:j+2}.
    %In the latter case, by \Cref{obs:Ilkyoo}, $\rho'(S') \leq 2$  when all vertices $v\in S'$ satisfy $\rho'(v)=6$, and $\rho'(S') \leq 0$ when 
    %$S'$ has a vertex $v$ satisfying $\rho'(v)=4$.
    % {\PB I don't think so; I think the potential could also be $\leq 0$ and the graph would not belong to $\mathcal I$}
    Let $S \subseteq V(G)$ be obtained from $S'$ by deleting $u^*$ and adding $U$.
    %If $\rho_G(U) \leq 4$, then 
    Then,
    $\rho_{G}(S) \leq \rho'(S') -\rho'(u^*)+\rho_{G}(U) \leq 2-6+4 = 0$, contradicting \Cref{lem:j+2}.
    %If $\rho_G(U) = 6$,
    %then $|N(U') \setminus U| \leq 1$, implying that $u^*$ has at most one neighbor in $G'$.
    %Therefore, $G'[S'] \not \in \mathcal I$,
    %and hence $\rho'_{G'}(S') \leq 0$.
    %Furthermore, if $\rho_G(U) = 6$, then $U \subseteq \Pi_6$, so $\rho'(u^*) = 6$.
    %Then, $\rho_G(S) \leq \rho'_{G'}(S') - \rho'(u^*) + \rho_G(U) = \rho'_{G'}(S') \leq 0$, again contradicting Lemma \ref{lem:j+2}.
    %{\PB This ending part was changed to allow $\rho(U) = 6$ under special conditions. Please check. Thank you!
    %Actually, we probably don't need this.
    %To be continued. {\IC okay}
    %}
\end{proof}

%{\PB 
%I wonder if this lemma should go earlier, since it doesn't use the gap lemma or any other lemmas at all.
%At the same time, we don't really need it until now. I think one of the big questions remaining is how to organize all these lemmas.
%}
%{\IC yeah that's a good question. I feel like a good order would be (a) some ad hoc lemmas (b) some lemmas about general $\epsilon$ (c) lemmas when $\epsilon=\frac16$ (d) reducibility lemmas - (with subsections like now) just a thought. I can try to organize after finishing reading - I only have a couple pages left.}
%{\PB Actually I was wrong, this uses the gap lemma at the end.}

\subsection{Pendent $C_2$ blocks of $G$}
The main goal of this subsection
is to show that $G$ has no pendent $C_2$ block.
In other words, we aim to show that no $2$-vertex of $G$ is incident with parallel edges.

\begin{lemma}
\label{lem:double-edges}
    For each pair $u,v  \in V(G)$, $|E_G(u,v)| \leq 2$.  
\end{lemma}
\begin{proof}
If $|E_G(u,v)| \geq 3$, then $\rho_{G}(\{u,v\}) \leq 6+6 - 3(4) = 0$, contradicting \Cref{lem:j+2}.
\end{proof}

\begin{lemma}
\label{lem:unique-weak}
    $|\Pi_3| \neq 1$.
\end{lemma}
\begin{proof}
    Suppose that $\Pi_3$ contains a unique vertex $x$.
    Obtain $\rho': V(G) \rightarrow \{3,4,6\}$ from $\rho$ by letting $\rho'(x) = 4$ and $\rho'(v) = \rho(v)$
    for each $v \in V(G) \setminus \{x\}$.
    Assume without loss of generality that $\omega(x) = 1_x$.
    Since $\rho$ was chosen to maximize $\rho_G(V(G))$,
    and since $(G,\rho)$ is $\frac 15$-critical,
    there is a $\frac 15$-distribution $\Phi'$ for $(G,\rho',H,L,\mathcal D)$.
    For each $(H,L)$-coloring $\phi'$ of $G$, 
    write $\Pr(\phi')$
    for the probability that an $(H,L)$-coloring of $G$ randomly sampled from $\Phi'$ is equal to $\phi'$.

    Now, we construct a distribution $\Phi$ on $(H,L)$-colorings $\phi$ of $G$ as follows.
    For each $i \in \{1,2,3\}$, write $p_i$ for the probability that $L_{\mathcal D}$ forbids $i_x$.
    Whenever $L_{\mathcal D}$
    forbids $i_x$,
    we choose an index $j \in \{1,2,3\}$ with conditional probability equal to entry $A_{ji}$ 
    of the matrix $A$ from the proof of  \Cref{lem:pendent-heavy}.
    Then, sample an $(H,L)$-coloring $\phi'$
    from $\Phi'$ conditioning on the event that $\phi'(x) = j_x$.
    Finally, define $\phi = \phi'$.
    This process gives us a distribution $\Phi$ on $(H,L)$-colorings of $G$.
    As each diagonal entry of $A$ is $0$,
    the probability that $\phi$ assigns a forbidden color to $x$ is $0$.

    Now, we claim that $\Phi$ is a $\frac 15$-distribution for $(G,\rho,H,L,\mathcal D)$.
        As shown in \Cref{lem:pendent-heavy},
    writing $\mb p = [p_1, p_2, p_3]^T$,
    each index $j$ is chosen with probability equal to entry $j$ of $A\mb p = [\frac 25, \frac 3{10}, \frac 3{10}]^T$; thus $j=1$ with probability $\frac 25$, and for each $j' \in \{2,3\}$, $j=j'$ with probability $\frac 3{10}$.
    Since an $(H,L)$-coloring $\phi'$
    sampled unconditionally from $\Phi'$
    satisfies $\Pr(\phi'(x) = 1_x) = \frac 25$ and $\Pr(\phi'(x) = 2_x) = \Pr(\phi'(x) = 3_x) = \frac 3{10}$,
    it follows that 
    when we condition on the event 
    $\phi'(x) = j_x$,
    each $\phi'$ satisfying $\phi'(x) = j_x$
    is chosen with conditional probability $\frac 52 \Pr(\phi')$ when $j=1$ and $\frac{10}{3} \Pr(\phi')$ when $j \in \{2,3\}$, and all other $(H,L)$-colorings of $G$ are chosen with probability $0$.
    Therefore, each $(H,L)$-coloring $\phi'$ of $G$ is sampled with overall probability $\frac 25 \cdot \frac 52 \Pr(\phi') = \Pr(\phi')$ or $\frac 3{10} \cdot \frac{10}{3} \Pr(\phi') = \Pr(\phi')$.
    Thus,
    for each $c \in V(H)$, $\Pr(c \in \phi) = \Pr(c \in \phi')$, and hence $\Phi$ is a $\frac 15$-distribution for $(G,\rho,H,L,\mathcal D)$, a contradiction.
\end{proof}

\begin{lemma}
\label{lem:C2C4}
    If $u,v \in V(G)$ induce $C_2$, then $H[L(u), L(v)]$  consists of a $C_2$-component and a $C_4$-component.
\end{lemma}
\begin{proof}
Suppose that $u$ and $v$ are joined by two parallel edges. 
If $u \in \Pi_3$, then $\rho_G(\{u,v\}) \leq 1$.
Then, by \Cref{lem:j+2}, $\rho_G(\{u,v\}) = 1$, so that $v \in \Pi_6$.
Furthermore, by \Cref{lem:j+2}, $V(G) = \{u,v\}$.
Then, $\Pi_3 = \{v\}$, contradicting \Cref{lem:unique-weak}.
If $\{u,v\}\not\subseteq\Pi_6$, then $\rho_G(\{u,v\})\leq 0$, contradicting \Cref{lem:j+2} or $(G, \rho)$ contains an exceptional $C_2$ subgraph, contradicting \Cref{obs:no-exception}.
Therefore, without loss of generality, $u,v \in \Pi_6$.
As each edge joining $u$ and $v$ corresponds to a perfect matching in $H$, $E_H(L(u), L(v))$ consists of three $C_2$-components, a $C_2$-component and a $C_4$-component, or a $C_6$-component. 

If $H[L(u), L(v)]$ consists of three $C_2$-components or a $C_6$-component, then it is easy to check that the distribution that assigns an $(H,L)$-coloring to $G[\{u,v\}]$ uniformly at random gives a fractional $(H,L)$-packing.
As $\rho(v) = 6$ and $\rho_G(\{u,v\}) = \rho(u) - 2 \leq 4$,
 \Cref{lem:gen-diamond} is violated.
Therefore,
$H[L(u), L(v)]$ consists of a $C_2$-component and a $C_4$-component.
\end{proof}

\begin{lemma}
\label{lem:weak-parallel}
    For each $x \in \Pi_3$ 
    and neighbor $v \in N(x)$, $|E_G(x,v)| = 1$.
\end{lemma}
\begin{proof}
Let $x \in \Pi_3$, and let $v \in N(x)$. 
By \Cref{lem:double-edges}, $|E_G(x,v)| \leq 2$. Suppose that $|E_G(x,v)| = 2$. Then, $\rho_{G}(\{x,v\}) \leq 3 + 6 - 2(4) = 1$. 
By \Cref{lem:j+2}, this inequality is tight, so $\rho(v) = 6$.
Furthermore, $V(G) = \{x,v\}$.
Then, $|\Pi_3| = 1$, violating \Cref{lem:unique-weak}.
\end{proof}

\begin{lemma}
    \label{lem:cut-4}
    Let $v \in \Pi_{\geq 4}$ be a cut-vertex of $G$. Let $G_1$ and $G_2$ be connected induced subgraphs of $G$ for which $G_1 \cup G_2 = G$ and $V(G_1) \cap V(G_2) = \{v\}$.
    Let $\rho'$ be a potential function on $G$ obtained from $\rho$ by updating $\rho'(v) = 4$.
    Furthermore, let $L'$ be obtained from $L$ by possibly updating the basepoint of $v$ if $\rho(v) = 6$.
    If $V(G_2) \subseteq \Pi_{\geq 4}$, 
    then either
    \begin{enumerate}
        \item $(G_1,\rho',H,L',\mathcal D)$ has no $\frac 15$-distribution, or
        \item $(G_2,\rho',H,L',\mathcal D)$ has no $\frac 15$-distribution.
    \end{enumerate}
\end{lemma}
\begin{proof}
Suppose that $(G_1,\rho',H,L',\mathcal D)$ has a $\frac 15$-distribution $\Phi_1$ and that 
$(G_2,\rho',H,L',\mathcal D)$ has
a $\frac 15$-distribution $\Phi_2$.

We construct a distribution $\Phi$ on $(H,L)$-colorings $\phi$ of $G$ as follows. 
        First, we fix an outcome $L_{\mathcal D}$.
        We sample an $(H,L_{\mathcal D})$-coloring $\phi_1$ of $G_1$ using $\Phi_1$, and we let $\phi$ agree with $\phi_1$ on $V(G_1)$. 
        Next, for each $(H,L)$-coloring $\phi_2$ of $G_2$, let $\Pr(\phi_2)$ be the probability that an $(H,L)$-coloring of $G_2$ randomly sampled (independently of $\mathcal D$) using $\Phi_2$ equals $\phi_2$. 
        Then, if $\phi_1(v) = \omega(v)$, we choose each $(H,L)$-coloring $\phi_2$ of 
        $G_2$ with conditional probability $\frac 52\Pr(\phi_2)$ if $\phi_2(v) = \omega(v)$ and with conditional probability $0$ otherwise. If $\phi_1(v) = i_v$ for $i_v \neq \omega(v)$,
        we choose each $(H,L)$-coloring $\phi_2$ of $G_2$ with conditional probability $\frac{10}{3} \Pr(\phi_2)$ if $\phi_2(v) = i_v$ and with conditional probability $0$ otherwise. 
        Since a coloring $\phi_2$ randomly sampled from $\Phi_2$ satisfies $\Pr(\phi_2(v) = \omega(v)) = \frac 25$ and $\Pr(\phi_2(v) = i_v) = \frac 3{10}$ for each $i_v \neq \omega(v)$, these conditional distributions are valid.
    Furthermore, we see that each coloring $\phi_2$ is used to define $\phi$ on $G_2$ precisely with probability $\Pr(\phi_2)$.
    Then, we let $\phi$ agree with $\phi_2$ on 
    $G_2$.
    Since $L_{\mathcal D}$ never forbids a color at a vertex of $G_2$,
    this process gives us a distribution $\Phi$ on $(H,L_{\mathcal D})$-colorings of $G$.

    We claim that $\Phi$ is a $\frac 15$-distribution for $(G,\rho,H,L,\mathcal D)$.
    Indeed, consider a vertex $w \in V(G)$.
    If $w \in V(G_1)$,
    then for each $c \in L(w)$, $\Pr(\phi(w) = c) = \Pr(\phi_1(w) = c)$.
    If $w \in V(G_2)$, then for each $c \in L(w)$, 
$\Pr(\phi(w) = c) = \Pr(\phi_2(w) = c)$.
Furthermore, if $\rho(v) = 4$, then as $L(v)$ and $L'(v)$ assign the same basepoint to $v$, $\Pr(\phi(v) = i_v) = \Pr(\phi_1(v) = i_v) = \Pr(\phi_2(v) = i_v) = \frac 3{10}$ for each $i_v \in L(v) \setminus \{\omega(v)\}$.
Therefore, the fact that $\Phi$ is a $\frac 15$-distribution for $(G,\rho,H,L,\mathcal D)$ follows from the fact that $\Phi_1$ and $\Phi_2$ are $\frac 15$-distributions on their respective tuples.
This contradicts the choice of $(G,\rho,H,L,\mathcal D)$ and completes the proof.   
\end{proof}

\begin{lemma}
\label{lem:not-butterfly}
    $G$ is not isomorphic to the graph obtained from $K_{1,2}$ by giving each edge multiplicity $2$.
\end{lemma}
\begin{proof}
    Suppose that $V(G) = \{v,w,x\}$ and that $|E_G(v,w)| = |E_G(w,x)| = 2$.
    If $V(G) \neq \Pi_6$, then $\rho_G(V(G)) \leq 0$, contradicting \Cref{lem:j+2}; therefore, $V(G) = \Pi_6$.
    Without loss of generality, $H$ contains the edges $i_v i_w, i_w i_x$ for each $i \in \{1,2,3\}$.
    By \Cref{lem:C2C4}, without loss of generality, $H$ contains the edges $1_v 2_w, 2_v 1_w$, and $H[L(w),L(x)]$ consists of a $C_2$ and a $C_4$. We consider two cases.
    \begin{enumerate}
        \item First, suppose 
        that $H$ contains the edges $1_w 2_x, 2_w 1_x$.
        We define $\phi_1 = (1_w, 3_v, 3_x)$, $\phi_2 = (2_w, 3_v, 3_x)$, $\phi_3 = (3_w, 1_v, 1_x)$, $\phi_4 = (3_w, 2_v, 2_x)$.
        We choose $j \in \{1,\dots, 4\}$ uniformly at random and give $G$ the $(H,L)$-coloring $\phi_j$.
        It is easy to see that each $c \in V(H)$ is assigned with probability at least $\frac 14$, so this gives us a $\frac 15$-distribution for $(G,\rho,H,L,\mathcal D)$, a contradiction.
        \item Otherwise, suppose
        without loss of generality
        that $H$ contains the edges $2_w 3_x, 3_w 2_x$.
        We define $\phi_1 = (1_w, 3_v, 2_x)$, $ \phi_2 =  (1_w, 3_v, 3_x)$, $\phi_3 = (2_w, 3_v, 1_x)$, $\phi_4 =  (3_w, 1_v, 1_x)$, $\phi_5 =  (3_w, 2_v, 1_x)$.
        We choose $j \in \{1,\dots,5\}$ uniformly at random and give $G$ the $(H,L)$-coloring $\phi_j$. It is easy to see that each $c \in V(H)$ is assigned with probability at least $\frac 15$, so this gives us a $\frac 15$-distribution for $(G,\rho,H,L,\mathcal D)$, a contradiction.
        
    \end{enumerate}
\end{proof}

\begin{lemma}
\label{lem:butterfly}
    Let $v$ be a $2$-vertex. If $v$ is joined to a neighbor $w$ by parallel edges,
    then $|E_G(v,w)| = 2$, $v,w \in \Pi_6$,
    $d(w) = 4$,
    and $w$ has a neighbor $x \in \Pi_6 \setminus \{v\}$ for which $|E_G(w,x)| = 2$.
\end{lemma}
\begin{proof}
    By~\Cref{lem:weak-parallel}, $\rho(v) \geq 4$, and by~\Cref{lem:double-edges}, $|E_G(v,w) | = 2$.
    Also, by \Cref{lem:weak-parallel},
    $\rho(w) \geq 4$. 
    If $\{v,w\}\not\subseteq\Pi_6$, then either $(G, \rho)$ has an exceptional $C_2$ subgraph  or $\rho_G(\{v, w\})= 0$.
    In both cases, we have a contradiction.
    Therefore, $v,w\in\Pi_6$.

    By \Cref{lem:C2C4}, we may assume $i_vi_w$ is an edge of $H$ for $i\in\{1,2,3\}$ and that 
    $1_v2_w,  2_v1_w \in E(H)$.
    We let $G' = G-v$.
    We define $\rho':V(G') \rightarrow \{3,4,6\}$ so that $\rho'(w) = 4$ and $\rho'(u) = \rho(u)$ for all $u \in V(G') \setminus \{w\}$.
    We also obtain $L'$ from $L$ by restricting to $V(G')$ and updating the basepoint of $w$ to $3_w$.
    
    Suppose that $(G',\rho',H,L',\mathcal D)$
    has a $\frac 15$-distribution $\Phi'$.
    We obtain a $\frac 15$-distribution $\Phi$ for $(G,\rho, H,L,\mathcal D)$ as follows.
    We randomly sample $\phi'$ from $\Phi'$,
    and we define $\phi$ so that $\phi(u) = \phi'(u)$ for each $u \in V(G')$.
    Then, we choose a non-neighbor $c$ of $\phi'(w)$
    from $L(v)$ uniformly at random, and we assign $\phi(v) = c$.
    We claim that $\Phi$ is a $\frac 15$-distribution for $(G,\rho,H,L,\mathcal D)$.
    It is easy to check that the relevant probabilities hold for each vertex in $ V(G')$.
    Furthermore, $\Pr(\phi(v) = 1_v) = \Pr(\phi(v) = 2_v) = \frac 12 \Pr(\phi'(w) = 3_w) = \frac 15$. Additionally, 
    $\Pr(\phi(v) = 3_v) = 1 - \Pr(\phi'(w) = 3_w) = \frac 35$. Thus, $\Phi$ is 
    a $\frac 15$-distribution for $(G,\rho,H,L,\mathcal D)$,
     a contradiction. Therefore,  $(G',\rho',H,L',\mathcal D)$ has no $\frac 15$-distribution.

     Therefore $G'$ has a vertex subset $U$
     for which  $G'[U]$ has a $(\rho',\frac 15)$-critical spanning subgraph.
     As $G$ is $(\rho,\frac 15)$-critical, it follows that $\rho'$ does not agree with $\rho$ on $U$; therefore, $U$ contains $w$.
    Now, suppose that 
    $\rho'_{G'}(U) \leq 0$.
    Then,
    as $\rho(w) = \rho'(w) + 2$, we have
     $\rho_{G}(U \cup \{v\}) = \rho'_{G'}(U) + 2 - 2(4) + 6 \leq 0$, contradicting \Cref{lem:j+2}.
    Therefore, $\rho'_{G'}(U) \geq 1$, so that $(G'[U],\rho')$ has an
    exceptional subgraph.
    As $\rho'(w) = 4$ and $w \in U$,
    \Cref{obs:Ilkyoo}
    implies that $G'[U]$ is a $C_2$ for which $\rho'_{G'}(U) = 2$.
    Therefore, $w$ is joined to a neighbor $x$ by parallel edges.
    As $\rho_G(\{v,w,x\}) \geq 1$ by \Cref{lem:j+2}, $x \in \Pi_6$.

    Finally, we show that $d(w) = 4$. By \Cref{lem:weak-deg},  $d(w) \leq 5$.
    Suppose that $d(w) = 5$, 
    and write $y$ for the unique neighbor of $w$ in $V(G) \setminus \{w,x\}$.
    Then, as $\rho_G(\{v,w,x\})=2$,
    \Cref{lem:j+2}
    implies that
    $d(x) = 2$.
    Without loss of generality, $H$ contains the edges $i_w i_y,
    i_w i_x$ for $i \in \{1,2,3\}$. We consider two cases.
    \begin{enumerate}
        \item First, suppose that $H$ contains the edges $1_w 2_x, 2_w 1_x$. Let $G'' = G-x$.
        As $G$ is $(\rho,\frac 15)$-critical, $(G'',\rho,H,L,\mathcal D)$ has a $\frac 15$-distribution $\Phi''$. We construct a distribution $\Phi$ for $(G,\rho,H,L,\mathcal D)$ as follows. For each outcome $L_{\mathcal D}$, we sample a random $(H,L_{\mathcal D})$-coloring $\phi''$ of $G''$.
        We let $\phi(z) = \phi''(z)$ for each $z \in V(G'')$. Then, we let $\phi(x) = i_x$ if and only if $\phi''(v) = i_v$. As $N_H(i_x) = N_H(i_v)$ for each $i \in \{1,2,3\}$,
        $\phi$ is an $(H,L_{\mathcal D})$-coloring of $G$.
        Furthermore, as $\Phi''$ is a $\frac 15$-distribution, and as $\Pr(\phi(x) = i_x) = \Pr(\phi''(v) = i_v) \geq \frac 15$ for each $i \in \{1,2,3\}$, the random $(H,L_{\mathcal D})$-coloring $\phi$ gives a $\frac 15$-distribution for $(G,\rho,H,L,\mathcal D)$.
        \item Otherwise, suppose without loss of generality that $H$ contains the edges $2_w 3_x$ and $3_w 2_x$.
        Let $G'' = G - \{v,w,x\}$. As $G$ is $(\rho,\frac 15)$-critical, 
         $(G'',\rho,H,L,\mathcal D)$ has a $\frac 15$-distribution $\Phi''$. We construct a distribution $\Phi$ for $(G,\rho,H,L,\mathcal D)$ as follows. For each outcome $L_{\mathcal D}$, we sample a random $(H,L_{\mathcal D})$-coloring $\phi''$ of $G''$.
         We let $\phi(z) = \phi''(z)$ for each $z \in V(G'')$.
         We also define $\phi_1 = (1_w, 3_v, 2_x)$, $ \phi_2 =  (1_w, 3_v, 3_x)$, $\phi_3 = (2_w, 3_v, 1_x)$, $\phi_4 =  (3_w, 1_v, 1_x)$, $\phi_5 =  (3_w, 2_v, 1_x)$.

         For each $i \in \{1,2,3\}$, we write $p_i = \Pr(\phi''( y ) = i_{y})$.
         We also define $q_i = \frac 52 p_i - \frac 12$, and we observe $q_i \in [0,1]$ and $q_1 + q_2 + q_3 = 1$.
         Whenever $\phi''(  y ) = i_y $, we let $\phi$ agree with $\phi_j$ on $\{w,v,x\}$ with conditional probability equal to entry $A_{ji}$ in the following matrix:
         \[
            A = \frac 1{10} \begin{bmatrix}
            0 & \frac{q_1 + q_2}{p_2} & \frac{1 + q_3}{p_3} \\
            0 & \frac{q_1 + q_2}{p_2} & \frac{1 + q_3}{p_3} \\
            \frac{2q_1 + 2q_2}{p_1}  & 0 & \frac{2q_3}{p_3} \\
            \frac{2q_1 + q_3}{p_1} & \frac{2q_2 + q_3}{p_2} & 0 \\
            \frac{2q_1 + q_3}{p_1} & \frac{2q_2 + q_3}{p_2} & 0
            \end{bmatrix}.
         \]
         As each column of $A$ sums to $1$ and has nonnegative entries, the columns of $A$ give valid conditional probability distributions.
    Letting $\mb p = [p_1, p_2, p_3]^T$, we see that each coloring $\phi_j$ is used to define $\phi$ on $\{w,v,x\}$ with probability equal to entry $j$ of $A \mb p$, which is $1/5$.
    Therefore, for each $u \in \{w,v,x\}$ and $i \in\{1,2,3\}$, $\Pr(\phi(u) = i_u) \geq \frac 15$.
    As $\Phi''$ is a $\frac 15$-distribution, it follows that the random coloring $\phi$ gives a $\frac 15$-distribution for $(G,\rho,H,L,\mathcal D)$, a contradiction. 
    \end{enumerate}
    As we reach a contradiction in both cases, it follows that $d(w) = 4$.
\end{proof}

\begin{lemma}
\label{lem:3-parallel-method}
    Let $u,v \in \Pi_6$ be $3$-vertices joined by parallel edges. Suppose that the following hold:
    \begin{enumerate}
        \item $u'$ and $v'$ are the respective neighbors of $u$ and $v$ in $V(G) \setminus \{u,v\}$, $u'\neq v'$.
        \item $H$ contains the edges $i_v i_{v'}, i_u i_{u'}, i_u i_v$ for $i \in \{1,2,3\}$, as well as  $1_u 2_v, 2_u 1_v$.
    \end{enumerate}
    Let $G'$ be obtained from $G$ by deleting $u,v$ and adding an edge $u'v'$.
    Let $H'$ be obtained from $H$ by deleting $L(u) \cup L(v)$
    and adding edges $i_{u'} i_{v'}$ for $i \in \{1,2,3\}$.
    Then, $(G',\rho,H',L,\mathcal D)$ has no $\frac 15$-distribution.
\end{lemma}
\begin{proof}
        Suppose that $(G',\rho,H',L,\mathcal D)$ has a $\frac 15$-distribution $\Phi'$.
        We construct a $\frac 15$-distribution $\Phi$
        for $(G,\rho,H,L,\mathcal D)$  as follows. As $(G,\rho,H,L,\mathcal D)$ is chosen to have no $\frac 15$-distribution, this will give us a contradiction.

        Consider an outcome $L_{\mathcal D}$ of $\mathcal D$.
        First,
        we sample an $(H',L_{\mathcal D})$-coloring $\phi'$ from $\Phi'(L_{\mathcal D})$. Then, we define an $(H,L_{\mathcal D})$-coloring $\phi$ of $G$ as follows.
        For $w \in V(G) \setminus \{u,v\}$, we let $\phi(w) = \phi'(w)$.
        Write $p_{ij}$ for the probability that  $\phi'(u') = i_{u'}$ and $\phi'(v') = j_{v'}$,
        where the probabilities are considered with respect to the distribution $\Phi$ (and not $\Phi(L_{\mathcal D})$).
        %Note that $p_{12} + p_{21} + p_{13} + p_{23} + p_{31} + p_{32} = 1$.
        
        By symmetry of $u$ and $v$, we assume without loss of generality that $p_{31} + p_{32} \geq p_{13} + p_{23}$
        By symmetry of the labels $1$ and $2$, we assume without loss of generality that $p_{21} \geq p_{12}$.
        We write $\phi_1 = (1_u,3_v)$, $\phi_2 = (2_u, 3_v)$, $\phi_3 = (3_u, 1_v)$, $\phi_4 = (3_u, 2_v)$.

        Now, we proceed according to the first applicable case below.
        In each case, we provide a matrix $A$ 
        with columns indexed in order by $(1_{u'}, 2_{v'})$, 
        $(2_{u'}, 1_{v'})$,  
        $(1_{u'}, 3_{v'})$, 
        $(2_{u'}, 3_{v'})$, 
        $(3_{u'}, 1_{v'})$, $(3_{u'}, 2_{v'})$, 
        and such that each column of $A$ has nonnegative entries summing to $1$.
        Given the random coloring $\phi'$,
        we consider the column $j$ of $A$ corresponding to the pair $(\phi'(u'), \phi'(v'))$, 
        and we let $\phi$ agree with $\phi_i$ on $(u,v)$ with conditional probability $A_{ij}$.
        In this way, the probability that $\phi$ agrees with $\phi_i$ on $(u,v)$ is equal to entry $i$ of $A\mb p$, where $\mb p = [p_{12}, p_{21}, p_{13}, p_{23}, p_{31},p_{32}]^T$.
        \begin{enumerate}[(a)]
             \item If $p_{31} + p_{32} \geq \frac 25$,
        then we 
        let $\phi$ assign the coloring $\phi_i$ on $(u,v)$ with probability equal to entry $i$ in the vector in (a) below. 
        As $p_{12} + p_{13} \geq \frac 15$ and $p_{21} + p_{23} \geq \frac 15$, each entry is at least $\frac 15$. 
        \item If $p_{12} + p_{13} + p_{23} \geq \frac 25$,
        then  we 
        let $\phi$ assign the coloring $\phi_i$ on $(u,v)$ with probability equal to entry $i$ in the vector in (b) below. 
        As $p_{21} \geq p_{12}$ and $p_{31} + p_{32} \geq p_{13} + p_{23}$,
        each entry is at least $\frac 15$.
        \item If $p_{12} + p_{13} + p_{23} < \frac 25$,
        then $p_{21} + p_{31} + p_{32} > \frac 35$. As case (a) does not apply, $p_{21} > \frac 15$.
        If $p_{21} + p_{31} > \frac 25$, then we 
        let $\phi$ assign the coloring $\phi_i$ on $(u,v)$ with probability equal to entry $i$ in the vector in (c) below; otherwise, if $p_{21} + p_{31} \leq \frac 25$, 
        then we use the vector in (c').
        In both cases, each entry is at least $\frac 15$.
        As the vector in (c') is fairly involved, we compute it explicitly,  as shown below. 
       \end{enumerate}
       % \end{itemize}
        In each case, each coloring $\phi_i$ is assigned with probability at least $\frac 15$,
        which implies that $\Phi$ is a $\frac 15$-distribution, a contradiction.

        \[
                    \begin{bmatrix}
                        0 & 0 & 0 & 0 & \frac 12 & \frac 12  \\ 
                        0 & 0 & 0 & 0 & \frac 12 & \frac 12  \\
                        1 & 0 & 1 & 0 & 0 & 0\\
                        0 & 1 & 0 & 1 & 0 & 0 
                    \end{bmatrix}
                    \begin{bmatrix}
                        p_{12} \\ p_{21} \\ p_{1 3} \\ p_{23} \\ p_{3 1} \\ p_{32} 
                    \end{bmatrix}
                    \indent 
                        \begin{bmatrix}
                        0 & \frac 12 & 0 & 0 &  \frac 12 & \frac 12 \\ 
                        \frac 12 & 0 & 0 & 0 & \frac 12 & \frac 12  \\
                        \frac 12 & 0 & \frac 12 &  \frac 12 & 0 & 0\\
                        0 & \frac 12 & \frac 12 & \frac 12 & 0 & 0 
                    \end{bmatrix}
                    \begin{bmatrix}
                        p_{12} \\ p_{21} \\ p_{1 3} \\ p_{23} \\ p_{3 1} \\ p_{32} 
                    \end{bmatrix}
\indent
            \begin{bmatrix}
                0 & 1 - \frac 1{5p_{21}} & 0 & 0 & 1 & 0 \\
                1 & 0 & 0 & 0 & 0 & 1\\
                0 & 0 & 1 & 1 & 0 & 0 \\
                0 & \frac 1{5 p_{21}} & 0 & 0 & 0 & 0 
            \end{bmatrix}
                                \begin{bmatrix}
                        p_{12} \\ p_{21} \\ p_{1 3} \\ p_{23} \\ p_{3 1} \\ p_{32} 
                    \end{bmatrix}
                \]
                \[
                    \ \ \ \ \text{(a)} \ \ \ \ \  \ \ \ \ \ \ \ \ \ \ \ \ \ \ \ \ \ \ \ \ \ \ \ \ \ \ \ \ \ \ \ \text{(b)}  \ \ \ \ \ \ \ \ \ \ \ \ \ \ \ \ \ \ \ \ \ \ \ \ \ \ \ \ \ \ \ \ \ \ \ \ \text{(c)} \ \ \ \ \ \ \ \ \ \ \ 
                \]
                \[
\begin{bmatrix}
                0 & 1 - \frac 1{5p_{21}} & 0 & 0 & 1 & \frac{\frac 25 - p_{21} - p_{31}}{p_{32}} \\
                1 & 0 & 1 - \frac{1}{5(p_{13}+p_{23})}  & 1 - \frac{1}{5(p_{13}+p_{23})}  & 0 & \frac{p_{21} + p_{31} + p_{32} - \frac 25}{p_{32}} \\
                0 & 0 & \frac{1}{5(p_{13}+p_{23})} & \frac{1}{5(p_{13}+p_{23})}  & 0 & 0 \\
                0 & \frac 1{5 p_{21}} & 0 & 0 & 0 & 0 
\end{bmatrix}
\begin{bmatrix}
    p_{12} \\ p_{21} \\ p_{1 3} \\ p_{23} \\ p_{3 1} \\ p_{32} 
\end{bmatrix}
    =
    \begin{bmatrix}
         - \frac 15 + \frac 25 \\[0.4em]
        1 - \frac{p_{13} + p_{23}}{5(p_{13} + p_{23})} - \frac 25 \\[0.4em]
        \frac{p_{13} + p_{23}}{5(p_{13} + p_{23})}\\[0.4em]
        \frac 15
    \end{bmatrix}
    =
    \begin{bmatrix}
        \frac 15\\[0.4em]
        \frac 25\\[0.4em]
        \frac 15\\[0.4em]
        \frac 15
    \end{bmatrix}
                \]
             \[
                    \text{(c')}  \]

    \end{proof}

Now, we are ready to prove the main lemma of this subsection.

\begin{lemma} 
\label{lem:full-2-parallel}
    No $2$-vertex is incident with parallel edges.
\end{lemma}
\begin{proof}
%{\PB This proof was shortened. Please check.}
    Suppose a $2$-vertex $v\in V(G)$ is joined to a neighbor $w$ by parallel edges.
    By \Cref{lem:butterfly},
    $|E_G(v,w)| = 2$, $v,w \in \Pi_6$, $d(w) = 4$, 
    and $w$ has a neighbor $x \in \Pi_6
    \setminus\{v\}$ for which $|E_G(w,x)| = 2$.
    Without loss of generality, $H$ contains the edges $i_v i_w, i_w i_x$ for each $i \in \{1,2,3\}$.
    Furthermore, by \Cref{lem:C2C4}, $H$ contains the edges $1_v 2_w, 2_v 1_w$ without loss of generality.
    We consider two cases.
    \begin{enumerate}
        \item Suppose $H$ contains the edges $1_w 2_x, 2_w 1_x$. Then, we let $G' = G - v$.
        As $G$ is $(\rho,\frac 15)$-critical, $(G',\rho,H,L,\mathcal D)$ has a $\frac 15$-distribution $\Phi'$. We construct a distribution $\Phi$ for $(G,\rho,H,L,\mathcal D)$ as follows. For each outcome $L_{\mathcal D}$, we sample a random $(H,L_{\mathcal D})$-coloring $\phi'$ of $G'$.
        We let $\phi(z) = \phi'(z)$ for each $z \in V(G')$. Then, we let $\phi(v) = i_v$ if and only if $\phi'(x) = i_x$. 
        As $N_H(i_v) = N_H(i_x)
        \cap L(w)$ for each $i \in \{1,2,3\}$,
        $\phi$ is an $(H,L_{\mathcal D})$-coloring of $G$.
        Furthermore, as $\Phi'$ is a $\frac 15$-distribution, and as $\Pr(\phi(v) = i_v) = \Pr(\phi'(x) = i_x) \geq \frac 15$ for each $i \in \{1,2,3\}$, the random $(H,L_{\mathcal D})$-coloring $\phi$ gives a $\frac 15$-distribution for $(G,\rho,H,L,\mathcal D)$.
        \item Otherwise, suppose without loss of generality that $H$ contains the edges $2_w 3_x, 3_w 2_x$.
        Since $\rho(\{v, w, x\})=3(6)-4(4)=2$, by  \Cref{lem:j+2}, $d(x) \leq 3$.
        Furthermore, $V(G) \neq \{v,w,x\}$ by \Cref{lem:not-butterfly}; therefore, $d(x) = 3$.
        Let $y$ be the unique neighbor of $x$ in $V(G) \setminus \{w\}$.
        Without loss of generality, $H$ contains the edges $i_x i_y$ for $i \in \{1,2,3\}$.
        Furthermore, $y \in \Pi_{\geq 4}$. Indeed, if $y \in \Pi_3$, then $\rho_G(\{v,w,x,y\}) = 1$; then,  \Cref{lem:j+2} implies that $V(G) = \{v,w,x,y\}$, and \Cref{lem:unique-weak} is contradicted.

        Observe that $y$ is a cut-vertex of $G$, and let $G_1=G[\{v,w,x,y\}]$ and $G_2=G-\{v,w,x\}$.
        We claim that $\rho(y) = 6$. Indeed, suppose $\rho(y) = 4$. 
        Then, \Cref{lem:cut-4} implies that for some $i \in \{1,2\}$, $(G_i, \rho,H,L,\mathcal D)$ has no $\frac 15$-distribution, contradicting the criticality of $(G,\rho)$. Therefore, $y \in \Pi_6$.

        Let $\rho'$ be obtained from $\rho$ by updating $\rho'(y) = 4$, and let $L'$ be obtained from $L$ by updating the basepoint of $y$ to $1_y$.
        By \Cref{lem:cut-4}, for some $i \in \{1,2\}$, $(G_i,\rho',H,L',\mathcal D)$ has no $\frac 15$-distribution.

        We claim that $(G_1, \rho', H,L',\mathcal D)$ has a $\frac 15$-distribution. To see this, we observe that assigning one of the following $(H,L)$-colorings $\phi_j$ to $G_1$ uniformly at random causes each $c \in V(H)$ to be assigned with appropriate probability: $\phi_1 = (1_v,3_w,1_x,3_y)$, 
        $\phi_2 = (1_v, 3_w, 1_x, 3_y)$, 
        $\phi_3 = (2_v, 3_w, 1_x, 3_y)$, 
        $\phi_4 = (2_v, 3_w, 1_x, 2_y)$, 
        $\phi_5 = (3_v,2_w,1_x,2_y)$, 
        $\phi_6 = (3_v,2_w,1_x,2_y)$,
        $\phi_7 = (3_v,1_w,2_x,1_y)$,
        $\phi_8 = (3_v,1_w,2_x,1_y)$,
        $\phi_9 = (3_v,1_w,3_x,1_y)$,
        $\phi_{10} = (3_v,1_w,3_x,1_y)$.
        In particular, each of $2_y$ and $3_y$ is assigned to $y$ with probability $\frac{3}{10}$. Thus, $(G_1, \rho', H,L',\mathcal D)$ has a $\frac 15$-distribution.

        Therefore, $(G_2, \rho', H,L',\mathcal D)$ has no $\frac 15$-distribution. 
        Thus, $G_2$ has some vertex subset $U$ for which $G_2[U]$ contains a spanning $(\rho',\frac 15)$-critical subgraph.
        Since $G$ is $(\rho, \frac 15)$-critical, $\rho$ and $\rho'$ do not agree on $U$; therefore, $U$ contains $y$.
        If $\rho'_{G'}(U) \leq 0$, then $\rho_G(U) \leq 2$, which implies that $\rho_G(U \cup \{v,w,x\}) \leq \rho_G(U) - 4 + \rho_G(\{v,w,x\}) = \rho_G(U) - 2 \leq 0$, contradicting \Cref{lem:j+2}. Therefore,
        $(G'[U],\rho')$ has a spanning exceptional subgraph, and $\rho'_{G'}(U) \geq 1$.
        Thus, by \Cref{obs:Ilkyoo},
        $G'[U]$ is a $C_2$, $\rho'(y) = 4$, and $U$ contains a single additional vertex $z$ for which $\rho'(z) = \rho(z) = 6$, and such that $z$ is joined to $y$ by parallel edges.

        Now, as $\rho_G(\{v,w,x,y,z\}) = 2$,
        at most one edge $e$ joins $\{v,w,x,y,z\}$ to $V(G) \setminus \{v,w,x,y,z\}$.
        If no such edge $e$ exists, then $d(z) = 2$ and $d(y) = 3$, contradicting \Cref{lem:butterfly}.
        Similarly, if $e$ exists and is not incident with $z$, then $d(z) = 2$ and $y$ is incident with at most one parallel edge pair; thus,  \Cref{lem:butterfly} is contradicted again. Therefore, we conclude that $e$ is incident with $z$, so that $d(z) = 3$, and $z$ has a single neighbor $z' \in V(G) \setminus \{y\}$ (see  \Cref{fig:butterfly-chain}).

        \begin{figure}
        \begin{center}
        \begin{tikzpicture}[scale=1.2, every node/.style={circle, draw, fill=gray!30, inner sep=2pt}]

        \node[draw = white,fill = white] (z) at (0,-0.5) {$v$};
        \node[draw = white,fill = white] (z) at (1,-0.5) {$w$};
        \node[draw = white,fill = white] (z) at (2,-0.5) {$x$};
        \node[draw = white,fill = white] (z) at (3,-0.5) {$y$};
        \node[draw = white,fill = white] (z) at (4,-0.5) {$z$};
        \node[draw = white,fill = white] (z) at (5,-0.5) {$z'$};

        \node (v) at (0,0) {};
        \node (w) at (1,0) {};
        \node (x) at (2,0) {};
        \node (y) at (3,0) {};
        \node (z) at (4,0) {};
        \node (z') at (5,0) {};

          \draw[bend left=10] (v) to (w);
          \draw[bend right=10] (v) to (w);
        \draw[bend left=10] (x) to (w);
          \draw[bend right=10] (x) to (w);
        \draw (x) -- (y);
        \draw[bend right=10] (y) to (z);
        \draw[bend left=10] (y) to (z);
        \draw (z) -- (z');
        \end{tikzpicture}
        \end{center}
            \caption{The figure shows the graph described in the proof of \Cref{lem:full-2-parallel}.}
            \label{fig:butterfly-chain}
        \end{figure}
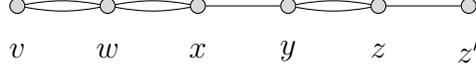

        Now, we are ready for our final contradiction. As $y$ and $z$ are $3$-vertices of $\Pi_6$ joined by parallel edges,
        \Cref{lem:3-parallel-method} implies that the graph $G'$ obtained from $G$ by deleting $\{y,z\}$ and adding an edge $x z'$ has a vertex subset $U$ for which $G'[U]$ has a spanning $(\rho,\frac 15)$-critical subgraph.
        As $G$ is $(\rho,\frac 15)$-critical, $U$ contains $x$ and $z'$.
        If $(G'[U],\rho)$ contains a spanning exceptional $C_2$, then we must have $xz' \in E(G)$, a contradiction to our assumption that $N_G(x) = \{w,y\}$.
        If $G'[U]$ 
        contains a spanning subgraph $I_m \in \mathcal I$, then as $I_m$ is $2$-connected, Menger's theorem implies that there exists a path $P$ from $x$ to $z'$ in $G'$ that avoids the edge $xz'$.
        Thus, as $P$ is also a path in $G$,
        $x$ and $z'$ are joined in $G$ by two edge-disjoint paths, a contradiction to our assumption that $xy$ is a cut-edge of $G$. Therefore, $\rho_{G'}(U) \leq 0$.

        Now, we see that $\rho_G(U \cup \{y,z\}) = \rho_{G'}(U) + 4 - 8 + \rho_G(\{y,z\}) = \rho_{G'}(U) \leq 0$, and thus \Cref{lem:j+2} is contradicted. This completes the proof.
    \end{enumerate}
\end{proof}

\subsection{Terminal blocks of $G$}
The main goal of this subsection
is to show that if $G$ is not $2$-connected, then every terminal block of $G$ contains a vertex of $\Pi_3$.

\begin{lemma}
    \label{lem:2-char}
    If $S \subseteq \Pi_{\geq 4}$ 
    satisfies $|E_G(S,\overline S)| = 1$,
    then $\rho_G(S) = 2$.
\end{lemma}
\begin{proof}
    If $G[S]$ is disconnected, then 
    by the connectivity of $G$,
    each component of $G[S]$ has an edge in its boundary, which contradicts 
    the assumption that 
    $|E_G(S,\overline S)| = 1$. Therefore, $G[S]$ is connected.
    By \Cref{lem:j+2}, $\rho_G(S) \geq 2$.
    By \Cref{lem:pendent-heavy}, $|S| \geq 2$.
    Now, let $v \in S$ be the unique vertex with a neighbor in $ V(G) \setminus S$.

    We claim that $\rho(v) = 6$.
    Indeed, suppose $\rho(v) = 4$.
        Let $G_1 = G[S]$, and let $G_2=G-(S\setminus\{v\})$. 
        As $G$ is $(\rho,\frac 15)$-critical,
        for each $i \in \{1,2\}$, $(G_i,\rho,H,L,\mathcal D)$ has a $\frac 15$-distribution, contradicting \Cref{lem:cut-4}.
        
    Therefore, $\rho(v) = 6$.
        We define $\rho'(v) = 3$ and $\rho'(u) = \rho(u)$ for each $u \in S \setminus \{v\}$.
        By \Cref{lem:reducible},
        there is a $\frac 15$-admissible distribution $\mathcal D'$ on the forbidden color at $v$
        such that $(G[S],\rho',H,L,\mathcal D')$ is not $\frac 15$-flexible.
        Then, there is a subset $S' \subseteq S$ for which $G[S']$ has a spanning $(\rho',\frac 15)$-critical subgraph $G'$.
        Since $S' \subsetneq V(G)$ and $G$ is $(\rho, \frac 15)$-critical,
        $S'$ contains $v$.
        As $\rho'(v) = 3$, $G'$ is not exceptional, and hence $\rho'_G(S') \leq 0$.
        Therefore, $\rho_G(S') \leq 3$, and by parity, $\rho_G(S') \leq 2$.
        Since $G[S]$ is connected and $S'$ contains $v$, $|E_G(S',\overline{S'})| \geq 2$
        whenever
        $S' \subsetneq S$; therefore, $S' = S$ by \Cref{lem:j+2}, and $\rho_G(S)\leq 2$.
        Finally, by \Cref{lem:j+2}, $\rho_G(S) = 2$.
\end{proof}

\begin{lemma} 
\label{lem:potential-2}
    For each nonempty subset $U \subseteq V(G)$,
    if $\rho_G(U) \leq 2$ and $U \setminus \{u\} \subseteq \Pi_{\geq 4}$ for some $u \in U$,
    then $U = V(G)$.
\end{lemma}
\begin{proof}
    %If $U \subseteq \Pi_{\geq 4}$, then the conclusion follows from Lemma \ref{lem:term-4}.
    Suppose that the lemma is false, and let $U \subsetneq V(G)$ be the largest subset of $V(G)$ that violates the lemma.
    We observe that $U \subseteq \Pi_{\geq 4}$.
    Indeed, if not, then 
    there is a unique vertex $u \in U \cap \Pi_3$.
    Then, by parity, $\rho_G(U) = 1$, and then \Cref{lem:j+2} implies that $V(G) = U$,
    contradicting the assumption that $U \subsetneq V(G)$.
    %contradicting \Cref{lem:unique-weak}.
    
    Since each vertex of $U$ has even potential and $\rho_G(U) \geq 1$ by \Cref{lem:j+2},
    $\rho_G(U) = 2$, and
    there is a unique vertex $v \in U$ with a single neighbor $w \in V(G) \setminus U$.
    If $\rho(w) = 3$, then $\rho_G(U + w) = 1$, and then \Cref{lem:j+2} implies that $V(G) = U + w$, contradicting \Cref{lem:unique-weak}.
    Therefore, $\rho(w) \geq 4$.
    If $\rho(w) = 4$, then $\rho_G(U + w) = 2$.
    Furthermore, as $d_G(w) \geq 2$ by \Cref{lem:pendent-heavy},
    $U+w \subsetneq V(G)$.
    This contradicts our choice of $U$.
    Therefore, $\rho(w) = 6$.

    Let $\rho':V(G) \rightarrow \{3,4,6\}$ be obtained from $\rho$ by updating $\rho'(w) = 4$.
    Let $G_1 = G[U \cup \{w\}]$, and
    let $G'' = G - U$.
    By \Cref{lem:cut-4},
    for some $i \in \{1,2\}$, $(G_i,\rho',H,L,\mathcal D)$ has no $\frac 15$-distribution.
    \begin{claim}
    $(G_1,\rho',H,L,\mathcal D)$ has a $\frac 15$-distribution.   
    \end{claim}
    \begin{proof}
        If not, then there is a subset $S \subseteq V(G_1)$ containing $w$ for which $\rho'_G(S) \leq 0$ or $G[S]$ contains a spanning exceptional graph.
        As $w$ is joined to $U$ by a cut-edge, no exceptional graph of $(G_1,\rho')$ contains $w$; therefore,
        $\rho'_G(S) \leq 0$.
        Then, $\rho_G(S) \leq 2$, and hence $\rho_G(S - w) \leq 0$, contradicting \Cref{lem:j+2}.
    \end{proof}

    \begin{claim}
        $(G_2, \rho',H,L,\mathcal D)$ has a $\frac 15$-distribution.
    \end{claim}
    \begin{proof}
        If no such distribution exists, then there is a subset $S \subseteq V(G_2)$ containing $w$ for which $\rho'_G(S) \leq 0$ or $(G[S],\rho')$ has a spanning exceptional $C_2$ of potential $2$.
        In the latter case, as $\rho'(w) = 4$, 
        this implies that $G[S]$ is a $C_2$, so that $w$ has a neighbor $w' \in V(G_2)$ joined to $w$ by parallel edges.
        By~\Cref{lem:weak-parallel}, $\rho(w')\geq 4$.
        If $\rho(w')= 4$, then $(G[\{w,w'\}],\rho)$ is exceptional, a contradiction. Therefore, $\rho(w') = 6$.
        Then, by \Cref{lem:full-2-parallel}, $d_G(w') \geq 3$.
        If $w'$ has a neighbor in $U$, then $\rho_G(U + w + w') \leq \rho_G(U) + \rho(w) + \rho(w') - 16 \leq -2$, contradicting \Cref{lem:j+2}.
        Otherwise,
        if $w'$ has no neighbor in $U$, then $w'$ has a neighbor in $V(G) \setminus (U + w)$; then,
        $U + w + w' \subsetneq V(G)$, and $\rho_G(U + w + w') \leq \rho_G(U) + \rho(w) + \rho(w') - 12 = \rho_G(U) \leq 2$, contradicting the maximality of $U$.
    \end{proof}
    Thus, we have a contradiction by \Cref{lem:cut-4}.
\end{proof}

\begin{lemma}
\label{lem:term-4}
    If $B$ is a terminal block of $G$ with $V(B) \subseteq \Pi_{\geq 4}$, then $B = G$.
\end{lemma}
\begin{proof}
    Let $B$  be 
    a terminal block of $G$, and suppose that $V(B) \subseteq \Pi_{\geq 4}$ and $B \neq G$.
    If $\rho_G(V(B)) = 2$, then we are done by  \Cref{lem:potential-2}; therefore, we assume that $\rho_G(V(B)) \geq 4$.
    Let $w$ be the cut-vertex of $B$, and let $G' = G - (V(B) - w)$.
    Since $V(B) \subseteq \Pi_{\geq 4}$,  \Cref{lem:pendent-heavy} implies that $B$ is not a $K_2$; therefore, $B$ is $2$-connected. 
    In particular, $d_B(w) \geq 2$.
    Therefore, by  \Cref{lem:j+2},
    $\rho_G(V(G')) \geq 3$.

    Since $\rho_G(V(B)) \geq 4$, \Cref{lem:2-char} implies that $d_{G'}(w) \geq 2$.
    Furthermore, we observe that if $w$ is adjacent to some neighbor $v \in V(B)$ by parallel edges, then $d_B(w) \geq 3$. 
    Indeed, in this case, if $d_B(w) = 2$, then as $B$ is $2$-connected,
    $V(B) = \{v,w\}$.
    Then, $d_G(v) = 2$, 
    contradicting  \Cref{lem:full-2-parallel}.

    Now, define 
    $\rho':V(G) \rightarrow \{3,4,6\}$ so that $\rho'(v) = \rho(v)$ for each $v \in V(G) \setminus \{w\}$ and $\rho'(w) = 4$.
    By \Cref{lem:cut-4},
    either $(B, \rho',H,L,\mathcal D)$ has no $\frac 15$-distribution, or 
    $(G',\rho',H,L,\mathcal D)$ has no $\frac 15$-distribution.

    \begin{claim}
    There is a $\frac 15$-distribution for 
    $(B, \rho',H,L, \mathcal D)$.
    \end{claim}
    \begin{proof}
    Suppose that no such distribution exists.
    Then, since $\rho'(w) = 4$, the minimality of $(G,\rho)$ implies that 
    there is a subset $S \subseteq V(B)$ containing $w$ for which
    either   
    $G[S]$ is a $C_2$ and $\rho'_G(S) = 2$, or
    $\rho'_G(S) \leq 0$.
    In the former case, 
    $w$ is joined to some $v \in V(B)$ by parallel edges, so as observed above, $d_B(w) \geq 3$.
    As $d_{G'}(w) \geq 2$,  \Cref{lem:weak-deg}
    implies that $d_G(w) = 5$.
    We let $U = \{v,w\}$,
    and we observe that $\rho_G(U) \leq 4$;
    therefore, by  \Cref{lem:j+2},
    $E_G(U,\overline U)$ consists of exactly three edges,
    each one incident with $w$.
    Then, $d_G(v) = 2$, and since $v$ is incident with parallel edges, \Cref{lem:full-2-parallel} is contradicted.
    Therefore, we assume that $\rho'_G(S) \leq 0$ and hence $\rho_G(S) \leq 2$.
    Since $S$ contains $w$, 
    and since $w$ has two neighbors in $V(G') \setminus S$,
    \Cref{lem:j+2} is contradicted.
    \end{proof}
    \begin{claim}
        There is a $\frac 15$-distribution for $(G',\rho',H,L,\mathcal D)$.
    \end{claim}
    \begin{proof}
        If $\rho(w) = 4$, then $\rho = \rho'$, and the claim follows from the minimality of $(G,\rho)$. Therefore, we assume that $\rho(w) = 6$.
        Suppose that no such distribution exists. Then,
        there is a subset $S \subseteq V(G')$ containing $w$ for which 
        either $(G[S],\rho')$ has a spanning exceptional subgraph or
        $\rho'_G(S) \leq 0$.
        As $G$ is $(\rho,\frac 15)$-critical, $\rho$ and $\rho'$ do not agree on $S$; therefore, $w \in S$.
        If $\rho'_G(S) \leq 0$, then $\rho_G(S) \leq 2$,
        and the fact that $d_B(w) \geq 2$ gives us a contradiction by \Cref{lem:j+2}.
        Therefore,
        $\rho'_G(S) \geq 1$, and
        $(G[S],\rho')$ has a spanning exceptional subgraph.
        By \Cref{obs:Ilkyoo},
        $(G[S],\rho')$ is an exceptional $C_2$,       
        and thus a vertex $w' \in V(G')$ is joined to $w$ by parallel edges.
         We write $U = \{w,w'\}$.
        By \Cref{lem:weak-parallel}, 
        $\rho(w') \geq 4$. 
        If $\rho(w') = 4$, then the subgraph $G[U]$ is exceptional, a contradiction; therefore, $\rho(w') = 6$,  and 
         $\rho_G(U) = 4$.
        By \Cref{lem:full-2-parallel},
        $d_G(w') \geq 3$.
        Since $d_B(w) \geq 2$,
        it follows from  \Cref{lem:j+2} that $|E_G(U, \overline U) | = 3$ and that $d_G(w') = 3$.
        Furthermore, as $w$ is a cut-vertex of $B$, $w'$ has no neighbor in $B$, so $V(B)  + w' \subsetneq V(G)$.
        Thus,
        % {\PB as $|E_G(U, \overline U) | = 3$,}
        $B + w'$ has only one edge on its boundary and is a subset of $\Pi_{\geq 4}$,
        and thus $\rho_G(B + w') = 2$ by \Cref{lem:2-char}.
        Then, we have a contradiction by  \Cref{lem:potential-2}.
    \end{proof}
    Now, using the same technique as in \Cref{lem:2-char}, we combine the distributions for $(B,\rho',H,L,\mathcal D)$ and $(G',\rho',H,L,\mathcal D)$
    to construct a $\frac 15$-distribution for $(G,\rho,H,L,\mathcal D)$, a contradiction.
\end{proof}

\subsection{Parallel edges}
The goal of this section is to show that no two $3^-$-vertices in $G$ are joined by parallel edges.
We first need to forbid the following graph as a subgraph of $G$.
Let $H_5$ be the \emph{house graph} in \Cref{fig:I5}.

\begin{figure}
\begin{center}

\begin{tikzpicture}[xscale=0.7,scale=1.5, every node/.style={circle, draw, fill=gray!30, inner sep=2pt}]
  % Define positions of the 5 vertices on a circle

\node[draw = white, fill = white] (z) at (1.3,-0.6) {$v$};
\node[draw = white, fill = white] (z) at (-1.3,-0.6) {$u$};
\node[draw = white, fill = white] (z) at (-1.3,0.2) {$u'$};
\node[draw = white, fill = white] (z) at (1.3,0.2) {$v'$};

%\node[draw = white, fill = white] (z) at (-105:1.8) {$H_5$};

\node[draw = white, fill = white] (z) at (0,0.8) {$w$};

\node (v1) at (-1,0) {};
\node (v2) at (1,0) {};
\node (v3) at (0,0.5) {};

\node (u1) at (-1,-0.5) {};
\node (u2) at (1,-0.5) {};

  % Draw single edges
  \draw (u2) -- (v2) -- (v3);
  \draw (v3) -- (v1) -- (u1);
  \draw (v1) -- (v2);

  % Draw double edges using slight bending
  \draw[bend left=10] (u1) to (u2);
  \draw[bend right=10] (u1) to (u2);
\end{tikzpicture}
\begin{tikzpicture}[xscale=0.7,scale=1.5, every node/.style={circle, draw, fill=gray!30, inner sep=2pt}]
  % Define positions of the 5 vertices on a circle

\node[draw = white, fill = white] (z) at (1.3,-0.6) {$v$};
\node[draw = white, fill = white] (z) at (-1.3,-0.6) {$u$};
\node[draw = white, fill = white] (z) at (-1.3,0.2) {$u'$};
\node[draw = white, fill = white] (z) at (1.3,0.2) {$v'$};

%\node[draw = white, fill = white] (z) at (-105:1.8) {$H_5$};

%\node[draw = white, fill = white] (z) at (0,0.8) {$w$};

\node (v1) at (-1,0) {};
\node (v2) at (1,0) {};
%\node (v3) at (0,0.5) {};

\node (u1) at (-1,-0.5) {};
\node (u2) at (1,-0.5) {};

  % Draw single edges
  \draw (u2) -- (v2);
  \draw (v1) -- (u1);
  \draw (v1) -- (v2);

  % Draw double edges using slight bending
  \draw[bend left=10] (u1) to (u2);
  \draw[bend right=10] (u1) to (u2);
\end{tikzpicture}
\end{center}
\caption{On the left is the house graph $H_5$, and on the right is its subgraph $H_4$. If $u$ and $v$ are deleted from $H_5$ and an edge is added from $u'$ to $v'$,
as in \Cref{lem:kck},
then $I_3$ is obtained. 
%{\PB I don't think we actually need $H_4$}
}
\label{fig:I5}
\end{figure}
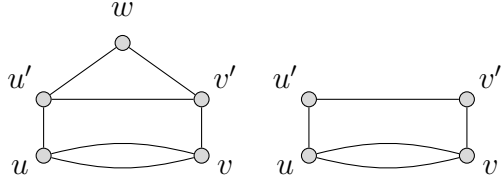

\begin{lemma}
\label{lem:I5}
    $G$ does not have $H_5$ as a subgraph.
\end{lemma}
\begin{proof}
Suppose that there is a subset $U \subseteq V(G)$ for which $G[U]$ contains a spanning subgraph isomorphic to $H_5$.
 (See \Cref{fig:I5} for the labels of the vertices of $H_5$.)
If $H_5$ is not  isomorphic to  $G[U]$, then
$\rho_G(U) \leq -2$; therefore, $H_5$ is an induced subgraph of $G$.
If $\rho(u) \leq 4$ for some $u \in U$, then $\rho_G(U) \leq 0$, a contradiction; therefore, $U \subseteq \Pi_6$.
Therefore, $\rho_G(U) = 2$, and hence by  \Cref{lem:potential-2}, $U = V(G)$, so that $G$ is isomorphic to $H_5$.
By \Cref{lem:C2C4},
we assume without loss of generality that $H$ contains the edges $i_u i_v$ for each $i \in \{1,2,3\}$, as well as the edges $1_u 2_v$ and $2_u 1_v$.
We also assume without loss of generality that $H$ contains the edges $i_u i_{u'}$ and $i_v i_{v'}$ for $i \in \{1,2,3\}$.

We construct a distribution on $(H,L)$-colorings of $G$.
Let $B_0$ be a bipartite graph with partite sets $L(u')$ and $L(v')$.
Let $i_{u'}$ and $j_{v'}$ be adjacent in $B_0$ if and only if $i_{u'} j_{v'} \not \in E(H)$ and $(i,j) \neq (3,3)$. 
 Let $M_0$ be a perfect matching in $B_0$.
 % {\IC where is $M_0$ used?}
 % {\PB I think it's supposed to help us make $H'$. }
 % {\IC thanks!}
Obtain $G'$ from $G$ 
by identifying $u'$ and $v'$ into a new vertex $u^*$ and deleting the loop.
Note that $G'$ is obtained from $I_3$ by adding a new vertex that is joined to the $2$-vertex of $I_3$ by two parallel edges.
Define a list $L'(u^*) = \{1_{u^*}, 2_{u^*}, 3_{u^*}\}$,
and let $L'(x) = L(x)$ for all $x \in V(G') \setminus \{u^*\}$.
Define a function $f:V(H) \rightarrow (V(H) \setminus (L(u') \cup L(v'))) \cup L'(u^*)$
that maps each color $i_{u'}$
and its unique neighbor $j_{v'} \in N_{M_0}(i_{u'})$ to $i_{u^*}$,
and that acts as the identity on $V(H) \setminus (L(u') \cup L(v'))$.
Then,
obtain $H'$ by taking the image of $H$ under $f$ 
as a full graph homomorphism
(which maps edges to edges and non-edges to non-edges) and deleting any parallel edges induced by $L(u^*)$, so that $L(u^*)$ induces a clique.
Since
the image of $H$ under $f$ has no loop (because we identified among non-edges given by $M_0$),
every $(H',L')$-coloring $\phi'$ of $G'$ can be transformed into an $(H,L)$-coloring $\phi$ of $G$ by  
letting $\phi = f^{-1} (\phi')$.

Before we construct a distribution on $(H',L')$-colorings $\phi'$ of $G'$, we need the following claim.
If $E_H(L'(u^*), L'(w))$ contains a $4$-cycle $C$,
then
let $j_{u^*}$ be the unique vertex of $L'(u^*
)$ not belonging to $C$.
Otherwise, choose $j_{u^*} \in L'(j_{u^*})$ arbitrarily.
Let $J = \{j_{u^*},j_{u^*},(j+1)_{u^*}, (j+1)_{u^*}, (j+2)_{u^*}\}$ be a multiset, with indices wrapping around.
Let $K = \{(1_u, 3_v), (2_u, 3_v), (3_u, 1_v), (3_u, 2_v)\}$ be the set of the four non-adjacent color pairs $(i_u, i'_v) \in L'(u) \times L'(v)$.
Let $B$ be a bipartite graph with parts $J$ and $K$,
where a pair $(i_u, i'_v) \in K$ is adjacent to a color $c \in J$
if and only if $c \not \in N_{H'}(i_u) \cup N_{H'}(i'_v)$.
See \Cref{fig:house-bip}.

\begin{figure}
\begin{tikzpicture}[xscale=0.7,scale=1.5, every node/.style={circle, draw, fill=gray!30, inner sep=2pt}]

\def\usx{2}
\def\wx{\usx+2}
\def\wy{0.5}
\def\h{-1}   
\def\y{0.3}
\def\bip{\wx+3.5}
\def\by{-1}

\node[draw = white, fill = white] (z) at (\usx,3.5*\y) {$L'(u^*)$};

\node[draw = white, fill = white] (z) at (\wx,\wy+4.2*\y) {$L'(w)$};

\node[draw = white, fill = white] (z) at (-1,\h-1.2*\y) {$L'(u)$};

\node[draw = white, fill = white] (z) at (1,\h-1.2*\y) {$L'(v)$};

\node[draw = white, fill = white, rectangle] (z) at (\wx+0.5,\wy+3*\y) {$1_w$};
\node[draw = white, fill = white, rectangle] (z) at (\wx+0.5,\y+\wy) {$3_w$};
\node[draw = white, fill = white, rectangle] (z) at (\wx+0.5,\wy+2*\y) {$2_w$};

\node (1u*) at (\usx,2*\y) {};
\node (2u*) at (\usx,\y) {};
\node (3u*) at (\usx,0) {};

\node (3u) at (-1,\h) {};
\node (2u) at (-1,\h+\y) {};
\node (1u) at (-1,\h+2*\y) {};

\node (3v) at (1,\h) {};
\node (2v) at (1,\h+\y) {};
\node (1v) at (1,\h+2*\y) {};

\node (3w) at (\wx,\y+\wy) {};
\node (2w) at (\wx,\wy+2*\y) {};
\node (1w) at (\wx,\wy+3*\y) {};

\draw (1v) to (1u*);
\draw (2v) to (2u*);
\draw (3v) to (3u*);

\draw [out = 180, in = 180,looseness = 2] (1u*) to (3u);
\draw [out = 180, in = 115,looseness = 0.5] (2u*) to (1u);
\draw [out = 180, in = 150,looseness = 1] (3u*) to (2u);

\draw (1u) -- (1v) -- (2u) -- (2v) -- (1u);
\draw (2u*) -- (2w) -- (3u*) -- (3w) -- (2u*);

  \draw[bend left=5] (3u) to (3v);
  \draw[bend right=5] (3u) to (3v);

  \draw[bend left=5] (1u*) to (1w);
  \draw[bend right=5] (1u*) to (1w);

\node[draw = white, fill = white] (z) at (\bip - 0.5,\by) {$3_{u^*}$};
\node[draw = white, fill = white] (z) at (\bip - 0.5,\by+2*\y) {$2_{u^*}$};
\node[draw = white, fill = white] (z) at (\bip - 0.5,\by+4*\y) {$2_{u^*}$};
\node[draw = white, fill = white] (z) at (\bip - 0.5,\by+6*\y) {$1_{u^*}$};
\node[draw = white, fill = white] (z) at (\bip - 0.5,\by+8*\y) {$1_{u^*}$};

\node[draw = white, fill = white, rectangle] (z) at (\bip + 3,\by+6*\y) {$(1_u,3_v)$};
\node[draw = white, fill = white, rectangle] (z) at (\bip + 3,\by+0*\y) {$(3_u,2_v)$};
\node[draw = white, fill = white, rectangle] (z) at (\bip + 3,\by+2*\y) {$(3_u,1_v)$};
\node[draw = white, fill = white, rectangle] (z) at (\bip + 3,\by+4*\y) {$(2_u,3_v)$};

\node[draw = white, fill = white, rectangle] (z) at (\bip, \by-2*\y) {$J$};
\node[draw = white, fill = white, rectangle] (z) at (\bip+2, \by-2*\y) {$K$};

\node (a5) at (\bip,\by) {};
\node (a4) at (\bip,\by+2*\y) {};
\node (a3) at (\bip,\by+4*\y) {};
\node (a2) at (\bip,\by+6*\y) {};
\node (a1) at (\bip,\by+8*\y) {};
\node (b4) at (\bip + 2, \by) {};
\node (b3) at (\bip + 2, \by+2*\y) {};
\node (b2) at (\bip + 2, \by+4*\y) {};
\node (b1) at (\bip + 2, \by+6*\y) {};

\draw (a1) -- (b1) -- (a2) -- (b2) -- (a1);
\draw (b2) -- (a3) -- (b3) -- (a5) -- (b4);
\draw (b2) -- (a4) -- (b3);

\end{tikzpicture}
\caption{A cover $(H',L')$ of the graph $G'$ in the proof of \Cref{lem:I5}, along with the bipartite graph $B$ constructed from $(H',L')$. For each $x \in V(G')$, $L'(x)$ is shown with $1_x$ on top and $3_x$ on the bottom. Here, as $1_{u^*}$ is the unique vertex of $L'(u^*)$ that does not belong to a $4$-cycle in $E_{H'}(L'(u^*), L'(w))$,
$j_{u^*} = 1_{u^*}$.
Each pair $(i_u, i'_v) \in K$ is adjacent exactly to those $c \in J$ for which $c$ is adjacent to neither $i_u$ nor $i'_v$ in $H'$.
}
\label{fig:house-bip}
\end{figure}
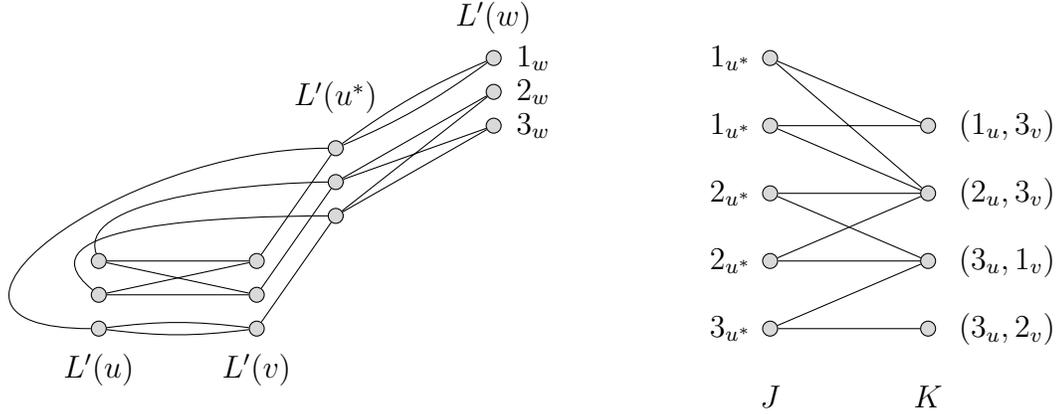

\begin{claim}
\label{claim:B-matching}
    $B$ has a matching $M$ that saturates $K$.
\end{claim}
\begin{proof}
We use Hall's theorem.
First, consider a pair $\mb i = (i_u, i'_v) \in K$.
Each color in $\mb i$ has a single neighbor in $L'(u^*)$,
and this neighbor appears in $J$ with multiplicity at most $2$;
therefore, $|N_B(\mb i)| \geq 5 - 2 \times 2 = 1$.

Next, consider two distinct pairs $\mb i, \mb i' \in K$.
We claim that $|N_B(\mb i) \cup N_B(\mb i')| \geq 3$.
For each $c \in L'(u^*)$,
 $c \not \in N_B(\mb i) \cup N_B(\mb i')$ 
 if and only if
$c$ is adjacent in $H'$ to a color of $\mb i$ and a color of $\mb i'$.
In particular, if $\mb i \cap \mb i' \neq \emptyset$, then at most one color $c \in L'(u^*)$ is adjacent to a color from both $\mb i$ and $\mb i'$,
and then $|N_B(\mb i) \cup N_B(\mb i')| \geq 5 - 1 \times 2 = 3$.
Otherwise, suppose $\mb i \cap \mb i' = \emptyset$,
 so that without loss of generality, $3_u \in \mb i$ and $3_v \in \mb i'$.
As no color of $L'(u^*)$ is adjacent in $H'$ to both $3_u$ and $3_v$,
only a single color $c \in L'(u^*)$
is adjacent to a color in $\mb i$ and a color in $\mb i'$.
Therefore, again,
$|N_B(\mb i) \cup N_B(\mb i')| \geq 5-1 \times 2 = 3$.

Next, 
let $\mb i, \mb i', \mb i'' \in K$ be three distinct pairs, and suppose that
for some $c \in L'(u^*)$,
$c \not \in N_B(\mb i) \cup N_B(\mb i') \cup N_B(\mb i'')$.
Then, it follows that $c$ has a neighbor via $H'$ in each of $\mb i, \mb i', \mb i''$.
Without loss of generality, $\mb i \cap \mb i' = \{3_v\}$, so that $3_v \in N_{H'}(c)$.
As $3_u \not \in N_{H'}(c)$ by our construction of $H'$, it follows that no color of $\mb i''$ is a neighbor of $c$ in $H'$, a contradiction.
Therefore, $|N_B(\mb i) \cup N_B(\mb i') \cup N_B(\mb i'')| = |J| = 5$.
This also implies that $|N_B(K)| = |J| = 5$.
Thus, $B$ has a matching $M$ that saturates $K$ by Hall's theorem.
\end{proof}

Now, we construct a distribution on $(H',L')$-colorings $\phi'$ of $G'$ as follows.
Using \Cref{claim:B-matching},
we find a matching $M$ in $B$ that saturates $K$.
We 
obtain a subgraph $M'$ of $B$ by 
adding one extra edge to $M$ so that $d_{M'}(j) = 1$ for each $j \in J$, and we observe that $d_{M'}(\mb i) \geq 1$ for each $\mb i \in K$.
Then, we choose an element $c \in J$ uniformly at random (accounting for multiplicity), and we assign $\phi'(u^*) = c$.
Next,
we let $\mb i = (i_u, i'_v)$ be the unique
element of $N_{M'}(c)$,
and we assign $\phi'(u) = i_u$ and $\phi'(v) = i'_v$.
Next, we choose a color $c' \in L'(w)$ uniformly at random from the non-neighbors of $c$ in $H'$, and we assign $\phi'(w) = c'$.
Finally, we let $\phi = f^{-1}(\phi')$ and observe that $\phi$ is an $(H,L)$-coloring of $G$.

We claim that the random process for producing $\phi$ gives us a $\frac 15$-distribution for $(G,\rho,H,L,\mathcal D)$.
For this, it suffices to show that for each $x \in V(G)$ and $i_x \in L(x)$, $\Pr(\phi(x) = i_x) \geq \frac 15$.
When $x \in \{u',v'\}$,
this follows from the fact that each color of $L'(u^*)$ is assigned as $\phi'(u^*)$ with probability in $\{\frac 15, \frac 25\}$.
When $x  \in \{u,v\}$,
this follows 
from the fact that $i_x$ belongs to some pair $\mb i \in K$,
which is used to define $\phi'$ at $(u,v)$
with probability at least $\frac 15$.
Finally, we consider the case when $x = w$.
We consider three cases, based on the structure of $H'[L(u^*) \cup L(w)]$.
\begin{enumerate}
\item 
If $H'[L(u^*), L(w)]$  consists of three copies of $C_2$,
then $\phi'$ assigns the neighbor of $i_x$ to $u^*$
with probability at most $\frac 25$, so $\Pr(\phi(x) = i_x) \geq \frac{1 - 2/5}{2} = \frac{3}{10}$.
\item 
If $H'[L(u^*), L(w)]$ 
is a $C_6$,
then $\phi(x) = i_x$ exactly when 
$\phi'$ assigns the unique non-neighbor of $i_x$ in $L'(u^*)$ to $u^*$, which occurs with probability at least $\frac 15$.
\item 
Finally, suppose that $H'[L(u^*), L(w)]$  consists of a $C_2$-component and a $C_4$-component.
If $i_x$ belongs to the $C_2$-component,
then $\phi(x) = i_x$ exactly when a non-neighbor of $i_x$ in $L'(u^*)$ is assigned to $u^*$,
which occurs with probability at least $\frac 35$.
If $i_x$ belongs to the $C_4$-component, then $\phi(x) = i_x$ with conditional probability $\frac 12$ whenever $\phi'(u^*) = j_{u^*}$;
 recall that $j_{u^*}$ is the unique non-neighbor of $i_x$ in $L'(u^*)$.
By construction, $j_{u^*}$ appears in $J$ with multiplicity $2$ and hence is assigned by $\phi'$ to $u^*$ with probability $\frac 25$.
Therefore, $\Pr(\phi(x) = i_x) = \frac 15$.
\end{enumerate}
Hence, $(G,\rho,H,L,\mathcal D)$
has a $\frac 15$-distribution, and this contradiction proves the lemma.
\end{proof}

\begin{lemma} 
\label{lem:kck}
    No two $3$-vertices of $ \Pi_6$ are joined by parallel edges.
\end{lemma}

\begin{proof}
    Suppose that $u, v \in  \Pi_6$ are $3$-vertices joined by parallel edges. 
    By \Cref{lem:double-edges}, $|E_G(u,v)| = 2$.
    Let $u'$ be the unique neighbor of $u$ apart from $v$, and let $v'$ be the unique neighbor of $v$ apart from $u$.    
    If $u' = v'$, then $G$ contains a copy of $I_1$, contradicting \Cref{obs:no-exception}; therefore, we assume that $u' \neq v'$.
    Let $G'$ be obtained from $G$ by deleting $u$ and $v$ and adding the edge $u'v'$. 
    Subject to the conditions above, let $u$ and $v$ be chosen so that $G'$ has no subgraph in $\mathcal I$ if possible. 
    Obtain $H'$ from $H$ by deleting $L(u)\cup L(v)$ and adding edges $i_{u'}i_{v'}$ for $i\in\{1,2,3\}$.

    By \Cref{lem:3-parallel-method},
    $(G',\rho,H',L,\mathcal D)$ has no $\frac 15$-distribution.
    Therefore, 
    for some $U \subseteq V(G')$, $(G'[U],\rho)$ has a spanning $(\rho,\frac 15)$-critical subgraph.
    Thus,
    either $\rho_{G'}(U) \leq 0$,
    or $\rho_{G'}(U) \geq 1$ and $(G'[U],\rho)$ has a spanning exceptional subgraph.
    We consider several cases.
    \begin{enumerate}
        \item Suppose that $\rho_{G'}(U) \leq 0$.
        Then, $\rho_G(U \cup \{u,v\}) = \rho_{G'}(U) - 4 + \rho_G(\{u,v\}) = \rho_{G'}(U) \leq 0$, contradicting \Cref{lem:j+2}.
        \item Suppose that $\rho_{G'}(U) \geq 1$ and $G'[U]$ has a spanning subgraph $I_m \in \mathcal I$.
    If $|E_{G'}(u',v')| = 1$, 
    then
    $G$ contains a copy of $I_{m+1}$ by \Cref{obs:Im+1}, which contradicts \Cref{obs:no-exception}. 
    So we may assume $|E_{G'}(u',v')| = 2$.
    By \Cref{obs:Ilkyoo},
    $U \subseteq \Pi_6$, which implies that $G'[U] = I_m$.
    Furthermore, $\rho_G(U \cup \{u,v\}) = \rho_{G'}(U) - 4 + \rho_G(\{u,v\}) = \rho_{G'}(U) = 2$,
    so \Cref{lem:potential-2} implies that $V(G) = U \cup \{u,v\}$.
    Hence, $V(G) = \Pi_6$.

    If  $m=1$, then $G$ is isomorphic to $H_5$, contradicting  \Cref{lem:I5}.
    Therefore,  $m \geq 2$, so $I_m$ has $2m + 1 \geq 5$ vertices and $m$ pairs of parallel edges. Let $x,y$ be two vertices disjoint from $u, v$  that are joined by parallel edges in $G'$. 
    Clearly, $d_G(x) = d_G(y) = 3 $,
    and $|E_G(x,y)| = 2$.
    Writing $x',y'$ for the respective neighbors of $x,y$ in $V(G) \setminus \{x,y\}$,
    we consider the graph $G''$ obtained 
    from $G$
    by deleting $x,y$ and adding the edge $x'y'$.
    We note that $\rho_{G''}(V(G'')) = 2$, and in particular, all vertices of $G''$ are in $\Pi_6$.
    Also,
    for each proper subset $U'' \subsetneq V(G'')$ containing $x'$ and $y'$,
    $\rho_{G''}(U'') = \rho_G(U'' \cup \{x,y\}) \geq 4$ by  \Cref{lem:potential-2}.
    Furthermore, for each proper subset $U '' \subsetneq V(G'')$ not containing both $x'$ and $y'$,
    $\rho_{G''}(U'') = \rho_G(U'') \geq 4$
    by \Cref{lem:potential-2}.
    Therefore, no proper subgraph of $G''$
    belongs to $\mathcal I$.
    Since $G''$ contains the $4$-cycle $uvv'u'$, 
    it also holds that $G'' \not \in \mathcal I$.
    Therefore, $G''$ has no subgraph in $\mathcal I$. 
    Thus, the existence of the pair $x,y$ contradicts the choice of $u,v$. %, and we assume that $G'$ has no subgraph in $\mathcal I$.
    \item Finally, suppose that $(G'[U], \rho)$ is an exceptional $C_2$.
    Then, $\rho_G(U \cup \{u,v\}) = \rho_{G'}(U) - 4 + \rho_G(\{u,v\}) = 2$, and hence \Cref{lem:potential-2} implies that
    $V(G) = U \cup \{u,v\}$.
    Thus, $G$ is isomorphic to $H_4$, and without loss of generality, $\Pi_6 = \{u',u,v\}$ and $\Pi_4 = \{v'\}$.
    Furthermore, as $(G',\rho,H',L,\mathcal D)$ has no $\frac 15$-distribution,    
    we may assume without loss of generality
    that $\omega(v') = 1_{v'}$ and that
    $H'$ contains the edges $1_{u'} 2_{v'}, 2_{u'} 1_{v'}$ (as noted after Observation \ref{obs:C2}).
    %{\IC(to be pedantic, \Cref{lem:C2C4} says this for subgraphs of $G$ only. is it  okay?)}
    %{\PB This is a good point.
    %What I really want to say is that there is only one way that an exceptional $C_2$ will forbid a distribution, and it involves this kind of cover. Maybe I should try harder to make this clear.
    %In fact, maybe it's worth pointing this out in the observation about why this is an exception.
    %}
     Recall that $H'$ contains the edges $i_{u'},i_{v'}$ for each $i \in \{1,2,3\}$.
    Thus, $H$ contains the edges $1_{u'} 2_{v'}, 2_{u'} 1_{v'}, 3_{u'} 3_{v'}$. Then, we find a $\frac 15$-distribution for $(G',\rho,H',L',\mathcal D)$
    by assigning $G$ 
    an $(H,L)$-coloring $\phi$ chosen uniformly at random from the following:
    $\phi_1 = (3_{u'}, 1_u, 3_v, 1_{v'})$,
    $\phi_2 = (1_{u'}, 2_u, 3_v, 1_{v'})$,
    $\phi_3 = (1_{u'}, 3_u, 2_v, 1_{v'})$,
    $\phi_4 = (1_{u'}, 3_u, 2_v, 1_{v'})$,
    $\phi_5 = (3_{u'}, 1_u, 3_v, 2_{v'})$,
    $\phi_6 = (1_{u'}, 2_u, 3_v, 2_{v'})$,
    $\phi_7 = (2_{u'}, 3_u, 1_v, 2_{v'})$,
    $\phi_8 = (2_{u'}, 3_u, 2_v, 3_{v'})$,
    $\phi_9 = (2_{u'}, 3_u, 1_v, 3_{v'})$,
    $\phi_{10} = (2_{u'}, 3_u, 2_v, 3_{v'})$.
    It is easy to check that this random assignment gives a $\frac 15$-distribution; in particular, $\Pr(\phi(v') = 2_{v'}) = \Pr(\phi(v') = 3_{v'}) = \frac 3{10}$. Thus, we have a contradiction.
    \end{enumerate}
\end{proof}

\begin{lemma}
\label{lem:3-parallel}
    No two $3^-$-vertices are joined by parallel edges.
\end{lemma}
\begin{proof}
    Suppose that two $3^-$-vertices $u,v$ are joined by parallel edges. 
    By \Cref{lem:weak-parallel}, 
    $u, v\in\Pi_{\geq 4}$. 
 By \Cref{lem:full-2-parallel}, each of $u$ and $v$ has degree  $3$. 
By \Cref{lem:kck}, $\{u, v\}\not\subseteq\Pi_6$.
 Now, either $\rho(\{u, v\})= 4+6-2(4)=2$, so there is an exceptional subgraph, or $\rho(\{u, v\})=4+4-2(4)=0$, which is a contradiction to \Cref{lem:j+2}.
\end{proof}

\begin{lemma}
\label{lem:deg2-parallel}
    No $2$-vertex
    is incident with parallel edges.
\end{lemma}
\begin{proof}
    This follows directly from Lemmas \ref{lem:weak-parallel} and \ref{lem:full-2-parallel}.
\end{proof}

\begin{lemma}
\label{lem:C2C2}
    Each $v \in V(G)$ is incident with at most one pair of parallel edges.
\end{lemma}
\begin{proof}
    Suppose that $v \in V(G)$ has neighbors $u,w$ for which $|E_G(v,u)| = |E_G(v,w)| =2$.
    If $\{u, v, w\}\not\subseteq\Pi_6$, then $\rho(\{u,v,w\})\leq 6+6+4-4(4)=0$, which contradicts \Cref{lem:j+2}. 
    Otherwise, 
    $\{u,v,w\} \subseteq  \Pi_6$, so that 
    $V(G) = \{u,v,w\}$ by  \Cref{lem:potential-2}. Then, $d(u) = 2$, contradicting \Cref{lem:deg2-parallel}.
\end{proof}

\section{ Discharging }
\subsection{Setup}
We aim to use discharging to show that $\rho_G(V(G)) \leq 0$, leading to a contradiction.
Recall that for each $v \in V(G)$, $\sigma(v) = \rho(v) - 2d(v)$ and $\sum_{v\in V(G)}\sigma(v) = \rho_{G}(V(G)) \geq 1$.
We let each vertex $v$ begin with charge $\sigma(v)$,
and then we aim to redistribute the charge in a way that shows that $\sum_{v \in V(G)} \sigma(v) \leq 0$, giving us a contradiction.

Call a vertex \emph{nonnegative} (resp., \emph{positive}) if $\sigma(v)$ is
 nonnegative (resp., positive). 
  Recall that by \Cref{lem:pendent-heavy}, there is no $1$-vertex in $\Pi_{\geq4}$.
 The positive vertices are
 $2$-vertices $v \in \Pi_6$ and $1$-vertices $v \in \Pi_3$. 
 Hence, every positive vertex $v$ satisfies $\sigma(v) \leq 2$.
A \emph{conductive} vertex is a vertex $v$ for which $d(v) \in \{2,3\}$ and $\rho(v) = 2d(v)$, or $d(v) = 2$ and $\rho(v) = 3$.
Note that every vertex $v$ for which $\sigma(v) = 0$ is conductive.
We define a \emph{conductive path} as a path whose internal vertices are all conductive. 
Note that two adjacent vertices also form a conductive path.

\begin{lemma}
\label{lem:pos-con}
    No two positive
    vertices are conductively connected.
\end{lemma}
\begin{proof}
    Suppose that $u$ and $v$ are two distinct positive vertices. 
    Let $T$ be a shortest conductive path in $G$ joining $u$ and $v$, and observe that $T$ is induced without 
    parallel edges by \Cref{lem:3-parallel}.
    Furthermore, each $w \in V(T)$ has at most one
    neighbor
    in $V(G) \setminus V(T)$,
    and if $\rho(w) \leq 4$, then $w$ has no
    neighbor in $V(G) \setminus V(T)$.
    Therefore, $T$ satisfies the conditions of \Cref{lem:reducible-tree}.
    As $T$ is a tree, we have a contradiction.
\end{proof}

\begin{lemma}
\label{lem:one-positive}
    $G$ has at least two positive vertices.
\end{lemma}
\begin{proof}
If $G$ has no positive vertex, then $\rho_G(V(G))\leq 0$, contradicting  \Cref{lem:j+2}.
Suppose that $G$ has exactly one positive vertex $p$. Note that $\sigma(p) = \rho(p) -  2 d(p) \leq 2$.
By \Cref{lem:j+2}, $2 \geq  \sum_{v \in V(G)} \sigma(v) =\rho_{G}(V(G)) \geq 1$.

    If $\sigma(p) = 1$, 
    then $\sigma(v) = 0$ for all $v \in V(G) \setminus \{p\}$; then, $\Pi_3 = \{p\}$, contradicting \Cref{lem:unique-weak}.
    Thus, $\sigma(p)=2$ so $p$ is a $2$-vertex in $\Pi_6$.
    Therefore, at most one vertex $x$ satisfies $\sigma(x) = -1$, and all other vertices $v \in V(G) \setminus \{p,x\}$ satisfy $\sigma(v) = 0$ and hence are conductive.
    If $x$ exists,  then, $\Pi_3 = \{x\}$, and \Cref{lem:unique-weak} is violated.
    Therefore, $x$ does not exist, and 
    $V(G) \subseteq \Pi_{\geq 4}$.
    Furthermore, $G$ is subcubic,
     and by \Cref{lem:3-parallel}, $G$ is simple.
    By \Cref{lem:term-4}, $G$ is $2$-connected, and by \Cref{lem:potential-2}, each $U \subsetneq V(G)$ satisfies $\rho_G(U) \geq 4$.
    We consider two cases.
    \begin{enumerate}
        \item 
    Suppose $p$ is the unique $2$-vertex of $G$.
    Then, every other vertex $v \in V(G) \setminus \{p\}$ satisfies $d(v) = \frac 12 \rho(v) = 3$.
    In particular, $V(G) \subseteq \Pi_6$.
    Let $G' = G-p$, and let $B$ be a terminal block of $G'$. 
    As $B$ has minimum degree $2$, $B$ is $2$-connected.
    If no vertex of $B$ is adjacent to $p$, then $B$ is a terminal block of $G$, a contradiction; therefore, $B$ has a neighbor $p'$ of $p$.
    Letting $w$ be the cut-vertex of $B$ in $G'$, we see that no vertex of $V(B) \setminus \{p',w\}$ has a neighbor in $V(G) \setminus V(B)$.
    Furthermore, as $G$ is $2$-connected, $p' \neq w$.
    Therefore, $\rho_G(V(B)) = \sigma_B(p') + \sigma_B(w) = 4$, and $B$
    contains two  $2$-vertices. Therefore, $B$ has a fractional $(H,L)$-packing by \Cref{lem:red-3}, contradicting  \Cref{lem:gen-diamond}.
    
    \item On the other hand, suppose that there is some $2$-vertex in $V(G) \setminus \{p\}$.
    As $V(G) \subseteq \Pi_{\geq 4}$ and $p$ is the unique positive vertex of $G$, this implies that every $v \in V(G) \setminus \{p\}$ satisfies $\rho(v) = 2d(v)$, and $|\Pi_4| \geq 1$.
    Let $w \in \Pi_4$ be a closest vertex of $\Pi_4$ to $p$, and let $P$ be a shortest path from $p$ to $w$.
    Observe that $V(P) \cap \Pi_4 = \{w\}$.
    Without loss of generality, $\omega(w) = 1_w$.
    Write $w'$ for the unique neighbor of $w$ in $V(G) \setminus V(P)$.

    Now, let $G' = G - V(P)$.
    We observe that since each $U \subseteq V(G')$ satisfies $\rho_G(U) \geq 4$,
    each component of $G'$ satisfies \Cref{lem:path}.
    Therefore, there exists a $\frac 15$-distribution $\Phi'$ for $(G',\rho,H,L,\mathcal D)$
    such that for each $v \in V(G')$ and $c \in L(v)$, $\Pr(\phi'(v) = c) \in [\frac 3{10}, \frac 25]$.
    We construct a $\frac 15$-distribution $\Phi$ for $(G,\rho,H,L,\mathcal D)$ as follows.
    First, we sample an $(H,L)$-coloring $\phi'$ of $G'$ using $\Phi'$, and we let $\phi$ agree with $\phi'$ on $G'$. 
    For each $i \in \{1,2,3\}$, let $p_i = \Pr(\phi'(w') = i_{w'})$.
    Also, define $q_i = 10p_i - 3$, and observe that $q_1 + q_2 + q_3 = 1$ and $q_i \in [0,1]$.

    Next, for each $v \in V(P)$,
    let $L'(v) = L(v) \setminus N_H(\phi')$.
    In particular, $|L'(v)| = 2$ for each $v \in V(P)$.
    Next, use \Cref{lem:tree-packing}
    to find
    two disjoint $(H,L')$-colorings $\phi_1, \phi_2$ of $P$.
    Finally, writing $\phi'(w') = i_{w'}$,
    we choose an index $j \in \{1,2,3\}$
    with conditional probability equal to entry $A_{ji}$ in the following matrix $A$:
    \[
        \frac 1{10} \begin{bmatrix} 0 & \frac{2}{p_2} & \frac{2}{p_3} \\
        \frac{1 + q_1 + q_2}{p_1} &0 & \frac{1 + q_3}{p_3} \\
        \frac{1 + q_1 + q_3}{p_1} & \frac{1+q_2}{p_2} & 0 
    \end{bmatrix}.
    \]
    Note that in each column $i$, each entry is at least $\frac{1+q_i}{10 p_i} = \frac{p_i-\frac 15}{p_i} = 1 - \frac{1}{5p_i} \geq 1 - \frac{1}{5(3/10)} = \frac 13$.

    Having chosen $j$, we let $\phi$ agree with the coloring $\phi^* \in \{\phi_1, \phi_2\}$ for which $\phi^*(w) = j_w$.

    Now, 
    we claim that $\Phi$ is a $\frac 15$-distribution for $(G,\rho,H,L,\mathcal D)$.
    First, for each $v \in V(G) \setminus V(P)$ and $c \in L(v)$, 
    by our assumption on $\Phi'$,
    $\Pr(\phi(v) = c) = \Pr(\phi'(v) = c) \geq \frac 15$, and when $v \in \Pi_4$,
    $\Pr(\phi(v) = c) = \Pr(\phi'(v) = c) = \frac 3{10}$ for each $c \in L(v) \setminus \{\omega(v)\}$.
    Next,
    we observe $V(P) \setminus \{w\} \subseteq \Pi_6$.
    By our choice of $\Phi'$,
    for each $v \in V(P) \setminus \{w\}$ and $c \in L(v)$,
    $\Pr(c \in L'(v)) \geq \frac 35$,
    and $\phi(v) = c$ with subsequent conditional probability at least $\frac 13$;
    therefore, $\Pr(\phi(v) = c) \geq \frac 15$.
    Finally,
    writing $\mb p = [p_1,p_2,p_3]^T$,
    we see that 
    the colors $(1_w, 2_w, 3_w)$ are assigned to $w$ according to the probability vector
     $A \mb p = [2/5,3/10,3/10]^T$.
     As $\omega(w) = 1_w$,
     $\Phi$ is a $\frac 15$-distribution for $(G,\rho,H,L,\mathcal D)$, a contradiction.
    \end{enumerate}    
\end{proof}

A vertex $w$ 
is
\emph{insulated} if $\sigma(w) \leq -2$. 
 By our definition of $\sigma$, 
 \begin{equation}\label{nin0}
 \mbox{
 each non-insulated negative vertex $v$  has $\sigma(v)=-1$
 and thus is a  $2$-vertex in $\Pi_3$. }   
\end{equation}
 By \Cref{lem:pos-con},
\begin{equation}\label{nin}
 \mbox{each   $2$-vertex in $\Pi_3$ is conductively connected with at most one positive vertex.}   
\end{equation}
 By \Cref{lem:weak-deg},
\begin{equation}\label{insul}
 \parbox{14cm}{no insulated vertex is in $\Pi_3$, each insulated vertex in $\Pi_4$ has degree $3$, and each insulated vertex in $\Pi_6$ has degree $4$ or $5$. }   
\end{equation}

We also observe that every vertex of $G$ is insulated, conductive, or positive.
Indeed, if $\sigma(v) \leq -2$, then $v$ is insulated. If $\sigma(v) = -1$, then $\rho(v) = 3$ and $d(v) = 2$, so that $v$ is conductive. If $\sigma(v) = 0$, then $\rho(v) = 6$ and $d(v) = 3$, or $\rho(v) = 4$ and $d(v) = 2$,
so that $v$ is conductive.
Finally, if $\sigma(v) \geq 1$, then $v$ is positive.

\begin{lemma}
\label{lem:positive-insulated}
Every positive vertex is conductively connected with an insulated vertex.
\end{lemma}
\begin{proof}
Suppose that $p$ is a positive vertex that is not conductively connected with an insulated vertex. 
As $p$ is not conductively connected with another positive vertex by  \Cref{lem:pos-con}, $V(G)$ consists entirely of $p$ and conductive vertices, contradicting  \Cref{lem:one-positive}.
\end{proof}

\begin{lemma} 
\label{lem:expensive-cut}
    If $p \in V(G)$ satisfies $\sigma(p) = 2$, 
    then $p$ is conductively connected with at least two negative vertices.
\end{lemma}
\begin{proof}  
    Suppose that $\sigma(p) = 2$.
    By \Cref{lem:positive-insulated},
    $p$ is conductively connected with at least one insulated vertex $w \in V(G)$.
    Note  that $w \in \Pi_{\geq 4}$.
     Let $U$ be the set of vertices that are conductively connected with $p$.

    Suppose that $w$ is the only negative vertex conductively connected with $w$.    
    Then, every vertex $v \in U\setminus\{p, w\}$ satisfies $\sigma(v) = 0$, and in particular, $U \subseteq \Pi_{\geq 4}$.
    Note that $p,w \in U$.
    If
    $U = V(G)$, then, $\rho_G(V(G)) = \sigma(p) + \sigma(w) \leq 0$, a contradiction.
    Therefore, $U \neq V(G)$, and by the definition of conductive connectivity, $w$ is a cut-vertex of $G$.
    Then, some terminal block of $G$ is a subset of $U$, contradicting  \Cref{lem:term-4}. 
\end{proof}

\begin{lemma} 
\label{lem:pos-UB}
    If $w \in V(G)$ is insulated, then $w$ is conductively connected with at most $2 d(w) - \rho(w) = -\sigma(w)$ positive vertices.
\end{lemma}
\begin{proof}
Let $Q$ be the set of all positive vertices  that are conductively connected with $w$,  and suppose that $|Q| \geq 2d(w) - \rho(w) + 1$. For each
$q \in Q$, let $U_q$ be the  set of vertices on some shortest conductive path joining $w$ and $q$. 
Then, let $U = \bigcup_{q \in Q} U_q$.
As each $U_q$ is chosen to be shortest, $G[U_q]$ is an induced path in $G$ for each $q \in Q$.
As no
conductive path joins two positive vertices by \Cref{lem:pos-con},
$G[U]$ is an induced tree $T$ in $G$.

We will show that each vertex $v$ of $T$ satisfies $d_{V(G)\setminus T}(v)\leq 1$, and if $v\in \Pi_{\leq 4}$, then $d_{V(G)\setminus T}(v)=0$.
First, if $v$ is a leaf of $T$, then since $v$ is positive,
$(d_G(v), \rho(v))\in\{(1,3),(2,6)\}$. 
Thus, $d_{V(G)\setminus T}(v)\leq 1$, and if $v\in \Pi_3$, then $d_{V(G)\setminus T}(v)=0$.
As $d_T(w)=|Q| \geq 2d_G(w) - \rho(w) + 1$,
\[d_{V(G) \setminus T}(w) = d_G(w) - d_T(w) \leq -d_G(w) + \rho(w) - 1.\]
Since $w$ is insulated, $(d_G(v), \rho(v))\in\{(3,4),(4,6),(5,6)\}$, so $d_{V(G)\setminus T}(w)\leq1$, and if $w\in\Pi_4$, then $d_{V(G)\setminus T}(v)=0$.
For each $v\in V(T)\setminus (Q\cup\{w\})$, $d_T(v) = 2$,
so
\[
d_{V(G) \setminus T}(v) = d_G(v) - 2 \leq 1
\]
by the conductivity of $v$,
and when $\rho(v) \leq 4$, $d_{V(G) \setminus T}(v) = 0$.
Therefore, the set $U$ satisfies all conditions of
\Cref{lem:reducible-tree}.
As $G[U]$ is a tree, we have a contradiction.
\end{proof}

\subsection{The procedure}

%Now, we proceed with discharging as follows.
Let each $v \in V(G)$ begin with charge $\sigma(v)$.
Then, we carry out the following:
\begin{quote}
    If $\sigma(v) \geq 1$, then $v$ gives charge $1$ to each negative vertex conductively connected with $v$.
\end{quote}
We claim that after this step, each vertex of $G$ has a nonpositive charge, as follows:
\begin{itemize}
    \item Suppose that $\sigma(v) = 2$.
    By \Cref{lem:expensive-cut}, $v$ is conductively connected with at least two negative vertices. Therefore, $v$ gives away charge at least $2$, for a final charge of at most $0$.
    \item Suppose that $\sigma(v) = 1$, so that $\rho(v) = 3$ and $d(v) = 1$. Then $v$ is conductively connected with a negative vertex by \Cref{lem:positive-insulated}. Therefore, $v$ gives away at charge least $1$, for a final charge of at most $0$.
    \item Suppose that $\sigma(v) = 0$, so that $v$ is conductive. Then $v$
    neither gives nor receives charge, so the final charge of $v$ is  $0$.
    \item Suppose that $\sigma(v) = -1$, so that $\rho(v) = 3$ and $d(v) = 2$.
    Then, as $v$ is conductive,
     \Cref{lem:pos-con} implies that
     $v$ is conductively connected with at most one positive vertex and thus receives charge at most $1$, for a final charge of at most $0$.
    \item Suppose that $\sigma(v) \leq -2$, so that $v$ is insulated.
    Then, by \Cref{lem:pos-UB},
    $v$ is conductively connected with at most $2d(v) - \rho(v)$ positive vertices.
    Hence, as $v$ receives charge $1$ from each conductively connected positive vertex,
    the final charge of $v$ is at most $\sigma(v) + 2d(v) - \rho(v) = 0$.
\end{itemize}
Since each vertex $v \in V(G)$ satisfies $\sigma(v) \leq 2$, these cases are exhaustive. Therefore, since each vertex ends the discharging process with nonnegative charge, and  no charge is created or destroyed during discharging, we have that 
\[\rho_G(V(G)) = \sum_{v \in V(G)} \sigma(v) \leq 0.\]
Therefore, $G$ is not a counterexample to \Cref{thethm}. This final contradiction completes the proof.

\bibliographystyle{abbrvurl}
{\small
\bibliography{ref}}

\begin{thebibliography}{10}

\bibitem{arXiv_2024BiBr}
R.~Bi and P.~Bradshaw.
\newblock Graphs of maximum average degree less than $\frac {11}{3}$ are flexibly $4$-choosable, 2024.
\newblock \href {https://arxiv.org/abs/2408.08393} {\path{arXiv:2408.08393}}.

\bibitem{arXiv_2025BiBr}
R.~Bi and P.~Bradshaw.
\newblock Flexible list coloring of graphs with maximum average degree less than $3$, 2025.
\newblock \href {https://arxiv.org/abs/2310.02979} {\path{arXiv:2310.02979}}.

\bibitem{2022BrMaSt}
P.~Bradshaw, T.~Masa\v{r}\'ik, and L.~Stacho.
\newblock Flexible list colorings in graphs with special degeneracy conditions.
\newblock {\em J. Graph Theory}, 101(4):717--745, 2022.
\newblock \href {https://doi.org/10.1002/jgt.22849} {\path{doi:10.1002/jgt.22849}}.

\bibitem{2024CaCaDaKa}
S.~Cambie, W.~Cames~van Batenburg, E.~Davies, and R.~J. Kang.
\newblock List packing number of bounded degree graphs.
\newblock {\em Combin. Probab. Comput.}, 33(6):807--828, 2024.
\newblock \href {https://doi.org/10.1017/s0963548324000191} {\path{doi:10.1017/s0963548324000191}}.

\bibitem{CaCa2024}
S.~Cambie and W.~C. van Batenburg.
\newblock Fractional list packing for layered graphs, 2024.
\newblock \href {https://arxiv.org/abs/2410.02695} {\path{arXiv:2410.02695}}.

\bibitem{arXiv_CaCaZh2023}
S.~Cambie, W.~C. van Batenburg, and X.~Zhu.
\newblock Disjoint list-colorings for planar graphs, 2023.
\newblock \href {https://arxiv.org/abs/2312.17233} {\path{arXiv:2312.17233}}.

\bibitem{2022ChClFeHoMaMa}
I.~Choi, F.~C. Clemen, M.~Ferrara, P.~Horn, F.~Ma, and T.~Masa\v{r}\'ik.
\newblock Flexibility of planar graphs---sharpening the tools to get lists of size four.
\newblock {\em Discrete Appl. Math.}, 306:120--132, 2022.
\newblock \href {https://doi.org/10.1016/j.dam.2021.09.021} {\path{doi:10.1016/j.dam.2021.09.021}}.

\bibitem{2020DvMaMuPa}
Z.~Dvo\v{r}\'ak, T.~Masa\v{r}\'ik, J.~Mus\'ilek, and O.~Pangr\'ac.
\newblock Flexibility of planar graphs of girth at least six.
\newblock {\em J. Graph Theory}, 95(3):457--466, 2020.
\newblock \href {https://doi.org/10.1002/jgt.22567} {\path{doi:10.1002/jgt.22567}}.

\bibitem{2021DvMaMuPa}
Z.~Dvo\v{r}\'ak, T.~Masa\v{r}\'ik, J.~Mus\'ilek, and O.~Pangr\'ac.
\newblock Flexibility of triangle-free planar graphs.
\newblock {\em J. Graph Theory}, 96(4):619--641, 2021.
\newblock \href {https://doi.org/10.1002/jgt.22634} {\path{doi:10.1002/jgt.22634}}.

\bibitem{2019DvNoPo}
Z.~Dvo\v{r}\'ak, S.~Norin, and L.~Postle.
\newblock List coloring with requests.
\newblock {\em J. Graph Theory}, 92(3):191--206, 2019.
\newblock \href {https://doi.org/10.1002/jgt.22447} {\path{doi:10.1002/jgt.22447}}.

\bibitem{2018DvPo}
Z.~Dvo\v{r}\'{a}k and L.~Postle.
\newblock Correspondence coloring and its application to list-coloring planar graphs without cycles of lengths 4 to 8.
\newblock {\em J. Combin. Theory Ser. B}, 129:38--54, 2018.
\newblock \href {https://doi.org/10.1016/j.jctb.2017.09.001} {\path{doi:10.1016/j.jctb.2017.09.001}}.

\bibitem{1980ErRuTa}
P.~Erd\H{o}s, A.~L. Rubin, and H.~Taylor.
\newblock Choosability in graphs.
\newblock In {\em Proceedings of the {W}est {C}oast {C}onference on {C}ombinatorics, {G}raph {T}heory and {C}omputing ({H}umboldt {S}tate {U}niv., {A}rcata, {C}alif., 1979)}, volume XXVI of {\em Congress. Numer.}, pages 125--157. Utilitas Math., Winnipeg, MB, 1980.

\bibitem{1978Fisk}
S.~Fisk.
\newblock The nonexistence of colorings.
\newblock {\em J. Combinatorial Theory Ser. B}, 24(2):247--248, 1978.
\newblock \href {https://doi.org/10.1016/0095-8956(78)90028-x} {\path{doi:10.1016/0095-8956(78)90028-x}}.

\bibitem{1958Grotzsch}
H.~Gr\"otzsch.
\newblock Zur {T}heorie der diskreten {G}ebilde. {VII}. {E}in {D}reifarbensatz f\"ur dreikreisfreie {N}etze auf der {K}ugel.
\newblock {\em Wiss. Z. Martin-Luther-Univ. Halle-Wittenberg Math.-Natur. Reihe}, 8:109--120, 1958/59.

\bibitem{book_2011Kaibel}
V.~Kaibel.
\newblock {\em Basic Polyhedral Theory}.
\newblock John Wiley \& Sons, Ltd, 2011.
\newblock \href {https://doi.org/10.1002/9780470400531.eorms0091} {\path{doi:10.1002/9780470400531.eorms0091}}.

\bibitem{2024KaMaMuPe}
H.~Kaul, R.~Mathew, J.~A. Mudrock, and M.~J. Pelsmajer.
\newblock Flexible list colorings: maximizing the number of requests satisfied.
\newblock {\em J. Graph Theory}, 106(4):887--906, 2024.
\newblock \href {https://doi.org/10.1002/jgt.23103} {\path{doi:10.1002/jgt.23103}}.

\bibitem{2022LiMaMuZe}
B.~Lidick\'y, T.~Masa\v{r}\'ik, K.~Murphy, and S.~Zerbib.
\newblock On weak flexibility in planar graphs.
\newblock {\em Graphs Combin.}, 38(6):Paper No. 180, 33, 2022.
\newblock \href {https://doi.org/10.1007/s00373-022-02564-1} {\path{doi:10.1007/s00373-022-02564-1}}.

\bibitem{2019Masarik}
T.~Masa\v{r}\'ik.
\newblock Flexibility of planar graphs without 4-cycles.
\newblock {\em Acta Math. Univ. Comenian. (N.S.)}, 88(3):935--940, 2019.

\bibitem{1976Vizing}
V.~G. Vizing.
\newblock Coloring the vertices of a graph in prescribed colors.
\newblock {\em Diskret. Analiz}, (29):3--10, 101, 1976.

\bibitem{arXiv_YaYa2020}
D.~Yang and F.~Yang.
\newblock Flexibility of planar graphs without {$C_4$} and {$C_5$}, 2020.
\newblock \href {https://arxiv.org/abs/2006.05243} {\path{arXiv:2006.05243}}.

\bibitem{2022Yang}
F.~Yang.
\newblock On sufficient conditions for planar graphs to be 5-flexible.
\newblock {\em Graphs Combin.}, 38(3):Paper No. 70, 15, 2022.
\newblock \href {https://doi.org/10.1007/s00373-022-02480-4} {\path{doi:10.1007/s00373-022-02480-4}}.

\end{thebibliography}

\end{document}